\documentclass{article}
\usepackage{amsmath,color}
\usepackage{amsfonts,amsthm}
\usepackage{amssymb,enumerate,enumitem,verbatim}
\usepackage{graphicx}
\usepackage{amsmath,setspace,scalefnt}
\usepackage[usenames,dvipsnames,svgnames,table]{xcolor}
\usepackage{amsfonts,amsthm}
\usepackage{amssymb,enumitem,verbatim}
\usepackage{graphicx}
\usepackage{tikz,caption,subcaption}
\usepackage{morefloats}
\usetikzlibrary{shapes}

\newcounter{capitalcounter}
\newcounter{capitalcounterbackup}

\newcounter{STEPJcounter}

\newcounter{claimcounter}

\newcounter{myfootnote}[page]

\newtheorem{lemma}{Lemma}[section]
\newtheorem{corollary}[lemma]{Corollary}
\newtheorem{theorem}[lemma]{Theorem}

\newtheorem{prop}[lemma]{Proposition}

\theoremstyle{definition}
\newtheorem{defn}[lemma]{Definition}

\newtheorem{claimalpha}[claimcounter]{Claim}

\theoremstyle{remark}
\newtheorem*{rmk}{Remark}

\newtheorem*{example}{Example}
\newtheorem*{nte}{Note}

\global\long\def\a{\alpha}

\global\long\def\d{\delta}
\global\long\def\l{\lambda}

\global\long\def\e{\varepsilon}

\global\long\def\N{\mathbb{N}}
\global\long\def\R{\mathbb{R}}

\global\long\def\P{\mathbb{P}}

\global\long\def\GG{G}
\global\long\def\HH{\mathcal{H}}

\global\long\def\T{\mathcal{T}}

\global\long\def\BB{\mathcal{B}}

\global\long\def\E{\mathbb{E}}

\global\long\def\TT{\mathcal{T}}

\global\long\def\PP{\mathcal{P}}
\global\long\def\UU{\mathcal{U}}

\global\long\def\re{\begin{rmk}}
\global\long\def\mark{\end{rmk}}
\global\long\def\ex{\begin{example}}
\global\long\def\ple{\end{example}}
\global\long\def\no{\begin{nte}}
\global\long\def\ted{\end{nte}}
\global\long\def\en{\begin{compactenum}}
\global\long\def\um{\end{compactenum}}
\global\long\def\li{\begin{compactitem}}
\global\long\def\st{\end{compactitem}}
\global\long\def\de{\begin{defn}}
\global\long\def\fn{\end{defn}}
\global\long\def\cor{\begin{corollary}}
\global\long\def\ary{\end{corollary}}
\global\long\def\lem{\begin{lemma}}
\global\long\def\ma{\end{lemma}}
\global\long\def\arr{\begin{array}}
\global\long\def\ay{\end{array}}
\global\long\def\pr{\begin{proof}}
\global\long\def\oof{\end{proof}}

\newif\ifdraft
\draftfalse 

\newif\ifspeed
\speedtrue 

\newcommand{\newcolour}{black}

\newif\ifdraft
\newif\ifpdraft
\drafttrue 
\draftfalse
\pdrafttrue
\newif\ifdone
\newif\ifdiss
\dissfalse
\donefalse

\title{Spanning trees in random graphs}

\author{Richard Montgomery\footnote{University of Birmingham, Birmingham, B15 2TT, UK; r.h.montgomery@bham.ac.uk}}
\date{}

\begin{document}
\maketitle

\begin{abstract}
For each $\Delta>0$, we prove that there exists some $C=C(\Delta)$ for which the binomial random graph $G(n,C\log n/n)$  almost surely contains a copy of every tree with $n$ vertices and maximum degree at most $\Delta$. In doing so, we confirm a conjecture by Kahn.
\end{abstract}

\section{Introduction}\label{1intro}

Since its inception by Erd\H{o}s and R\'enyi~\cite{ER59} in 1959, a major focus of the study of random graphs has concerned when any particular subgraph is likely to appear in the binomial random graph $G(n,p)$. We denote by $G(n,p)$ the random graph with $n$ vertices where each possible edge is included independently at random with probability $p$. In their early work, Erd\H{o}s and R\'enyi~\cite{ER59} asked, essentially, how large need $p=p(n)$ be before $G(n,p)$ asymptotically almost surely contains a Hamilton cycle? That is, here, a cycle with $n$ vertices. In 1976, P\'osa~\cite{posa76} and Korshunov~\cite{kor76} independently proved that we may take some $p=O(\log n/n)$ such that $G(n,p)$ almost surely contains a Hamilton cycle.
Refinements in the allowed probability~$p$ by Koml\'os and Szemer\'edi~\cite{KS83}, and Bollob\'as~\cite{bollo83}, followed, before Bollob\'as~\cite{bollo84} showed that if edges are added randomly one-by-one to the empty graph with $n$ vertices (forming the \emph{random graph process}), then the very edge whose addition increases the minimum degree to 2 almost surely creates a Hamilton cycle. Clearly, any previous graph in this process contains no Hamilton cycle, and thus here the containment of a Hamilton cycle is almost surely concurrent with the absence of a vertex with degree~0 or 1.

Beyond the Hamilton cycle, when are more general subgraphs $H$ with $n$ vertices likely to appear in the random graph $G(n,p)$?
Kahn and Kalai~\cite{KK07} conjectured in 2007 that, if $p_E$ is such that the expected number of copies of any subgraph of $H$ in $G(n,p_E)$ is at least $1$, then we may take some $p=O(p_E\log n)$ such that a copy of~$H$ almost surely appears in $G(n,p)$.
Riordan~\cite{Riordan00} had already given a second moment method capable of determining when many such subgraphs $H$ are likely to appear, typically those subgraphs which are not locally more dense than the whole subgraph. In a remarkable paper in 2008, Johansson, Kahn and Vu~\cite{JKV08} determined when $G(n,p)$ is likely to contain an $H$-factor, that is, $n/|H|$ vertex-disjoint copies of a fixed (strictly balanced) graph $H$. Recently, there has been much interest in determining when any particular graph with $n$ vertices and maximum degree at most $\Delta$ is likely to appear in $G(n,p)$ (see work by Alon and F\"uredi~\cite{AF92}, and others~\cite{CFNS16,DKRR08,FLN16}, as well as the recent comprehensive survey by B\"ottcher~\cite{Bot17}).
For the likely appearance of all these subgraphs, however, the edge probability must usually be significantly higher than that required to almost surely guarantee a Hamilton cycle.

Which other subgraphs can we expect to appear at around the same edge probability as the Hamilton cycle?
The random graph $G(n,O(\log n/n))$ almost surely has relatively few short cycles and maximum degree $O(\log n)$. Kahn~\cite{KLW14b} made the natural conjecture that, for each $\Delta>0$, there should be some $C>0$ such that, given any tree with $n$ vertices and maximum degree at most $\Delta$, $G(n,C\log n/n)$ almost surely contains a copy of that tree. Let us be precise with what we mean by `almost surely' here. More formally, for each $\Delta>0$, a constant $C=C(\Delta)$ was conjectured to exist such that, given any sequence of trees $\{T_n\}_{n\geq 1}$, where~$T_n$ has $n$ vertices and maximum degree at most $\Delta$, we have $\P(T_n\subset \GG(n,C\log n/ n))\to 1$ as $n\to \infty$.
Erd\H{o}s and R\'enyi~\cite{ER59} showed that we require $p=(\log n+\omega(1))/n$ if $G(n,p)$ is to be almost surely connected, or, equivalently, contain \emph{some} tree with $n$ vertices. Thus, this conjecture would be tight up to the constant $C=C(\Delta)$.

Let $\TT(n,\Delta)$ be the class of trees with $n$ vertices and maximum degree at most $\Delta$. We say that a subgraph of a graph $G$ is \emph{spanning} if it includes every vertex of $G$. Thus, we are concerned with when particular spanning trees are likely to appear in the random graph $G(n,p)$. The first progress towards a good understanding of this was made by Alon, Krivelevich and Sudakov~\cite{AKS07}, who studied the appearance of \emph{almost}-spanning trees. They showed that, for each $\e,\Delta>0$, and for sufficiently large $c=c(\e,\Delta)$, the random graph $\GG(n,c/n)$ almost surely contains a copy of every tree in $\TT((1-\e)n,\Delta)$.
This already gave a good answer to when almost-spanning trees can be expected to appear in the random graph, and the constant $c(\e,\Delta)$ was further improved by Balogh, Csaba, Pei and Samotij~\cite{BCPS10}, using a tree embedding theorem by Haxell~\cite{PH01}.
Note that almost-spanning trees may typically appear at lower edge probabilities than spanning trees, as their containment does not require connectivity. Note also that this result by Alon, Krivelevich and Sudakov shows that such a random graph almost surely contains a copy of every tree in $\TT((1-\e)n,\Delta)$ \emph{simultaneously}. We say that such a graph is $\TT((1-\e)n,\Delta)$-\emph{universal}.

Krivelevich~\cite{MK10} made the first substantial progress towards finding any particular spanning tree in the random graph, showing that, for any $\e>0$ and $T\in \T(n,\Delta)$, $\GG(n,\max\{(40\Delta/\e)\log n,n^\e\}/n)$ almost surely contains a copy of $T$.
This result is essentially tight if $\Delta=n^{\Omega(1)}$, but when $\Delta$ is small this is far from the expected bound $p=\Omega_\Delta(\log n/n)$. Seeking all such trees simultaneously,
Johannsen, Krivelevich and Samotij~\cite{JKS12} showed that there is some constant $C>0$ for which, if $\Delta\geq \log n$ and $p\geq C\Delta n^{-1/3}\log n$, then $\GG(n,p)$ is almost surely $\TT(n,\Delta)$-universal. Ferber, Nenadov and Peter~\cite{FNP13} showed that, for any $\Delta>0$, if $p=\omega(\Delta^{12} n^{-1/2}\log^3 n)$, then $\GG(n,p)$ is almost surely $\TT(n,\Delta)$-universal. For small $\Delta$, this improved the probability for which $G(n,p)$ is known to be almost surely $\TT(n,\Delta)$-universal, but it still left a large gap to the anticipated requirement $p=\Omega_\Delta(\log n/n)$.
Prior to this paper, these three results constituted the only general results on bounded degree spanning trees in random graphs, but further progress had been made for certain subclasses of $\TT(n,\Delta)$, as follows.

Alon, Krivelevich and Sudakov~\cite{AKS07} remarked that their almost-spanning tree result can be used to show that, if a tree $T\in \TT(n,\Delta)$ has linearly many leaves (at least $\e n$, say), then a copy of $T$ almost surely appears in $G(n,C\log n/n)$, for some $C=C(\e)$.
Hefetz, Krivelevich and Szab\'o~\cite{HKS12} improved this, showing that we may take $C=1+o(1)$ when $T\in \TT(n,\Delta)$ has linearly many leaves, or, alternatively, a linear length \emph{bare path}. Here, a bare path within a tree is one with no vertices branching off the interior vertices of the path into the tree.
Recent work by Glebov, Johannsen and Krivelevich~\cite{RomanPhD,GJK14} has proved an even sharper result for both these subclasses of trees, determining that in the random graph process such trees will appear at least as early as the Hamilton path. Importantly, their work finds all such trees simultaneously in the random graph, showing that $\GG(n,(1+o(1))\log n/ n)$ is almost surely universal for the class of trees with linearly many leaves or a linear length bare path. Using a delicate probabilistic argument, Kahn, Lubetzky and Wormald~\cite{KLW14a,KLW14b} were able to additionally determine that, for large $C>0$, $G(n,C\log n/n)$ is likely to contain spanning trees known as \emph{combs}. Roughly speaking, these combs have many medium length bare paths which each end in one leaf, and thus they interpolate in some sense between trees with linearly many leaves and trees with a long bare path.

In this paper, we show that a copy of every tree in $\TT(n,\Delta)$ is likely to appear in $G(n,\Theta_\Delta(\log n/n))$. This confirms the bound on the probability conjectured by Kahn, while additionally finding all such trees simultaneously.

\begin{theorem}\label{unithres}
For every $\Delta>0$, there is a constant $C$ such that the random graph $\GG(n,C\log n/ n)$ almost surely contains a copy of every tree with $n$ vertices and maximum degree at most $\Delta$.
\end{theorem}

It is very likely that the constant $C$ in Theorem~\ref{unithres} need not depend on $\Delta$; indeed, some $C=1+o(1)$ likely suffices. Further still, Glebov, Johannsen and Krivelevich~\cite{RomanPhD,GJK14} have conjectured that, for each $\Delta>0$, in almost every graph process the first graph containing a Hamilton path will contain a copy of every tree in $\TT(n,\Delta)$. Such improvements appear, however, beyond the current reach of the methods used here.

Krivelevich~\cite{MK10} observed that any tree with few leaves has, in compensation, many long bare paths. In particular, for each $k$, a tree with $n$ vertices must have  at least $n/4k$ leaves or at least $n/4k$ disjoint bare paths with length $k$. Spanning trees with many leaves may be embedded simultaneously and precisely using the work by Glebov, Johannsen and Krivelevich~\cite{RomanPhD,GJK14} mentioned above.
This reduces proving Theorem~\ref{unithres} to embedding spanning trees with many long bare paths.
We will embed these trees using a version of \emph{absorption}, which was introduced as a general method by R\"odl, Ruci\'nski and Szemer\'edi~\cite{RRS08}. In their terminology, we use a natural construction for \emph{absorbers}, similar in spirit to that used elsewhere by K\"uhn and Osthus~\cite{KO12} and Allen, B\"{o}ttcher, Kohayakawa and Person~\cite{ABKP13}. Our methods differ from previous implementations in using a system of paths, rather than a single long path, to create the \emph{reservoir} from the absorbers, while achieving this in a comparatively sparse graph.
Using absorption in this way, and relying on the work by Glebov, Johannsen and Krivelevich~\cite{GJK14}, it could be shown that $\GG(n,\log^2 n/n)$ is almost surely $\TT(n,\Delta)$-universal. However, we require several more ideas to reduce the edge probability to $C\log n/n$, for some large constant $C=C(\Delta)$.
\color{\newcolour}~Some of these ideas which may find use elsewhere include methods to, almost surely in a binomial random graph, embed almost-spanning bounded degree trees so that the uncovered vertices lie in a prespecified vertex set (see Section~\ref{4almostspanning}), give a sufficient expansion condition for Hamiltonian subsets (see Section~\ref{3hamcycles}) and find a very sparse yet well connected reservoir set (see Section~\ref{6connectionlemmas}). Additionally, we introduce connectors -- paths with additional structure to boost the chance they can be connected into other structures in random graphs (see Section~\ref{connectors}).\color{black}

As remarked above, we do not expect the best bound on the constant $C(\Delta)$ that can be shown using our current methods to be optimal. Consequently, we focus on the presentation of our methods over any potential reduction in the constant.
In the rest of this section, we will cover our notation. In Section~\ref{1boutline}, we will give an overview of the proof of Theorem~\ref{unithres} and an outline of the rest of the paper.

\subsection{Notation}\label{intronot}

A graph $G$ has vertex set $V(G)$, edge set $E(G)$ and order $|G|=|V(G)|$. For a subset $U\subset V(G)$, $G[U]$ is the subgraph of $G$ induced on the set $U$. The set of neighbours of a vertex $v$ is denoted by $N(v)$, and we use $d(v)=|N(v)|$ for the degree of a vertex. The maximum degree of a graph $G$ is denoted by $\Delta(G)$, and the minimum degree by $\delta(G)$. Where $U\subset V(G)$, we write $N(v,U)=N(v)\cap U$ for the set of neighbours of $v$ in $U$, and take $d(v,U)=|N(v,U)|$. Where $U\subset V(G)$, we set $N(U)=(\cup_{v\in U}N(v))\setminus U$ and $N'(U)=\cup_{v\in U}N(v)$. We mostly use $N(U)$, the \emph{exterior neighbourhood} of a set $U$, and remark when $N'(U)$ is used, to avoid confusion.
 We use further notation that naturally extends the notation above, such as $N(U,W)=N(U)\cap W$, without further description. Where multiple graphs are used, we indicate the graph considered in the subscript, for example using $d_G(v)$.

In a graph $G$, we say a vertex set $X\subset V(G)$ is \emph{independent} if there are no edges between vertices of $X$ in $G$. 
For each $n\in \N$, we write $[n]$ for $\{1,\ldots,n\}$, and, when $0\leq p\leq 1$, we let $\GG(n,p)$ be the random graph with vertex set $[n]$ where each possible edge is selected independently at random with probability~$p$. 
For any two subsets $A,B\subset V(G)$, we let $e_G(A,B)=\sum_{x\in A}d_G(x,B)$, noting that edges are counted twice if both their endvertices are in $A \cap B$. We say a path with $l$ vertices has \emph{length} $l-1$. An \emph{$x,y$-path} has endvertices $x$ and $y$. When we remove a path from a graph we will remove all the edges of the path and delete any resulting isolated vertices.

Given two subgraphs $H_1,H_2\subset G$, we say that $H_1\cup H_2$ is the subgraph of $G$ with vertex set $V(H_1)\cup V(H_2)$ and edge set $E(H_1)\cup E(H_2)$. Where $P\subset G$ is a path and $H\subset G$ is disjoint from~$P$ except for on one endvertex of $P$, we think of the graph $H\cup P$ as $H$ with the path $P$ added and set $H+P=H\cup P$. When $e=xy$ is an edge of a graph $G$ and $H\subset G$, we use $H+e$ for the graph with vertex set $V(H)\cup\{x,y\}$ and edge set $E(H)\cup\{xy\}$. Similarly, we use $H-e$ for the graph with vertex set $V(H)$ and edge set $E(H)\setminus\{xy\}$. 
For any set $W\subset V(G)$, we abbreviate $G[V(G)\setminus W]$ by $G-W$.

We say $f(n)=o(g(n))$ or $g(n)=\omega(f(n))$ if $f(n)/g(n)\to 0$ as $n\to\infty$. In particular, $\omega(1)$ denotes a function which tends to infinity with $n$. We say $f(n)=O(g(n))$ or $g(n)=\Omega(f(n))$ if there exists some constant $C$ for which $f(n)\leq Cg(n)$ holds for all $n$. If $f(n)=O(g(n))$ and $f(n)=\Omega(g(n))$, then we say that $f(n)=\Theta(g(n))$. Throughout this paper, we consider $\Delta$ to be an arbitrary constant while taking $n$ to be large. We sometimes use $\Delta$ in the subscript of our asymptotic notation to emphasise that the implicit constant depends on $\Delta$, e.g., using $\Theta_\Delta(g(n))$.

We use $\log$ for the natural logarithm, and often omit rounding signs when they are not crucial.




\section{Outline of the proof of Theorem \ref{unithres}}\label{1boutline}

We will find certain properties and structures within a typical random graph $G=\GG(n,C\log n/ n)$, with $C$ a large constant, before using these properties and structures to embed any tree $T\in \TT(n,\Delta)$. We will take a tree $T$, find within it some leaves or bare paths, remove these from $T$ to get a forest $T'$, embed $T'$ using almost-spanning tree embedding techniques, and, to finish, embed the removed leaves or bare paths from~$T$. Often, we will need to embed~$T'$ carefully to assist the later completion of the embedding. For this, we will use \emph{extendability methods} introduced by Glebov, Johannsen and Krivelevich~\cite{RomanPhD,GJK14}. These methods are detailed in Section~\ref{prelims} along with many other preliminary results we will need. We will now outline the extendability methods briefly, before dividing the proof of Theorem~\ref{unithres} into cases, and sketching our embedding in each case.

\subsection{Almost-spanning tree embedding methods}\label{firstalmostspan}
The embeddings used in this paper could be applied using a tree embedding theorem of Haxell~\cite{PH01} and some path connection methods like those of Broder, Frieze, Suen, and Upfal~\cite{BFSU96}. The embeddings we use are relatively complex, however, and therefore we are grateful for the recent developments by Glebov, Johannsen and Krivelevich~\cite{RomanPhD,GJK14} which allow us to carry out our embeddings with minimal technicalities. By developing methods used by Haxell~\cite{PH01}, Glebov, Johannsen and Krivelevich~\cite{RomanPhD,GJK14} introduced a refined and versatile method for growing trees by the addition of leaves or long bare paths, a method which can be applied to find trees in graphs which satisfy some simple expansion properties. We will refer to these methods as \emph{extendability methods}, and they are introduced in full in Section~\ref{johannsen}.

In short, we begin with a small subgraph satisfying some \emph{extendability property} within a larger graph~$G$ which itself satisfies some simple further conditions. Two key lemmas (Lemma~\ref{extendablevertex} and Lemma~\ref{extendableconnect}) allow us to add, respectively, leaves and bare paths to the subgraph while maintaining this extendability property. Using this, we can build up a copy of a large subgraph of the tree we are seeking. To apply these extendability methods, we need some spare vertices -- the subgraph we build typically must have $\Omega(n\log\log n/\log n)$ fewer vertices than~$G$. However, the embedding techniques are flexible enough that we can not only find any large subtree of~$T$ in~$G$, but, furthermore, we can root such a subtree of~$T$ at any chosen vertex in the graph. More generally, under some weak constraints, given a subset $Q$ of the vertices of~$T$ and a smaller subset $X$ of the vertices of~$G$, we can find an embedding of almost all of~$T$ in which the image of $Q$ contains $X$ (see Lemma~\ref{embedparentsfinal}). The flexibility in these embedding techniques will allow us to carry out the embeddings described below for each case of Theorem~\ref{unithres}.

\subsection{Division into cases}\label{casedivision}
Our embedding of a tree $T$ varies depending on the structure of $T$. To aid our embedding, we seek simple structures in $T$, for which we use the following key observation by Krivelevich~\cite{MK10} (as quoted in Section~\ref{sec33}).

\begin{lemma}\label{split} For any integers $n,k>2$, a tree with~$n$ vertices either has at least $n/4k$ leaves or a collection of at least $n/4k$ vertex disjoint bare paths, each with length $k$.
\end{lemma}

Essentially, we are concerned only with the leaves and bare paths in a tree. The other, more complex, structure remaining after their (selective) deletion will be embedded easily using the extendability methods outlined in Section~\ref{johannsen}. Our challenge is to embed the deleted leaves and bare paths, which must necessarily cover exactly the remaining vertices in the random graph.

We divide our embeddings into four cases. These are defined precisely below, but we will first introduce them informally to explain how they arise.
By modifying the techniques introduced by Glebov, Johannsen and Krivelevich~\cite{RomanPhD,GJK14}, we will be able to embed trees in $\TT(n,\Delta)$ with $\Omega(n\log\log n/\log n)$ leaves in $G(n,\Theta(\log n/n))$. These trees comprise our first case, Case A. It follows from Lemma~\ref{split} that the remaining trees in $\TT(n,\Delta)$ have a significant number of bare paths with length $\Omega(\log n/\log\log n)$. This is fortunate, as in $G(n,\Omega(\log n/n))$ we may expect to connect any two arbritrary vertices with paths with length $O(\log n/\log\log n)$. This will allow us to manipulate bare paths with length $\Omega(\log n/\log\log n)$ as we embed them.

Essentially, within a tree $T$, we ask how long can we find disjoint bare paths which cover at least $|T|/90$ vertices in $T$. The longer the bare paths, the more they aid us. If they have length between $\Theta(\log n/\log\log n)$ and $\Theta(\log^2 n/\log\log n)$ then we continue to use the leaves of~$T$ as well as these bare paths in our embedding. The trees with many bare paths of this medium length split naturally into two classes, giving Case B and Case C. Loosely speaking, trees in the latter case have many bare paths of length $\Omega(\log n)$. We cannot guarantee such trees have more than $O(n/\log n)$ leaves, and they may have many fewer. This leads us to consider sets in the random graph with size $O(n/\log n)$. Such sets in $G(n,\Theta(\log n/n))$ typically have average degree below 1, and this causes additional challenges which merit consideration in a separate case, that is, in Case C. Our final case, Case D, will cover trees with many bare paths with length $\Omega(\log^2 n/\log\log n)$. Such paths are sufficiently long that we may use them to finish the embedding of the tree without using its leaves in any special manner.

Our embeddings are easier if the leaves we find are pairwise a little distance from each other. To describe this, we use the following definition.
\de\label{sepdefn}
Given a graph~$G$, and a subset $Q\subset V(G)$, we say the set~$Q$ is \emph{$k$-separated in~$G$} if each pair of vertices from~$Q$ are a distance at least $k$ apart in~$G$.
\fn
In order to make our cases precise, we use this definition and the following parameters $\l(T)$ and $k(T)$, where $k(T)$ is defined using the further parameter $k'(T)$.




Given a tree~$T$, let $\l(T)$ be the size of a largest 20-separated set of leaves in~$T$ and let
\[
k'(T)=\max\{k:T\text{ contains at least }n/90k\text{ vertex disjoint bare paths with length }k\}.
\]
Our embedding is the same for all trees~$T$ with $k'(T)\geq 10^{-6}\log^2 n/\log\log n$, so, with this in mind, let
\[
k(T)=\min\{k'(T),10^{-6}\log^2n/\log\log n\}.
\]

We consider the following four subclasses of $\TT(n,\Delta)$.
\begin{itemize}
\item $\TT_A(n,\Delta)=\{T\in \TT(n,\Delta):\l(T)\geq n\log\log n/10^5\log n\}$
\item $\TT_B(n,\Delta)=\{T\in \TT(n,\Delta):10^3\log n/\log\log n \leq k(T)< 10^2\log n\}$
\item $\TT_C(n,\Delta)=\{T\in \TT(n,\Delta):10^2\log n \leq k(T)< 10^{-6}\log^2 n/\log\log n\}$
\item $\TT_D(n,\Delta)=\{T\in \TT(n,\Delta):k(T)= 10^{-6}\log^2 n/\log\log n\}$
\end{itemize}

In Section~\ref{sec33} we prove the following lemma, which confirms that, for sufficiently large $n$, these subclasses cover $\TT(n,\Delta)$.

\lem\label{splitclass}
Let $\Delta>0$. For sufficiently large $n$,
\[
\TT(n,\Delta)=\TT_A(n,\Delta)\cup\TT_B(n,\Delta)\cup \TT_C(n,\Delta)\cup \TT_D(n,\Delta).
\]
\ma

By this lemma, then, we can prove Theorem~\ref{unithres} separately for each subclass of $\TT(n,\Delta)$, as the sufficiently large condition for $n$ is satisfied naturally in the limit we consider. We refer to proving Theorem~\ref{unithres} for these subclasses as Case A to D of the theorem, respectively. We will now discuss our approach in each case. As described later, our actual implementation will vary slightly from these simplified sketches, but they cover our basic methods.

\subsection{Our embeddings in Cases A and B}
We will embed the trees in Case~A using a method developed by Glebov, Johannsen and Krivelevich~\cite{RomanPhD,GJK14}. We find in~$G$ two disjoint small sets $X_1$ and $X_2$ so that we can match any set $U_1\subset V(G)$ into any set $U_2\subset V(G)$ if $X_1\subset U_1$, $X_2\subset U_2$, $U_1\cap U_2=\emptyset$ and $|U_1|=|U_2|$. We will call such sets $X_1$ and $X_2$ \emph{matchmaker sets}. Next, we remove a suitably large and separated set of leaves from~$T$, calling the remaining subtree~$T'$. We let $Q$ be the set of the neighbours of the removed leaves in~$T$, so that~$T$ is formed from~$T'$ by adding a matching to $Q$. We embed~$T'$ in $G-X_2$ so that the image of $Q$, which we call $U_1$, contains $X_1$. To do this, we use a result of Glebov, Johannsen and Krivelevich~\cite{RomanPhD,GJK14}, which is given a new, more efficient, proof in Section~\ref{5embedparents}. We let $U_2$ be the vertices in~$G$ not covered by the image of~$T'$, so that $X_2\subset U_2$. By the properties of the matchmaker sets~$X_1$ and~$X_2$, we can then find a matching between $U_1$ and $U_2$, and use this to embed the leaves that we removed from~$T$, completing the embedding. This embedding is carried out in Section~\ref{7caseA}, where we prove Case~A of Theorem~\ref{unithres}.

As we need to cover $X_1$ by the image of $Q$, we wish the matchmaker sets $X_1$ and $X_2$ to be as small as possible. The minimum possible size of such sets~$X_1$ and~$X_2$ is closely related to the parameter $m$, which appears in some form throughout each of our embeddings. This parameter is the smallest integer such that there is an edge in~$G$ between any two disjoint sets with size $m$. As defined in Section~\ref{johannsen}, we say this is the smallest integer $m$ such that $G$ is $m$-joined. A good bound on the value of $m$ is given later in Proposition~\ref{generalprops}. In the embeddings in Case~A and Case~B, this gives
\begin{equation}\label{mrough}
m=\Theta(\log (np)/p)=\Theta(n\log\log n/\log n).
\end{equation}
In order to find a matching between $U_1$ and $U_2$, we estimate the size of the neighbourhood in $U_2$ of subsets in $U_1$, and thus we may show that Hall's matching condition holds. Subsets in $U_1$ with size between $m$ and $|U_1|-m$ can be shown to have a sufficiently large neighbourhood as the graph is $m$-joined; the matchmaker sets need only handle the remaining subsets of $U_1$ (more details are given in Lemma~\ref{matchings}). This means that the size of the sets $X_1$ and $X_2$ cannot be smaller than~$m$, but we are able to construct such sets with size comparable to $m$ (as shown in Proposition~\ref{goodset}).

The embedding described for Case~A does not work if there are too few leaves, as then the set $Q$ is too small for its image to cover the $\Theta(m)$ vertices in $X_1$. In Case~B, therefore, we first select a subset~$R$ of $V(G)$, so that~$R$ is large but contains $o(n)$ vertices. Let $m'$ be the smallest integer such that there is an edge in~$G$ between any two disjoint subsets of~$R$ with size $m'$. As we have restricted the number of vertices from which we choose these disjoint subsets, we are able to take a stronger bound on $m'$ than~$m$. This can best be seen by considering the far reaches of Case~B, where we will use $|R|=\Theta(n/\log n)$ and (from Proposition~\ref{generalprops}) get $m'=\Theta(n/\log n)\ll m$.

In general, we can then find subsets $X_1$ and $X_2$ with size $\Theta(m')$ so that we can match a set $U_1\supset X_1$ into a set $U_2\supset X_2$ whenever $U_1\cap U_2=\emptyset$, $|U_1|=|U_2|$, and $U_1\setminus X_1$ and $U_2\setminus X_2$ are subsets of~$R$. We remove some leaves from our tree~$T$ in Case~B, and embed the resulting subtree~$T'$ so that the image of the set $Q$, which is the set of neighbours of the removed leaves, contains the set $X_1$. Since the sizes of $X_1$ and $X_2$ are related to $m'$, not $m$, they are now both smaller than the set $Q$. We can then complete the embedding of~$T$ as in we did in Case~A, by finding a matching from the copy of $Q$ to the set of the unused vertices, provided we have ensured that the embedding of~$T'$ covers all of the vertices in $V(G)\setminus (R\cup X_2)$.

That we are able to embed~$T'$ to cover $V(G)\setminus (R\cup X_2)$ is shown by Lemma~\ref{almostspan}, which will form a standard part of the remaining embeddings. The covering is achieved using the bare paths in~$T'$. Roughly speaking, part of $T'$ is embedded in $G-X_2$, before a cycle covering the unused vertices in $V(G)\setminus (R\cup X_2)$ is found (using expansion techniques and P\'osa rotation). This cycle is then split into subpaths, which are connected up as bare paths in the embedding of~$T'$, using the extendability methods, and in particular Lemma~\ref{extendableconnect}. To allow room to build up the first part of~$T'$, and to connect up these paths, we use some of the vertices in~$R$. We connect the subpaths of the cycle into the embedding by first matching their endvertices into~$R$, and, similarly, matching into~$R$ the vertices from the partial embedding of~$T'$ to which these paths need connected. We then link the ends of these matchings appropriately using paths inside~$R$. This completes the embedding of~$T$ apart from the leaves attached to $Q$. There will then be the right number of vertices left unused in~$R$ to add to $X_2$ and we can find a matching, as described above, to complete the embedding.
This embedding of~$T'$ is given in Section~\ref{4almostspanning}, using work in Section~\ref{3hamcycles} which finds cycles covering subsets in a random graph. We prove Case~B of Theorem~\ref{unithres} in Section~\ref{8caseB}.

As we consider trees~$T$ in Case~B with an increasing value of $k(T)$, the size of the set $Q$ decreases. Therefore, to carry out the embedding, we need to decrease the size of $X_1$ and $X_2$, and so we must decrease the size of~$R$. It is therefore increasingly hard to cover all the vertices outside of $R\cup X_2$ by the embedding of~$T'$. Fortunately, Lemma~\ref{almostspan} is able to take advantage of the increasing length $k(T)$ of the bare paths in~$T'$. Eventually, the size of~$R$ will decrease until it contains almost exclusively (internally) isolated vertices. Due to this, we cannot find the paths required within~$R$ to connect the subpaths from the cycle into the partial embedding of~$T'$, and so the method outlined above breaks down. This is our first problem, but two other problems also arise at around the same place, that is, when $k=\Theta(\log n)$. The second problem is that our method for creating matchmaker sets $X_1$ and $X_2$ which are sufficiently small that they can be covered by the image of $Q$ also fails. Finally, our third problem is that, as the size of the set~$R$ decreases, we eventually cannot match the endvertices of the subpaths of the cycle into the set~$R$.

\subsection{Our embeddings in Cases C and D}
As mentioned above, we need to surmount three problems if we are to modify the embedding for Case~B to use for Case~C.
After solving these problems, a small conceptual change is sufficient to embed the trees in Case~D.

At the start of our embedding in Case~C, we will carefully choose a set~$R$ in our random graph that is small but internally well connected (as detailed in Section~\ref{6connectionlemmas}). This will address our first problem by enabling us to connect many different vertex pairs using paths in~$R$. We then build what we call a \emph{$\lambda$-device}, which is in effect a construction of matchmaker sets $X_1$ and $X_2$, so that, given any appropriately sized subset~$U$ of~$R$, we can find a matching between the sets~$X_1$ and $X_2\cup U$. The construction of the $\l$-device (in Section~\ref{9absorbC}) is more involved than the construction of the matchmakers sets for Case~B, but does not contain much more technical difficulty. The use of a $\l$-device overcomes the second problem mentioned above, as the sets $X_1$, $X_2$ and~$R$ are much smaller than they previously were. We then proceed with an embedding similar to the embedding in Case~B. We take a tree in Case~C and remove some of its leaves to get a tree~$T'$, before embedding~$T'$ into the graph~$G-X_2$ so that the image of the set $Q$ of the neighbours of the leaves removed from~$T$ contains the set $X_1$. We make sure when embedding the tree~$T'$ that we cover all the vertices not in $R\cup X_2$, before using the properties of the $\l$-device to find an appropriate matching to attach the deleted leaves and finish the embedding of~$T$.

To embed~$T'$ and cover all the vertices not in $R\cup X_2$, we use the same method as we did in Case~B. That is, we embed most of~$T'$ and cover the remaining vertices not in $R\cup X_2$ by a cycle, which we break into subpaths. Here we encounter the third, and final, problem, where we cannot find a matching from the endvertices of these subpaths into~$R$.
This difficulty is overcome by the use of \emph{$(l,\gamma)$-connectors} (constructed in Section~\ref{connectors}). These are small subgraphs which have spanning paths between many different pairs of endvertices. Instead of matching the endvertices of the subpaths from the cycle into~$R$ directly, we connect each of them into~$R$ via an $(l,\gamma)$-connector in an appropriate way, before finding paths within the set~$R$ to attach the subpaths and the $(l,\gamma)$-connectors into the embedding of~$T'$.

All this allows us to progress further, and embed the trees in Case~C. The point where we switch to Case~D is somewhat arbitrary, but a switch is necessary somewhere (which can be delayed by increasing our constant $C(\Delta)$). For Case~D, the only change we make to our embedding is to use a different $\l$-device $(X_1,X_2)$. Given an appropriately sized subset $U$ of~$R$, this new $\l$-device joins specified pairs of vertices in~$X_1$ by disjoint paths with a specified length, whose internal vertices together cover exactly the vertices in $X_2\cup U$. This exact covering will allow us to cover the final vertices in the graph~$G$ and complete the embedding. The tree~$T'$ in this case is formed by removing bare paths, rather than leaves, from~$T$ and these bare paths are reinserted using the $\l$-device and the remaining vertices from~$R$.

The construction of the $\l$-device for Case~D shares some ideas with the construction of the $\l$-device for Case~C, and is also given in Section~\ref{9absorbC}. The $(l,\gamma)$-connectors we use are defined and constructed at the start of Section~\ref{11caseCD}, before we put our methods together to prove Cases~C and~D of Theorem~\ref{unithres}.

In Cases~C and~D, our embeddings are using a new application of the absorption method, which was introduced as a general principle by R\"odl, Ruci\'nski and Szemer\'edi~\cite{RRS08}. In this method, a subgraph to be found is partially embedded while a \emph{reservoir} is developed. The difficulties in embedding what remains of the subgraph into the decreasing space in the parent graph are mitigated by allowing the use of some vertices from the reservoir. The remaining vertices are then \emph{absorbed} into the early parts of the embedding by using the special properties of the reservoir. In our methods outlined above, the set $R$ functions as the reservoir, while the remaining vertices at the end of the embedding are absorbed into the $\lambda$-device.

\subsection{Further complications}
The outlines above give an overview of the proof of Theorem~\ref{unithres}, while making certain simplifications to avoid the more technical details.
In particular, in Cases~B,~C, and~D, the embedding of the tree~$T'$ to cover all the vertices outside the set $R\cup X_2$ is more intricate than described here, though our description conveys the basic idea. Similarly, complications arise with the use of $(l,\gamma)$-connectors in Cases C and D. We will need many of these connectors to carry out the sketch above, more than are eventually used for the stated purpose. Therefore, we will need to absorb the unused connectors somehow into the embedding of~$T$. We also often split our tree into large pieces, and accomplish different tasks while embedding each piece. This avoids looking for many different properties while embedding any one subtree.

Our embeddings are also complicated by the universality of Theorem~\ref{unithres} --
we must reveal all the edges of our graph and gain the properties we require to embed any spanning tree. We will reveal the edges of our graph in multiple rounds to find these properties.
Our embedding is effectively the same for all the trees in Case~B, but the properties we require from our random graph differ slightly based on the size of the set~$R$, which is in turn determined for each tree~$T$ by the value of $k(T)$. We therefore show that, for each relevant value of $k$, we can, with probability $1-o(n^{-1})$, find in our random graph a copy of all the trees $T\in \TT_B(n,\Delta)$ with $k(T)=k$. Thus, a union bound over these different events will show that our random graph almost surely contains every tree in $\TT_B(n,\Delta)$. We also take a similar approach in Case~C.

\section{Preliminaries}\label{prelims}

As mentioned in Section~\ref{1boutline}, we will use a flexible framework developed by Glebov, Johannsen and Krivelevich~\cite{RomanPhD, GJK14} for embedding a tree into a larger graph; we detail this framework in Section~\ref{johannsen}. We will use several results on leaves and bare paths in trees; these results are developed in  Section~\ref{sec33}. In Section~\ref{secsplittree}, we will show how a tree can be divided into subtrees satisfying certain properties. The probabilistic results we need for $G(n,p)$ are given in Sections~\ref{sec00} and~\ref{shortpathsgnp}.
Expansion properties and matchings will be important, for example when using the extendability methods or embedding leaves.
In Sections~\ref{sec11} to~\ref{secfinalmatch} we prove a mixture of results concerning matchings and expansion properties.

\subsection{Embedding trees and $(d,m)$-extendability}\label{johannsen}
Friedman and Pippenger~\cite{FP87} showed that trees can be found in graphs with certain natural expansion conditions. These expansion conditions, however, require the graph to have many more vertices than the tree. Their method was developed by Haxell~\cite{PH01} to give a general theorem for embedding trees into graphs with certain expansion conditions, which in some cases allows the tree~$T$ to have almost as many vertices as the graph. This theorem was first used by Balogh, Csaba, Pei and Samotij~\cite{BCPS10} to embed almost-spanning bounded degree trees in the random graph.

Glebov, Johannsen and Krivelevich~\cite{RomanPhD,GJK14} recently modified Haxell's method to develop a very flexible approach for embedding bounded degree trees. The key definition is that of a \emph{$(d,m)$-extendable subgraph}. Given a copy of a subtree of a tree with maximum degree~$d$ in a graph~$G$, if the copy of the subtree is $(d,m)$-extendable in $G$ then, subject to certain conditions, we can add a further vertex of $G$ to extend the copy to a larger subtree by attaching a leaf appropriately. Furthermore, we can do this so that the copied subtree remains $(d,m)$-extendable. This allows us to iterate, building up a copy of the full tree vertex-by-vertex (see Corollary~\ref{treebuild}). The following definition refers to a subgraph~$S$ of a graph~$G$, and to the degree $d_S(x)$ of a vertex $x$ in the subgraph~$S$. Furthermore, for this definition we recall that $N'(U)=\cup_{u\in U}N(u)$.

\de\label{extendabledefn}
Let $d,m\in \N$ satisfy $d\geq 3$ and $m\geq 1$, let~$G$ be a graph, and let $S\subset G$ be a subgraph of~$G$. We say that~$S$ is \emph{$(d,m)$-extendable} if~$S$ has maximum degree at most~$d$ and
\begin{equation}\label{extendthing}
|N'(U)\setminus V(S)|\geq(d-1)|U|-\sum_{x\in U\cap V(S)}(d_S(x)-1)
\end{equation}
for all sets $U\subset V(G)$ with $|U|\leq 2m$.
\fn

Away from this definition, we will mostly use the external neighbourhood, $N(U)$, noting that $|N'(U)|\geq |N(U)|$ for any vertex set $U$ of any graph. Typically, we only turn to $N'(U)$ when deducing properties from the $(d,m)$-extendability of a subgraph. Note that to show $S\subset G$, with $\Delta(S)\leq d$, is $(d,m)$-extendable it is sufficient to show that for every subset $U\subset V(G)$, with $|U|\leq 2m$, we have $|N(U)\setminus V(S)|\geq d|U|$, so that (\ref{extendthing}) then holds. Typically, we will use this stronger condition to show that a subgraph is $(d,m)$-extendable, and we record this as follows.

\begin{prop}\label{uttriv}
Let $d,m\in \N$ satisfy $d\geq 3$ and $m\geq 1$. Suppose the graph $G$ contains a subgraph~$S$ which has maximum degree at most $d$. If every set $U\subset V(G)$ with $|U|\leq 2m$ satisfies $|N(U,V(G)\setminus V(S))|\geq d|U|$, then $S$ is $(d,m)$-extendable in~$G$.\hfill\qed
\end{prop}

The following lemma by Glebov, Johannsen and Krivelevich~\cite{RomanPhD,GJK14} gives conditions under which a $(d,m)$-extendable subgraph can be extended by a leaf yet remain $(d,m)$-extendable.

\lem \label{extendablevertex} \cite[Lemma 5.2.6]{RomanPhD} Let $d,m\in \N$ satisfy $d\geq 3$ and $m\geq 1$, let~$G$ be a graph, and let~$S$ be a $(d,m)$-extendable subgraph of~$G$. Suppose every subset $U\subset V(G)$, with $m\leq |U|\leq 2m$, satisfies
\begin{equation}\label{blahd}
|N_G'(U)|\geq |S|+2dm+1.
\end{equation}
Then, for every vertex $s\in V(S)$ with $d_S(s)\leq d-1$, there exists a vertex $y\in N_G(s)\setminus V(S)$ such that the graph $S+sy$ is $(d,m)$-extendable.\hfill\qed
\ma

We will often wish to apply Lemma~\ref{extendablevertex} when we have the following crucial graph property.
\begin{defn}\label{mjoined}
When $m\in \N$, we say a graph~$G$ is \emph{$m$-joined} if there is an edge in~$G$ between any two disjoint vertex sets which each contain at least~$m$ vertices.
\end{defn}
In order to apply Lemma~\ref{extendablevertex} when a graph is~$m$-joined, we will need the following corollary.
\begin{corollary}\label{mextend}
Let $d,m\in \N$ satisfy $d\geq 3$ and $m\geq 1$, and let~$G$ be an~$m$-joined graph with~$n$ vertices. Let~$S$ be a $(d,m)$-extendable subgraph of~$G$ which satisfies $|S|\leq n-2dm-3m-1$. Then, for every vertex $s\in V(S)$ with $d_S(s)\leq d-1$, there exists a vertex $y\in N_G(s)\setminus V(S)$ such that the graph $S+sy$ is $(d,m)$-extendable.
\end{corollary}
\pr If $U\subset V(G)$, then there are no edges between $U$ and $V(G)\setminus (U\cup N(U))$. Therefore, as~$G$ is~$m$-joined, if $|U|\geq m$, then $|V(G)\setminus (U\cup N(U))|<m$. Therefore, if $m\leq |U|\leq 2m$, then
\[
|N(U)|\geq n-|U|-m\geq n-3m\geq |S|+2dm+1.
\]
Thus, \eqref{blahd} holds for each set $U\subset V(G)$ with $m\leq |U|\leq 2m$. The corollary follows by Lemma~\ref{extendablevertex}.
\oof

We will often use an extendable subgraph which has no edges, for which we use the following definition.

\de Given a graph~$G$ and a subset $U\subset V(G)$, the subgraph $I(U)\subset G$ has the vertex set~$U$ and no edges. That is, $I(U)$ is the vertex set $U$ considered to be an empty graph, so that~$U$ is an independent set.
\fn

Corollary~\ref{mextend} can be used repeatedly to build up a copy of a tree by adding leaves. A typical use of this is given in the following corollary.

\begin{corollary}\label{treebuild} Let $d,m\in \N$ satisfy $d\geq 3$ and $m\geq 1$, and let~$T$ be a tree, with maximum degree at most $d/2$, which contains the vertex $t\in V(T)$. Let~$G$ be an~$m$-joined graph and suppose~$R$ is a $(d,m)$-extendable subgraph of~$G$ with maximum degree $d/2$. Let $v\in V(R)$ and suppose $|R|+|T|\leq |G|-2dm-3m$.
Then, there is a copy~$S$ of~$T$ in $G-(V(R)\setminus \{v\})$, in which~$t$ is copied to~$v$, so that $R\cup S$ is $(d,m)$-extendable in~$G$.
\end{corollary}
\pr As is well known, every tree with at least two vertices has at least two leaves. Therefore, we can iteratively remove leaves not equal to~$t$ from~$T$, and produce a sequence of trees $T_n=T\supset T_{n-1}\supset\ldots\supset T_1=I(\{t\})$, where $n=|T|$, so that, for each $i$, $T_i$ is formed from $T_{i-1}$ by adding a leaf.

Recall that $v\in V(R)$. Starting with $S_0:=I(v)$, for each $2\leq i\leq n$, by applying Corollary~\ref{mextend} to the $(d,m)$-extendable subgraph $R\cup S_{i-1}$, we can extend $S_{i-1}$ to give $S_i$, a copy of $T_i$,  in $G-(V(R)\setminus \{v\})$ so that $R\cup S_i$ is $(d,m)$-extendable in~$G$, and~$t$ is copied to $v$. When we are finished, $S=S_n$ satisfies the requirements of the corollary.
\oof


As shown by Glebov, Johannsen and Krivelevich~\cite{RomanPhD,GJK14}, removing a leaf from, or adding an edge between two vertices with degree at most $d-1$ in, a $(d,m)$-extendable subgraph maintains the extendability, as follows.
\lem \cite[Lemma 5.2.7]{RomanPhD} \label{extendreverse} Let $d,m\in\N$ satisfy $d \geq 3$ and $m \geq 1$, let~$G$ be a graph,
and let~$S$ be a subgraph of~$G$. Furthermore, suppose there exist vertices $s \in V (S)$ and
$y \in N_G(s)\setminus V(S)$ so that the graph $S+ys$ is $(d, m)$-extendable.
Then~$S$ is $(d, m)$-extendable.\hfill\qed
\ma
\lem \cite[Lemma 5.2.8]{RomanPhD} \label{edgeinsert} Let $d,m\in\N$ satisfy $d\geq 3$ and $m\geq 1$, let~$G$ be a graph, and let~$S$ be a $(d,m)$-extendable subgraph of~$G$. If $s,t\in V(S)$ with $d_S(s),d_S(t)\leq d-1$ and $st\in E(G)$, then $S+st$ is $(d,m)$-extendable in~$G$.\hfill\qed
\ma

In addition, Glebov, Johannsen and Krivelevich~\cite{RomanPhD,GJK14} proved the following connection lemma, which allows a path to be added to a $(d,m)$-extendable subgraph while keeping the extendability. We record a slightly stronger statement in Lemma~\ref{extendableconnect} than in~\cite{RomanPhD}, but this follows with an almost identical proof. We remark why we require a stronger statement at the end of this section.

\lem \cite[Lemma 5.2.9]{RomanPhD} \label{extendableconnect} Let $d,m\in\N$ satisfy $m\geq 1$ and $d\geq 3$ and let~$G$ be an~$m$-joined graph. Let~$S$ be a $(d,m)$-extendable subgraph of~$G$ with at most $|G|-10dm$ vertices.

Let $k=\lceil\log (2m)/ \log(d-1)\rceil$ and let $j$ satisfy $1\leq j\leq k$. Suppose there are two disjoint vertex sets~$A$ and~$B$ in~$S$, each with size at least $2m/(d-1)^j$, so that all the vertices in~$A$ and $B$ have degree at most $d/2$ in~$S$. Then there exist vertices $a\in A$, $b\in B$, an $a,b$-path $P$ with length $2j+1$ and interior vertices in $V(G)\setminus V(S)$, and an edge $e$ of the path $P$, so that $S+P-e$ is $(d,m)$-extendable in~$G$.\hfill\qed
\ma

In combination, Lemmas~\ref{extendablevertex},~\ref{extendreverse},~\ref{edgeinsert} and~\ref{extendableconnect} provide a very flexible framework for manipulating almost-spanning trees as we embed them into our graph. We will use Lemma~\ref{extendablevertex} through Corollary~\ref{treebuild}, and Lemma~\ref{extendableconnect} through the following two corollaries.

\begin{corollary} \label{extendableconnectcor} Let $d,m,l\in\N$ satisfy $m\geq 1$ and $d\geq 3$, and let~$G$ be an~$m$-joined graph. Let~$S$ be a $(d,m)$-extendable subgraph of~$G$ with at most $|G|-10dm-2l$ vertices.

Suppose there are two disjoint vertex sets~$A$ and~$B$ in~$S$, each with size at least $2m/(d-1)^{l}$, so that all the vertices in~$A$ and $B$ have degree at most $d/2$ in~$S$. Then there exist vertices $a\in A$, $b\in B$, an $a,b$-path $P$ with length $2l+1$ and interior vertices in $V(G)\setminus V(S)$, and an edge $e$ of the path $P$, so that $S+P-e$ is $(d,m)$-extendable.
\end{corollary}
\pr Let $k=\lceil\log (2m)/ \log(d-1)\rceil$. If $l\leq k$ then the result of the corollary follows directly from Lemma~\ref{extendableconnect}. Let us assume then that $l>k$, so that the condition $|A|,|B|\geq 2m/(d-1)^{l}$ is equivalent to $|A|,|B|>0$, and pick two vertices $a\in A$ and $b\in B$.

By Corollary~\ref{treebuild} we can find a path~$Q$ with length $2l-2k\geq 2$ which has $a$ as an endvertex, and whose other vertices lie outside~$S$, so that $S+Q$ is $(d,m)$-extendable. Let the endvertex of~$Q$ which is not $a$ be $a_0$.

Note that $|S+Q|=|S|+2l-2k\leq |G|-10dm$. Applying Lemma~\ref{extendableconnect} to $\{a_0\}$, $\{b\}$ and $S+Q$, with $j=k$, we can find an $a_0,b$-path $P$, with interior vertices not in $V(S+Q)$ and length $2k+1$, and an edge $e$ of $P$, so that $S+Q+P-e$ is $(d,m)$-extendable. Then $Q+P$, which is an $a,b$-path with length $2l+1$ and interior vertices outside of~$S$, is as required.
\oof

\begin{corollary}\label{trivial}
Let $d,m,l\in\N$ satisfy $m\geq 1$ and $d\geq 3$. Letting $k=\lceil\log (2m)/ \log(d-1)\rceil$, suppose $l\geq 2k+1$. Let~$G$ be an~$m$-joined graph. Let~$S$ be a $(d,m)$-extendable subgraph of~$G$ with at most $|G|-10dm-(l-2k-1)$ vertices.

Suppose $a$ and $b$ are two distinct vertices in~$S$, both with degree at most $d/2$ in~$S$. Then there is an $a$,$b$-path $P$, with length $l$ and internal vertices outside of~$S$, so that $S+P$ is $(d,m)$-extendable.
\end{corollary}
\pr By Corollary~\ref{treebuild}, we can find a path~$Q$, with length $l-2k-1$ which has $a$ as an endvertex, and whose other vertices lie outside~$S$, so that $S+Q$ is $(d,m)$-extendable. Let $a_0$ be the endvertex of~$Q$ which is not $a$, unless $l-2k-1=0$, when~$Q$ is a single vertex, in which case let $a_0=a$.

Applying Lemma~\ref{extendableconnect} to $\{a_0\}$, $\{b\}$ and $S+Q$, with $j=k$, we can find an $a_0,b$-path $P$, with interior vertices not in $V(S+Q)$, and an edge $e$ of $P$, so that $S+Q+P-e$ is $(d,m)$-extendable. Then $Q+P$ is an $a,b$-path and the interior vertices of $Q+P$ lie outside of~$S$. Furthermore, by Lemma~\ref{edgeinsert} with the edge $e$, $S+Q+P$ is $(d,m)$-extendable, as required.
\oof

\re When applying the lemmas and corollaries concerning extendability, we will often refer to the extendable subgraph that we are growing as the \emph{working subgraph}. When checking the conditions of the corollaries apply, we will refer to the condition that requires an upper bound on the size of the working subgraph as the \emph{size condition}. When applying Corollary~\ref{trivial}, we will refer to the condition that $l\geq 2\lceil\log (2m)/ \log(d-1)\rceil+1$ as the \emph{length condition}.
\mark

\re In proving an important lemma later, Lemma~\ref{coverlots}, we will wish to add paths to a $(d,m)$-extendable subgraph, before later removing some of these paths while maintaining extendability. Our only tool to remove vertices while retaining extendability is Lemma~\ref{extendreverse}, which can remove leaves. If we have added a path between two vertices $x$, $y$ in the subgraph, both of which have other neighbours in the subgraph, then we cannot remove that path by removing leaves. However, instead, we will add a path between $x$ and $y$ with a missing edge to get a $(d,m)$-extendable structure, using Corollary~\ref{extendableconnectcor}. As the subgraph without the missing edge is $(d,m)$-extendable and was created by adding two separate paths (whose new endpoints lie in the missing edge), we can remove these paths, leaf by leaf, while maintaining extendability. If, on the other hand, we decide that we do want the $x,y$-path we can add the additional edge, which, by Lemma~\ref{edgeinsert}, also maintains extendability.
\mark

\subsection{Leaves and bare paths in trees}\label{sec33}
The following lemma by Krivelevich~\cite{MK10} shows that trees contain either many leaves or many vertex disjoint bare paths, that is, vertex disjoint paths which each have no branching points as interior vertices.

\lem \cite[Lemma 2.1]{MK10}\label{Kleaves} Let $k,l,n\in \N$. Let~$T$ be a tree with~$n$ vertices and at most $l$ leaves. Then~$T$ contains at least $n/(k+1)-(2l-2)$ vertex disjoint bare paths with length $k$.
\ma

By taking $l=n/4k-1$ in this lemma, we get that a tree $T$ with $n\geq 2$ vertices
either has at least $n/4k$ leaves or at least $n/4k$ vertex disjoint bare paths with length $k$. This is Lemma~\ref{split}, stated in Section~\ref{casedivision}. We will now augment Lemma~\ref{split} to show, more precisely, that, if a tree does not have many leaves, then it contains either many disjoint bare paths or a large set of well separated leaves (see Definition~\ref{sepdefn}).
Using this, in Lemma~\ref{farleaves}, we then show that the subclasses in Section~\ref{casedivision} do indeed cover $\TT(n,\Delta)$, for $n$ large, as claimed by Lemma~\ref{splitclass}.

\lem\label{farleaves}
Let $n,k,d\in\N$ satisfy $n\geq 60k$ and $k\geq 4d$. Suppose a tree~$T$ with~$n$ vertices has at most $n/5d$ leaves. Then~$T$ either contains a $2d$-separated set of at least $n/40k$ leaves or a collection of $n/40k$ vertex disjoint bare paths with length $k$.
\ma
\ifdraft
\else
\pr
By Lemma~\ref{split}, as~$T$ has at most $n/5d$ leaves, we may take $m:=\lceil n/4d\rceil$ disjoint bare paths~$P_i$, $i\in[m]$, with length~$d$ in~$T$. Contract each path $P_i$ to a single edge $e_i$, and contract all the other edges in~$T$ which are not in a path $P_i$. The resulting tree,~$S$ say, has one edge, $e_i$, for each path $P_i$, so $|S|= m+1$. Let $k_0:=\lceil k/d\rceil\leq 5k/4d$. By Lemma~\ref{split}, the tree~$S$ either has at least $m/4k_0$ leaves or contains at least $m/4k_0$ vertex disjoint bare paths with length $k_0$.

If~$S$ has at least $m/4k_0\geq n/16dk_0\geq n/20k$ leaves, then let these leaves be the vertices $v_i$, $i\in I$, contained in the edges $e_i$, $i\in I$, so that $|I|\geq n/20k$. As $n\geq 60 k$, we have that $|S|\geq 3$, and thus each edge $e_i$ can contain at most one leaf. For each $i\in I$, consider the tree which was contracted to get the vertex $v_i$, and call it $T_i$. Either $|T_i|=1$, in which case the vertex is unchanged and is a leaf in~$T$, which we call $u_i$, or $T_i$ has at least  two leaves, one of which must be a leaf in~$T$, which we call $u_i$. If $i\neq j$ then $u_i$ and $u_j$ are separated from each other in the tree by both $P_i$ and $P_j$, so are a distance at least $2d$ apart in~$T$, and thus we have a suitable set of leaves.

Suppose then that~$S$ has at least $m/4k_0\geq n/20k$ vertex disjoint bare paths with length $k_0$, and fix such a set of paths. Take a path in this set and look at the trees that were contracted together to get each internal vertex. If none of these trees had any leaves in~$T$ then the bare path in~$S$ came from a bare path with length at least $dk_0\geq k$ in~$T$. If this is true for at least $n/40k$ of the paths in~$S$ then we have the required number of vertex disjoint paths with length $k$. If not, then at least $n/40k$ of the paths have an internal vertex which came from contracting a tree that contains a leaf of~$T$. Collecting a leaf like this from each such path gives a set of leaves which are pairwise a distance at least $2d$ apart in~$T$.
\oof
\fi

\pr[Proof of Lemma~\ref{splitclass}] Let $\delta=1/400\Delta^{20}$ and suppose $T\in\TT(n,\Delta)$ satisfies $\lambda(T)\leq \d n$. If~$T$ has at least $n/100$ leaves, then, as the ball of radius 20 around any leaf in~$T$ contains at most $2\Delta^{20}$ vertices, we have that $\lambda(T)\geq n/200\Delta^{20}>\d n$, a contradiction. Let $k'=k'(T)$. Note that if $k'\leq 40$, then, by the definition of $k'(T)$ and Lemma~\ref{split}, as $T$ does not have $n/90(k'+1)$ vertex disjoint bare paths with length $k'+1$, $T$ has at least $n/4(k'+1)>\d n$ leaves, a contradiction.

Therefore $k'\geq 40$, and hence by the definition of $k'$ and Lemma~\ref{farleaves}, applied with $d=10$, we have that $\l(T)\geq n/40(k'+1)$. Therefore, $k'(T)\l(T)\geq n/100$, and, thus, for sufficiently large~$n$, Cases~A and~B cover all the trees $T\in \TT(n,\Delta)$ with $k(T)<10^2\log n$.
\oof

We will also wish to find a set of vertices which are far apart, even if the tree has many leaves. For this, we will use Lemma~\ref{Qfinder} and its following corollary.

\lem\label{Qfinder} Let $k\geq 0$ and $\Delta\geq 2$, and let~$T$ be a tree, with maximum degree at most $\Delta$, which contains at least $3\Delta^k$ vertices. Then there exists a subset $Q\subset V(T)$ of vertices which is $(2k+2)$-separated in~$T$, with $|Q|\geq |T|/(8k+8)\Delta^k$. Furthermore, we can find such a set~$Q$ so that each vertex is a leaf of~$T$, or is an interior vertex of a bare path in~$T$.
\ma
\pr Remove leaves from~$T$ repeatedly in $k$ rounds, to get the subtree~$S$, with $|S|\geq |T|/\Delta^k$. By Lemma~\ref{split},~$S$ either has at least $|S|/(8k+8)$ vertex disjoint bare paths of length $2k+2$ or at least $|S|/(8k+8)$ leaves.

If the first possibility occurs, then let~$Q_0$ be a set of vertices, with one vertex taken from the middle of each of the vertex disjoint bare paths of length $2k+2$. The vertices in~$Q_0$ are therefore pairwise a distance at least $2k+2$ apart in~$T$. For each vertex $q\in Q_0$, if $q$ has degree 2 in $T$ then let $l_q=q$, and, otherwise, there must be some leaf, $l_q$ say, of the tree attached to $q$ in $T$ which is removed in the formation of $S$. The set $Q=\{l_q:q\in Q_0\}$ then satisfies the requirements of the lemma.

If the second possibility occurs, then, for each leaf $r$ of~$S$, there must be a leaf, $l_r$ say, for which there is an $r,l_r$-path with length $k$ in~$T$ and edges outside of~$S$, otherwise $r$ would have been removed in the creation of~$S$. Let $Q=\{l_r:r\in L(S)\}$, where $L(S)$ is the set of leaves of~$S$. As $|S|\geq 3$, the leaves of~$S$ are pairwise a distance at least 2 apart in~$T$, and, therefore, the vertices in~$Q$ are pairwise a distance at least $2k+2$ apart in~$T$, so form a set as required by the lemma.
\oof

\begin{corollary}\label{Qfinder2} Let $k\geq 0$ and $\Delta\geq 2$, and let~$T$ be a tree, with maximum degree at most $\Delta$, containing the vertex set $X$, where $|X|\geq 3\Delta^k$. Then there exists a subset $Q\subset X$ of vertices which is $(2k+2)$-separated in~$T$, with $|Q|\geq |X|/(8k+8)\Delta^k$.
\end{corollary}
\pr Let $T'$ be the smallest subtree of~$T$ which contains every vertex in $X$. If $T'$ contains a vertex not in $X$ which has degree two, then pick such a vertex, delete it, and join its two neighbours by an edge. Repeat this until no such vertex exists, and call the resulting tree $T''$. Note that $X\subset V(T'')$, and that the vertices in~$T$ are at least as far apart as they are in $T''$. Every leaf of~$T'$, and hence of~$T''$, is a vertex in~$X$, and every interior vertex in a bare path in $T''$ is also in~$X$. Therefore, by applying Lemma~\ref{Qfinder}, we can find a set satisfying our requirements.
\oof

Once we have found a well separated vertex set~$Q$ in a tree~$T$, we will embed~$T$ by starting with a single vertex and building up a tree by adding paths with a copy of some vertex from~$Q$ at the end. The following lemma will give us a good order in which to embed the vertices in~$Q$.

\lem\label{order} Let $k,l\in\N$. Suppose~$T$ is a tree containing the vertex $t\in V(T)$. Suppose $Q\subset V(T)$, with $|Q|=l$, is a $2k$-separated set in~$T$. Then, there is a sequence $T_1,\ldots,T_l$ of subtrees of~$T$ so that $t\in V(T_1)$ and, for each $i$, $1< i\leq l$, $T_{i}$ is formed from $T_{i-1}$ by adding a bare path of length at least $k$ which adds a vertex from~$Q$ as a leaf of $T_i$.
\ma
\pr
Given a sequence of vertices $x_1,\ldots,x_i\in V(T)$, for any $i$, let $T(x_1,\ldots,x_i)$ be the smallest subtree of~$T$ containing these vertices. Let $T_0=T(t)=I(\{t\})$. For each $i\in[l]$, pick a vertex $q_i$ in $Q\setminus\{q_1,\ldots,q_{i-1}\}$ so that, if $T_i=T(t,q_1,\ldots,q_i)$, then $|T_i|-|T_{i-1}|$ is minimised. Note that we allow $q_1=t$. We will show that these subtrees satisfy the lemma.

For each $i\in[l]$, the tree $T_{i}$ is formed from $T_{i-1}$ by adding a bare path to $T_i$ which ends in~$q_i$. Say, for each $i$, that path is $P_i$ with endvertices $v_i\in V(T_{i-1})$ and $q_i$. Suppose for some $i$ this path has length $k_1<k$. Let $j\geq 1$ be the smallest such $j$ for which $v_i\in V(T_j)$, so that $j<i$. The vertex $v_i$ must then belong on the path $P_j$. Let the distance between $v_i$ and $q_j$ on $P_j$ be $k_2$. If $k_1<k_2$, then at stage $j$ the vertex $q_i$ would have been added instead of $q_j$, contradicting the above process. We must therefore have $k_1\geq k_2$. There is then a path between $q_i$ and $q_j$ in $P_i\cup P_j$, which has length $k_1+k_2<2k$, a contradiction.
\oof

\subsection{Dividing trees}\label{secsplittree}
We will often divide a tree into pieces with different properties, for which we use the following definition.
\de
Where~$S$ is a tree, we say two subtrees $S_1$ and $S_2$ \emph{divide~$S$} if they cover~$S$ and intersect on precisely one vertex.
\fn
Dividing trees so that the subtrees satisfy certain properties is shown to be possible by the following proposition, after which we record some particularly useful consquences.

\begin{prop}\label{divide}
Given a tree~$S$, with a subset $Q\subset V(S)$, we can find two trees $S_1$ and $S_2$ which divide~$S$ for which $|Q\cap V(S_1)|, |Q\cap V(S_2)|\geq |Q|/3$.
\end{prop}
\pr Note that, if $0<|Q|\leq 3$, then we may easily find such subtrees by selecting $q\in Q$ and taking the subtrees~$S$ and $I(\{q\})$. Let us assume then that $|Q|\geq 4$.
In this proof, when we have a subtree $S_i\subset S$, indexed by any $i$, we will use $Q_i=Q\cap V(S_i)$ without further definition.

Take two subtrees $S_1$ and $S_2$, subject to $S_1$ and $S_2$ dividing~$S$ and $|Q_1|\geq |Q_2|$, so that $|Q_1|-|Q_2|$ is minimised and, among all such pairs of subtrees, so that $|S_1|-|S_2|$ is then minimised. Note that, if $s\in V(S)$, then~$S$ and $I(\{s\})$ divide~$S$, and hence such subtrees $S_1$ and $S_2$ must exist.

Suppose, for contradiction, that there are no subtrees satisfying the requirements of the lemma. As $S_1$ and $S_2$ are then not suitable subtrees, we have $|Q_2|<|Q|/3$, and, as $|Q_1|+|Q_2|\geq |Q|$, we have $|Q_1|>2|Q|/3$. Therefore,
\begin{equation}\label{tocontra}
|Q_1|-|Q_2|\geq \lceil 2|Q|/3\rceil- \lfloor |Q|/3\rfloor\geq 2.
\end{equation}
Let $v$ be the single vertex common to both $S_1$ and $S_2$. If $v$ has only one neighbour in $S_1$, $x$ say, then consider the two trees $S_3$ and $S_4$ with the vertex sets $V(S_1)\setminus\{v\}$ and $V(S_2)\cup\{x\}$ respectively. Then, as we have removed one vertex from $S_1$ to get $S_3$ and added one vertex to $S_2$ to get $S_4$, we have that, using (\ref{tocontra}),
\[
|Q_1|-|Q_2|\geq|Q_3|-|Q_4|\geq |Q_1|-|Q_2|-2\geq 0.
\]
Therefore, as $S_3$ and $S_4$ divide~$S$ and $|S_3|-|S_4|< |S_1|-|S_2|$, this contradicts the choice of $S_1$ and $S_2$.

Therefore, $v$ has at least two neighbours in $S_1$, so we may divide $S_1$ into two trees $S_5$ and $S_6$, with $|S_5|,|S_6|\geq 2$, so that $S_5$ and $S_6$ intersect only on $v$. Without loss of generality, say that $|Q_5|\geq |Q_6|$. Then, as $|Q_5|+|Q_6|\geq |Q_1|>2|Q|/3$, we must have that $|Q_5|>|Q|/3$. Consider the trees $S_5$ and $S_7=S_2\cup S_6$. As $|S_6|\geq 2$ and $|S_5|<|S_1|$, we have that $|S_5|-|S_7|<|S_1|-|S_2|$.
Therefore, as $|Q_5|-|Q_7|\leq |Q_1|-|Q_2|$, to avoid contradicting the choice of $S_1$ and $S_2$ we must have that $|Q_5|-|Q_7|<0$, and hence $|Q_7|>|Q_5|>|Q|/3$. Thus, as $S_5$ and $S_7$ divide $S$, these subtrees satisfy the requirements of the lemma, a contradiction.
\oof

The simplest use of Proposition~\ref{divide} demonstrates, as follows, that we may divide a tree into two large pieces.
\begin{corollary}\label{divideeven} Given a tree~$S$, we can find two trees $S_1$ and $S_2$ which divide~$S$ for which $|S_1|, |S_2|\geq |S|/3$.
\end{corollary}
\pr Taking $Q=V(S)$, by Proposition~\ref{divide} there exist two trees $S_1$ and $S_2$ which divide~$S$ for which $|S_1|,|S_2|\geq |S|/3$, as required.
\oof

It will be useful to divide a tree so that each subtree contains plenty of vertex disjoint bare paths. The following corollary shows that this can be done, if the original tree contains sufficently many such paths.
\begin{corollary}\label{dividepath} Let $l\in \N$. Given a tree~$S$, which contains the vertex disjoint paths $P_i$, $i\in [l]$, we can find two trees $S_1$ and $S_2$ which divide~$S$ so that both $S_1$ and $S_2$ contain at least $l/3-1$ of the paths $P_i$.
\end{corollary}
\pr For each $i\in [l]$, let $v_i$ be some vertex in $V(P_i)$. Let $Q=\cup_i\{v_i\}$. Then, by Proposition~\ref{divide}, there exist two trees $S_1$ and $S_2$ which divide~$S$ for which $|Q\cap V(S_1)|,|Q\cap V(S_2)|\geq |Q|/3=l/3$. Let $w$ be the single vertex in the set $V(S_1)\cap V(S_2)$. For each $i\in [l]$ with $v_i\in Q\cap V(S_1)$, if $w\notin V(P_i)$, then the path $P_i$ must be a subgraph of $S_1$. Therefore, as $w$ is in at most one of the paths~$P_i$, at least $|Q\cap V(S_1)|-1\geq l/3-1$ of the paths $P_i$ must be a subgraph of $S_1$. Similarly, at least $|Q\cap V(S_2)|-1\geq l/3-1$ of the paths $P_i$ must be a subgraph of~$S_2$.
\oof

Using Corollary~\ref{divideeven}, we can find a subtree with, approximately, a certain number of vertices.

\begin{prop}\label{littletree} Let $n,m\in \N$ satisfy $1\leq m\leq n/3$. Given any tree~$T$ with~$n$ vertices and a vertex $t\in V(T)$, we can find two trees $T_1$ and $T_2$ which divide~$T$ so that $t\in V(T_1)$ and $m\leq |T_2|\leq 3m$.
\end{prop}
\pr Find trees $T_1$ and $T_2$, with $t\in V(T_1)$, which divide~$T$, so that $|T_2|$ is minimised subject to $|T_2|\geq m$. Some such trees exist as $I(\{t\})$ and~$T$ satisfy these conditions. Suppose $|T_2|>3m\geq 3$. By Corollary~\ref{divideeven}, we can find trees $T_3$ and $T_4$ which divide $T_2$ so that $|T_3|,|T_4|\geq \lceil |T_2|/3\rceil\geq m+1\geq 2$. Let~$t'$ be the single vertex in both $T_1$ and $T_2$, and suppose, without loss of generality, that $t'\in V(T_3)$. Consider the trees $T_5=T_1\cup T_3$ and $T_4$ which divide~$T$. The tree $T_5$ contains~$t$, and $|T_2|>|T_4|\geq m$, a contradiction. Therefore, $|T_2|\leq 3m$.
\oof

\subsection{Edges between sets in the random graph}\label{sec00}
To confirm some standard facts concerning the random graph, we will use the following form of Chernoff's inequality (see, for example, Janson, \L uczak and Ruci\'nski~\cite[Corollary 2.3]{jlr11}).
\lem\label{chernoff} If $X$ is a binomial variable with standard parameters~$n$ and $p$, denoted $X=\mathrm{Bin}(n,p)$, and $\e$ satisfies $0<\e\leq 3/2$, then
\[
\P(|X-\E X|\geq \e \E X)\leq 2\exp\left(-\e^2\E X/3\right).\hfill\qedhere
\]
\ma

We will later require some of the properties of a random graph to hold with higher probability than the typical `almost sure' requirement, for which we use the following definition.
\begin{defn}
In a random graph $G=G(n,p)$, we say a property $P$ \emph{holds with very high probability} if
\[
\P(G\text{ has property }P)\geq 1-n^{-\omega(1)}.
\]
\end{defn}

A simple application of Lemma~\ref{chernoff} gives the following loose estimate on the number of edges we can expect in a random graph.

\begin{prop}\label{alledges} If $p=\omega(\log n/n^2)$, then, with very high probability, the random graph $\GG(n,p)$ has between $pn^2/4$ and $pn^2$ edges.
\end{prop}
\pr If $G=\GG(n,p)$, then $\E|E(G)|=p\binom{n}{2}$. Therefore, by Lemma~\ref{chernoff} with $\e=1/5$, if~$n$ is large, then
\[
\P(pn^2/4\leq |E(G)|\leq pn^2)\geq 1-2\exp\left(-p\binom{n}{2}/75\right)=1-n^{-\omega(1)}.\qedhere
\]
\oof

Throughout this paper
we will be very interested in a parameter denoted by~$m$. We will take~$m$ to be small, subject to the condition that we can reasonably expect an edge between any two disjoint subsets with size~$m$. The following proposition will allow us to take a good value for~$m$.

\begin{prop}\label{generalprops} If $np>20$, then, with very high probability, any two disjoint sets~$A$ and~$B$ of vertices of $G=\GG(n,p)$ with $|A|=|B|=\lceil 5\log(np)/ p\rceil$ have some edge between them.
\end{prop}
\pr
Let $m=\lceil 5\log(np)/ p\rceil$. If $q$ is the probability that there exist two disjoint subsets of size~$m$ which have no edge between them, then
\begin{equation}\label{qqqqqqqqqqqqqq}
q\leq \binom{n}{m}^2(1-p)^{m^2}\leq \left(\frac{en}{m}\right)^{2m}e^{-pm^2}
\leq (np)^{2m}e^{-5m\log(np)}= e^{-3m\log(np)}.
\end{equation}
If $p\leq n^{-1/2}$, then $m\log(np)\geq \log^2(np)/p\geq n^{1/2}$, and if $p\geq n^{-1/2}$, then $m\log(np)\geq\log^2(np)\geq \log^2 n/4$. Therefore, $m\log(np)=\omega(\log n)$, and hence, by (\ref{qqqqqqqqqqqqqq}), $q\leq n^{-\omega(1)}$, as required.
\oof

Proposition~\ref{generalprops} will cover most of the situations where we wish to have some edge between two large subsets in a random graph. However, we will also use a result concerning the number of edges we are likely to find between any two large sets. The following proposition is part of Proposition 3.2 in~\cite{AKS07} by Alon, Krivelevich and Sudakov, but as we require the property to hold with a higher probability we include a proof here.

\begin{prop}[]\label{AKS1}
Let $np>20$ and $G=\GG(n,p)$. Then, with very high probability,
the number of edges between any two disjoint sets of vertices~$A$ and $B$ with $p|A||B|\geq 24n$ is at least $p|A||B|/2$ and at most $3p|A||B|/2$.
\end{prop}
\pr For any two disjoint sets of vertices~$A$ and $B$, $e_G(A,B)$ is a binomially distributed variable with mean $p|A||B|$. Therefore, if $p|A||B|\geq 24n$ then, by Lemma~\ref{chernoff} with $\e=1/2$,
\[
\P(|e_G(A,B)-p|A||B||\geq p|A||B|/2)\leq 2\exp\left(-p|A||B|/12\right)\leq 2\exp(-2n).
\]
Therefore, the probability there are two disjoint sets of vertices~$A$ and $B$ with $p|A||B|\geq 24n$ and either $e_G(A,B)< p|A||B|/2$ or $e_G(A,B)>3p|A||B|/2$ is at most
\[
2^n\cdot 2^n\cdot 2e^{-2n}=2(4/e^2)^{-n}=n^{-\omega(1)}.\hfill\qedhere
\]
\oof


\subsection{Short paths in the random graph}\label{shortpathsgnp}
We will also need an upper bound on the number of paths of length 2 we can expect to find in a sparse random graph. Each potential such path does not appear independently of the others, complicating the proof of any such bound. For convenience, we will use the following result by Vu~\cite{Vu01} on the number of subgraphs of a random graph  isomorphic to some fixed balanced subgraph $H$ (a balanced subgraph $H$ is one whose density, calculated as $|E(H)|/|V(H)|$, is at least as large as the density of any subgraph). Theorem~\ref{subcount} records the part of Theorem 1.1 in \cite{Vu01}  we wish to use, and Corollary~\ref{fewpathslengthtwo} confirms the specific case we need.
\begin{theorem}\label{subcount}\cite[Theorem 1.1]{Vu01}
Let $0<\e\leq 1$. Suppose $H$ is a balanced graph with $k$ vertices, and $Y$ is the random variable counting the number of subgraphs isomorphic to $H$ in a random graph $\GG(n,p)$. Suppose further that $\e^2(\E Y)^{1/(k-1)}=\omega(\log n)$. Then there is a positive constant $c=c(H)$ such that
\[
\P(Y\geq (1+\e)\mathbb{E} Y)\leq \exp(-c\e^2(\mathbb{E}Y)^{1/(k-1)}).
\]
\end{theorem}
\begin{corollary}\label{fewpathslengthtwo}
If $n^3p^2=\omega(\log^2 n)$, then, with very high probability, $\GG(n,p)$ contains at most $n^3p^2$ paths of length two.
\end{corollary}
\ifdraft
\else
\pr
Let $P$ be the path of length two and let $Y$ be the number of subgraphs of $\GG(n,p)$ which are isomorphic to $P$. Then $(\mathbb{E}Y)^{1/2}=(3\binom{n}{3}p^2)^{1/2}=\omega(\log n)$. As $P$ is balanced, by Theorem~\ref{subcount} with $\e=1$, $\P(Y\geq 6\binom{n}{3}p^2)=n^{-\omega(1)}$.
\oof
\fi

\subsection{Matchings from expansion conditions}\label{sec11}
In what remains of this section, we will show several results on matchings and expansion. We use the common notion of a generalised matching, which, as is well-known (see, e.g., Bollob\'as~\cite{bollomod}) exists exactly when a \emph{generalised matching condition} holds, as follows.

\begin{defn}\label{genmatch}
Given a bipartite graph $G$ with vertex classes $A$ and $B$, and a function $f:A\to\N$, an \emph{$f$-matching from $A$ into~$B$} is a collection of disjoint sets $\{X_a\subset N(a):a\in A\}$ so that, for each $a\in A$, $|X_a|=f(a)$. If $d\in\N$, and $f(a)=d$ for each $a\in A$, we refer to an $f$-matching as a \emph{$d$-matching}. We refer to a $1$-matching as a \emph{matching}.
\end{defn}

\begin{theorem}\label{matchingtheorem} Given a bipartite graph $G$ with vertex classes $A$ and $B$, and a function $f:A\to\N$, if, for every subset $U\subset A$ we have $|N(U)|\geq \sum_{a\in U}f(a)$, then there is an $f$-matching from $A$ into $B$.
\hfill\qed
\end{theorem}

We can find matchings in graphs using certain simple expansion conditions, as follows.
\lem\label{matchings} Let $n,m,\Delta\in \N$ and $f:[n]\to[\Delta]$. Let $D=\sum_{i\in[n]}f(i)$. Suppose $H$ is a bipartite graph with vertex classes $A=[n]$ and $B=[D]$ where the following properties hold.
\begin{itemize} \refstepcounter{equation}\label{tah1}
\item If $U\subset A$ and $|U|\leq m$, then $|N(U)|\geq \Delta|U|$.\hfill(\arabic{equation})\refstepcounter{equation}\label{tah2}
\item If $U\subset B$ and $|U|\leq m$, then $|N(U)|\geq |U|$.\hfill(\arabic{equation})\refstepcounter{equation}\label{tah3}
\item If $U\subset A$ and $V\subset B$, with $|U|,|V|\geq m$, then $e(U,V)>0$.\hfill(\arabic{equation})
\end{itemize}
Then, there is an $f$-matching from~$A$ into $B$.
\ma
\pr We will show that the relevant general matching condition holds from~$A$ into $B$. Let $U\subset A$. If $|U|\leq m$, then, by (\ref{tah1}), we have $|N(U)|\geq \Delta|U|\geq \sum_{i\in U}f(i)$.

If $|U|\geq m$, then, as there are no edges between $U$ and $B\setminus N(U)$, by (\ref{tah3}), we have that $|B\setminus N(U)|\leq m$. Therefore, $|N(U)|=|B|-|B\setminus N(U)|\geq |B|-m$.
Thus, if $|U|\geq m$ and $\sum_{i\in U}f(i)\leq |B|-m$, then $|N(U)|\geq \sum_{i\in U}f(i)$.
If $|U|\geq m$ and $\sum_{i\in U}f(i) > |B|-m$, then let $U'=B\setminus N(U)$, so that, as $|N(U)|\geq |B|-m$, $|U'|\leq m$. Hence, by (\ref{tah2}), we have $|N(U')|\geq |U'|$. Then, as $U$ and $N(U')$ are disjoint subsets of~$A$, $|A\setminus U|\geq |N(U')|\geq |U'|$, so that
\begin{equation}
|N(U)|=|B\setminus U'|=|B|-|U'|\geq |B|-|A\setminus U|.\label{pah2}
\end{equation}
As $f(i)\geq 1$, for each $i\in A$,
\begin{equation}
|A\setminus U|\leq \sum_{i\in A\setminus U}f(i)=|B|-\sum_{i\in U}f(i).\label{pah3}
\end{equation}
Therefore, by (\ref{pah2}) and (\ref{pah3}), we have that $|N(U)|\geq \sum_{i\in U}f(i)$.
Thus, by Theorem~\ref{matchingtheorem}, an $f$-matching exists from~$A$ into~$B$.
\oof


\subsection{Expansion from large set connectivity}
We can find small set expansion conditions in $m$-joined graphs, as follows.

\begin{prop}\label{neat} Let $m,d\in \N$.  Suppose~$G$ is an~$m$-joined graph, and $W\subset V(G)$ satisfies $|W|\geq (3d+4)m$. Then, there is some set $B\subset V(G)$, with $|B|\leq m$, so that, if $U\subset V(G)\setminus B$ and $|U|\leq 2m$, then $|N(U,W\setminus B)|\geq d|U|$.
\end{prop}
\pr Let $B\subset V(G)$ be a maximal subset subject to $|B|\leq m$ and $|N(B,W)|< d|B|$. We will show that $B$ satisfies the conditions of the proposition. Let $U\subset V(G)\setminus B$ satisfy $0<|U|\leq 2m$, and suppose that $|N(U,W\setminus B)|<d|U|$. Then $|N(B\cup U,W)|< d(|B|+|U|)=d|B\cup U|$, and $|B\cup U|>|B|$. Therefore, by the choice of $B$, $|B\cup U|\geq m$. There are no edges between $B\cup U$ and $W\setminus (B\cup U\cup N(B\cup U))$, so, as $|B\cup U|\geq m$, $|W\setminus (B\cup U\cup N(B\cup U))|< m$. Therefore, as $|B\cup U|\leq 3m$,
\[
|N(B\cup U,W)|\geq |W|-|B\cup U|-m\geq (3d+4)m-3m-m\geq d|B\cup U|,
\]
a contradiction. Therefore, if $U\subset V(G)\setminus B$ and $|U|\leq 2m$, then $|N(U,W\setminus B)|\geq d|U|$, as required.
\oof

We will also use the following bipartite version of Proposition~\ref{neat}.

\begin{prop}\label{neat2} Let $m,d\in \N$. Suppose $H$ is a bipartite graph with vertex classes $X$ and $Y$, in which any two disjoint subsets with size~$m$ from $X$ and $Y$ respectively have some edge between them. Suppose $X_0\subset X$ and $Y_0\subset Y$ satisfy $|X_0|,|Y_0|\geq (3d+4)m$. Then, there is some set $B\subset V(H)$, with $|B|\leq 2m$, so that if $U\subset V(H)\setminus B$ and $|U|\leq m$, then $|N(U,(X_0\cup Y_0)\setminus B)|\geq d|U|$.
\end{prop}
\pr Let $B\subset V(H)$ be a maximal subset subject to $|B|\leq 2m$ and $|N(B,X_0\cup Y_0)|< d|B|$. We will show that $B$ satisfies the conditions of the proposition. Let $U\subset V(H)\setminus B$ satisfy $0<|U|\leq m$, and suppose that $|N(U,(X_0\cup Y_0)\setminus B)|<d|U|$. Then, $|N(B\cup U,X_0\cup Y_0)|< d(|B|+|U|)=d|B\cup U|$, and $|B\cup U|>|B|$. Therefore, by the choice of $B$, $|B\cup U|\geq 2m$. Then, either $|(B\cup U)\cap X|\geq m$ or $|(B\cup U)\cap Y|\geq m$. If $|(B\cup U)\cap X|\geq m$, then, as there are no edges between $(B\cup U)\cap X$ and $Y_0\setminus N((B\cup U)\cap X)$, we must have that $|Y_0\setminus N((B\cup U)\cap X)|\leq m$. Therefore, as $|B\cup U|\leq 3m$,
\begin{align*}
|N(B\cup U,X_0\cup Y_0)|&\geq |N((B\cup U)\cap X,Y_0)|-|(B\cup U)\cap Y_0|\geq |Y_0|-m-|B\cup U|\\
&\geq (3d+4)m-m-3m\geq d|B\cup U|,
\end{align*}
a contradiction. A similar contradiction follows if $|(B\cup U)\cap Y|\geq m$. Therefore, if $U\subset V(H)\setminus B$ and $|U|\leq m$, then $|N(U,(X_0\cup Y_0)\setminus B)|\geq d|U|$.
\oof

A further bipartite version of Proposition~\ref{neat} follows, simplified to give the specific case that we will use.

\begin{prop}\label{neat3} Let $m,d\in \N$. Suppose $H$ is a bipartite graph with vertex classes $X$ and $Y$, in which any two disjoint subsets with size~$m$ from $X$ and $Y$ respectively have some edge between them. If $|Y|\geq (3d+4)m$, then there is some set $B\subset X$, with $|B|\leq m$, so that, for any subset $U\subset X\setminus B$ with $|U|\leq 2m$, we have $|N(U)|\geq d|U|$.
\end{prop}
\pr Add edges between every pair of vertices from $X$ and every pair of vertices from $Y$, and call the resulting $m$-joined graph~$G$. As $|Y|\geq (3d+4)m$, by Proposition~\ref{neat}, there is a subset $B\subset V(G)$, with $|B|\leq m$, so that for every subset $U\subset V(G)\setminus B$ with $|U|\leq 2m$, we have $|N_G(U,Y\setminus B)|\geq d|U|$. Therefore, for every subset $U\subset X\setminus B$ with $|U|\leq 2m$, $|N_H(U,Y)|\geq |N_G(U,Y\setminus B)|\geq d|U|$, as required.
\oof

\subsection{Matchings and expansion in random graphs}
In a random graph, it is likely that we can find a general matching from a small set into a larger set, as follows.
\lem\label{logmatching} Let $n,d\in \N$ and $p>0$. Let $A,W\subset [n]$ be disjoint sets satisfying $p|W|\geq 80\log n$, $1\leq d\leq p|W|/4$ and $|A|\leq |W|/2d$. Then, with probability $1-o(n^{-2})$, there is a~$d$-matching from~$A$ into $W$ in the random graph $\GG(n,p)$.
\ma
\pr Let $l=|A|$ and label the vertices in~$A$ as $a_1,\ldots,a_l$. Let $W_0=\emptyset$. Reveal the edges between each vertex $a_i$ and the set $W$ in $\GG(n,p)$ in turn. At each reveal, if it is possible to find~$d$ neighbours of $a_i$ in $W\setminus W_{i-1}$, then do so, adding them to $W_{i-1}$ to get $W_i$. For each $i\in[l]$, we have $d(i-1)\leq |W|/2$. Therefore, by Lemma~\ref{chernoff}, for each $i\in [l]$, the probability we will not find~$d$ neighbours in $W\setminus W_{i-1}$ for the vertex $a_i$ is
\[
\P(\mathrm{Bin}(|W\setminus W_{i-1}|,p)< d)\leq \P(\mathrm{Bin}(|W|/2,p)\leq p|W|/4)\leq 2\exp(-p|W|/24)=o(n^{-3}).
\]
Therefore, the probability that this process fails to find a~$d$-matching from~$A$ into $W$ is $o(n^{-2})$.
\oof

Expansion properties are likely to exist in random graphs, as shown by the following lemma.

\lem\label{generalexpand} Let $n\in \N$, $p>0$ and $W\subset [n]$ satisfy $p|W|\geq 200\log n$. Let~$d$ satisfy $1\leq d\leq p|W|/240\log(np)$. Then, taking $G=\GG(n,p)$, with probability $1-o(n^{-2})$, any subset $U\subset V(G)$ with $|U|\leq |W|/4d$ satisfies $|N(U,W)|\geq d|U|$.
\ma
\pr Let $m=15\log(np)/ p$, so that $d\leq |W|/16m$. Reveal edges with probability $p/2$ among the vertex set $[n]$ to get the graph $G_1$. By Proposition~\ref{generalprops}, with probability $1-o(n^{-2})$, $G_1$ is~$m$-joined. Therefore, as $|W|\geq 16dm\geq (3(4d)+4)m$, there is, by Proposition~\ref{neat}, a subset $B\subset [n]$, with $|B|\leq m$, so that, if $U\subset [n]\setminus B$ and $|U|\leq m$, then $|N(U,W)|\geq 4d|U|$.

Reveal more edges with probability $p/2$ between $B$ and $W\setminus B$ in the graph $G_1$ to complete the random graph~$G$, in which each edge has been revealed with probability at most $p$. Note that
\begin{equation}\label{randomname}
|W\setminus B|=|W|-|B|\geq 16dm-m\geq 15dm,
\end{equation}
so that $|B|\leq |W|/15d$ and thus $(p/2)\cdot |W\setminus B|\geq 80\log n$. Using (\ref{randomname}), we have that
\[
(p/2)\cdot |W\setminus B|/4\geq 15dmp/8\geq 20d\log(np)\geq 4d.
\]
Therefore, by Lemma~\ref{logmatching}, with probability $1-o(n^{-2})$, there is a $4d$-matching from $B$ into $W\setminus B$. Given $U\subset V(G)$ with $|U|\leq m$, and noting either $|U\cap B|\geq |U|/2$ or $|U\setminus B|\geq |U|/2$, we have
\[
|N(U,W)|\geq \max\{|N(U\cap B,W)|,|N(U\setminus B,W)|\}-|U|\geq 4d\cdot |U|/2-|U|\geq d|U|.
\]

If $U\subset V(G)$ with $m\leq |U|\leq |W|/4d$, then, as $G_1$, and hence $G$, is~$m$-joined,
\[
|N(U,W)|\geq |W|-m-|U|\geq |W|/2\geq d|U|.\hfill\qedhere
\]
\oof

We will often use Lemma~\ref{generalexpand} through the following corollary.
\begin{corollary}\label{generalexpandcor} Let $n\in \N$, $p>0$ and $W\subset [n]$ satisfy $p|W|\geq 200\log n$. Let~$d$ satisfy $1\leq d\leq p|W|/240\log(np)$ and let $m\leq |W|/8d$. Then, taking $G=\GG(n,p)$, with probability $1-o(n^{-2})$, given any subset $X\subset V(G)\setminus W$, the subgraph $I(X)$ is $(d,m)$-extendable in $G[X\cup W]$.
\end{corollary}
\pr By Lemma~\ref{generalexpand}, with probability $1-o(n^{-2})$, any subset $U\subset V(G)$ with $|U|\leq |W|/4d$ satisfies $|N(U,W)|\geq d|U|$. Given any subset $X\subset V(G)\setminus W$, let $H=G[X\cup W]$. For every subset $U\subset V(H)$ with $|U|\leq 2m$,
\[
|N_H(U)\setminus X|\geq |N_H(U,W)|\geq d|U|.
\]
Therefore, by Proposition~\ref{uttriv}, $I(X)$ is $(d,m)$-extendable in $H$.
\oof

The following simple lemma shows that a matching is likely to exist between two equal sized large sets in a random graph.

\lem\label{matchingthres} Let $H$ be a random bipartite graph with classes~$A$ and $B$, with $|A|=|B|=n$, which has edges between~$A$ and $B$ present independently with probability $p\geq 250\log n/ n$. Then, with probability $1-o(n^{-2})$, there is a matching between~$A$ and $B$.
\ma
\pr We will show that, with probability $1-o(n^{-2})$, the conditions required for Lemma~\ref{matchings} hold for a matching between~$A$ and $B$.

By considering $H$ as a subgraph of $G(2n,p)$, and noting that, for sufficiently large~$n$, $p|B|\geq 200\log(2n)$, by Lemma~\ref{generalexpand}, with probability $1-o(n^{-2})$, any subset $U\subset A$ with $|U|\leq n/4$ satisfies $|N(U,B)|\geq |U|$. Similarly, with probability $1-o(n^{-2})$, any subset $U\subset B$ with $|U|\leq n/4$ satisfies $|N(U,A)|\geq |U|$.
By Proposition~\ref{generalprops}, with probability $1-o(n^{-2})$ any two disjoint subsets of~$A$ and $B$ respectively which each have at least $n/4$ vertices must have some edge between them in $H$.

Therefore, by Lemma~\ref{matchings} with $m=n/4$, with probability $1-o(n^{-2})$ there is a matching between~$A$ and $B$ in $H$.
\oof

As mentioned in Section~\ref{1boutline} (where we call them matchmaker sets), we will want to find a small subset in our random graph into which every sufficiently small vertex set expands, as follows.

\begin{prop}\label{goodset} Let $n\in \N$ and suppose $np\geq 10^4\log n$ and $m=10\log(np)/p$. Let $W\subset [n]$ satisfy $|W|\geq 3n/8$. With probability $1-o(n^{-2})$, in the random graph $G=\GG(n,p)$ there is some set $X\subset W$ with $|X|\leq 8m$ such that, if $U\subset V(G)\setminus X$ and $|U|\leq m$, then $|N(U,X)|\geq |U|$.
\end{prop}
\pr
Let $X_0\subset W$ satisfy $|X_0|=7m$ and $A=W\setminus X_0$. By Lemma~\ref{generalexpand}, with probability $1-o(n^{-2})$, any subset $U\subset V(G)$ with $|U|\leq m$ satisfies $|N(U,A)|\geq |U|$. By Proposition~\ref{generalprops}, with probability $1-o(n^{-2})$, the graph~$G$ is~$m$-joined.


By Proposition~\ref{neat}, there is a subset $B\subset V(G)$ with $|B|\leq m$ so that, if $U\subset V(G)\setminus B$ and $|U|\leq m$, then $|N(U, X_0\setminus B)|\geq |U|$.
Find a matching from $B\setminus A$ into $A$, using Theorem~\ref{matchingtheorem}, and let $D\subset A$ be the image of $B\setminus A$ under this matching, so that $|D|+|B\cap A|\leq m$.

Let $X=X_0\cup D\cup (B\cap A)$, so that $X\subset W$ and $|X|\leq 8m$. Given any subset $U\subset V(G)\setminus X$, we have $|N(U,X)|\geq |N(U\setminus B,X_0)|+|N(U\cap B,D)|\geq |U\setminus B|+|U\cap B|= |U|$, as required.
\oof

\subsection{Expansion from minimum degree conditions}
Our next lemma, Lemma~\ref{mindegexp3}, can be used to turn minimum degree conditions in a random graph into an expansion property. For certain values of~$m$,~$d$ and $D$ it implies that, in the random graph $G=\GG(n,p)$, if each vertex in a set~$A$ has at least $D$ neighbours in the set $B\subset V(G)\setminus A$, and $|A|\leq m$, then $|N(A,B)|\geq d|A|$.
The proof of the lemma follows a section of the proof by Alon, Krivelevich and Sudakov of Lemma 3.1 in \cite{AKS07}.

\lem \label{mindegexp3} Suppose $p$ satisfies $\log^{3}n/ n\geq p\geq 10^8/n$, and let $d=np/10^6\log(np)$, $m=160\log(np)/p$ and $D=np/100$. Then, with probability $1-o(n^{-2})$, $G=\GG(n,p)$ has no two disjoint sets~$A$ and $B$ with $0< |A|\leq m$, $|B|\leq d|A|$ and $e_G(A,B)\geq D|A|$.
\ma
\pr
If~$G$ does not have  this property then there exist two disjoint sets $A,B\subset V(G)$, where $0<|A|\leq m$, $|B|=d|A|$ and $e_{G}(A,B)\geq D|A|$ (adding vertices to $B$ if necessary). Let~$p_t$ be the probability two such sets occur with $|A|=t\leq m$. Then, we have
\begin{align}\allowdisplaybreaks
 p_t &\leq \binom{n}{t}\binom{n}{dt}\binom{dt^2}{Dt}p^{Dt} \nonumber \\
 &\leq\left(\frac{en}{t}\left(\frac{en}{dt}\right)^{d}\left(\frac{edtp}{D}\right)^{D}\right)^t \nonumber
\\ &\leq\left(\left(\frac{n}{t}\right)^{2d}\left(\frac{edtp}{D}\right)^{D}\right)^t \nonumber
\\ &=\left(\left(\frac{ednp}{D}\right)^{2d}\left(\frac{edtp}{D}\right)^{D-2d}\right)^t.\label{lastone}
\end{align}
If $t<10^4\log n$, then $tp\leq 10^4\log^{4}n/ n$, and hence, as $D\geq 10^4d$, by (\ref{lastone}), we have
\begin{align}
p_t&\leq \left(\left(np\right)^{2d}\left(\frac{e\log^4n}{n}\right)^{(10^{4}-2)d}\right)^t \nonumber \\
&\leq \left(\log^6 n\left(\frac{e\log^4n}{n}\right)^{10^{3}}\right)^{dt} \nonumber \\
&= o(n^{-3}).\label{lastone2}
\end{align}
If $t\geq 10^4\log n$, then, as $dtp\leq dmp=16D/10^3\leq D/20e$, by (\ref{lastone}) we have
\begin{align}\allowdisplaybreaks
 p_t &\leq\left(\left(np\right)^{2d}\left(\frac{1}{20}\right)^{np/200}\right)^t \nonumber \\
&= \left(\left(np\right)^{2np/10^6\log(np)}\left(\frac{1}{20}\right)^{np/200}\right)^t \nonumber \\
&\leq \left(e^{2np/10^6}\left(\frac{1}{20}\right)^{np/200}\right)^t \nonumber \\
&\leq \left(\frac{1}{20}\right)^{npt/400}\leq \left(\frac{1}{20}\right)^{2\log n}=o(n^{-3}).\label{lastone3}
\end{align}
Therefore, by (\ref{lastone2}) and (\ref{lastone3}), the probability such a pair of sets~$A$ and $B$ exists is at most $\sum_{t=1}^mp_t=o(n^{-2})$.
\oof

\subsection{Expansion from random matchings}\label{secfinalmatch}
We will use auxilliary graphs with certain expansion properties, yet small maximum degree. We can find such graphs by using random matchings, as follows.
\lem\label{randommatchingsexpand} Let $H$ be a random bipartite graph with classes $X$ and $Y$, $|X|=|Y|=n$, formed by taking the union of 25 independent random matchings between $X$ and $Y$. With probability $1-o(n^{-2})$, the following properties hold.
\begin{itemize}\refstepcounter{equation}\label{duh1}
\item If $A\subset X$ and $|A|\leq n/4$, then $|N(A)|\geq 2|A|$.\hfill(\arabic{equation})\refstepcounter{equation}\label{duh2}
\item If $B\subset Y$ and $|B|\leq n/4$, then $|N(B)|\geq 2|B|$.\hfill(\arabic{equation})\refstepcounter{equation}\label{duh3}
\item If $A\subset X$, $B\subset Y$ and $|A|,|B|\geq n/5$, then there is some edge between~$A$ and $B$.\hfill(\arabic{equation})
\end{itemize}
\ma
\ifdraft
\else
\pr
Given two sets $A\subset X$ and $B\subset Y$, and a random matching $M$ between $X$ and $Y$, the probability that $N_H(A)\subset B$ is
\begin{equation}\label{binoms}
\binom{|B|}{|A|}\binom{n}{|A|}^{-1}=\frac{|B|(|B|-1)\cdots(|B|-|A|+1)}{n(n-1)\cdots(n-|A|+1)}\leq \left(\frac{|B|}{n}\right)^{|A|}.
\end{equation}
Thus, for each $t\leq n/4$, the probability, $p_t$ say, that there is some set $A\subset X$ with $|A|=t$ and $|N_H(A)|<2t$ satisfies
\begin{equation}\label{EQNCCC}
p_t\leq \binom{n}{t}\binom{n}{2t}\left(\frac{2t}{n}\right)^{25t}
\leq \left(\frac{e^3}{4}\left(\frac nt\right)^{3}\left(\frac{2t}{n}\right)^{25}\right)^t
= \left(2e^3\left(\frac{2t}{n}\right)^{22}\right)^t.
\end{equation}
If $t\leq n^{1/2}$, then, by (\ref{EQNCCC}), $p_t\leq n^{-4}$. If $n^{1/2}\leq t\leq n/4$, then, for sufficiently large~$n$, by (\ref{EQNCCC}), $p_t\leq (2e^3(1/2)^{22})^t\leq (1/2)^t\leq n^{-4}$. Therefore, the probability that there is some set $A\subset X$ with $|A|\leq n/4$ and $|N_H(A)|<2|A|$ is at most $\sum_{t\leq n/4}p_t=o(n^{-2})$. A similar calculation shows that, with probability $1-o(n^{-2})$, $|N_H(B)|\geq 2|B|$ for every $B\subset Y$ with $|B|\leq n/4$.

Note that, if, for two sets $A\subset X$ and $B\subset Y$, we have $e(A,B)=0$, then $N_H(A)\subset Y\setminus B$.
The probability there are no two sets $A\subset X$ and $B\subset Y$, with $|A|=|B|=n/5$ and $e(A,B)=0$ is therefore, using (\ref{binoms}), at most
\[
\binom{n}{n/5}^2\left(\frac{4}{5}\right)^{25n/5}\leq (5e)^{2n/5}\left(\frac{4}{5}\right)^{5n}\leq \left(\frac{19}{20}\right)^{n/4}.
\]
Therefore, with probability $1-o(n^{-2})$, any two disjoint subsets $A\subset X$ and $B\subset Y$, each with size at least $n/5$, have some edge between them.
\oof
\fi

We will use Lemma~\ref{randommatchingsexpand} through the following lemma.
\lem\label{lotsrandmatch} Let $h\in \N$ with $3|h$, and let $m=2h/3$. Let $X$, $Y$ and $Z$ be disjoint vertex sets with $|X|=h$ and $|Y|=|Z|=m$. Let $X_1,X_2\subset X$ satisfy $|X_1|=|X_2|=m$ and $X=X_1\cup X_2$. Independently, take 25 random matchings between each of the pairs of sets $(X_1,Y)$, $(X_2,Y)$, $(X_1,Z)$, and $(X_2,Z)$, and let the graph $H$ be the union of these matchings. Then, with probablity $1-o(h^{-2})$, for every set $Z'\subset Z$ with $|Z'|=h/3$ there is a matching between $X$ and $Y\cup Z'$.
\ma
\pr By Lemma~\ref{randommatchingsexpand}, with probability $1-o(h^{-2})$, the edges in $H$ between each of the pairs $(X_1,Y)$, $(X_2,Y)$, $(X_1,Z)$, and $(X_2,Z)$ satisfy the corresponding properties to (\ref{duh1})-(\ref{duh3}). We will now show that~$H$ has the property we require.
Let $Z'\subset Z$ be any set with $|Z'|=h/3$. We will verify the relevant conditions in the graph between~$X$ and $Y\cup Z'$ to show a matching exists between these two sets using Lemma~\ref{matchings}.

Let $A\subset X$, and take $A'$ to be a subset of either $A\cap X_1$ or $A\cap X_2$, with $|A'|=\lceil |A|/2\rceil$. Let $B\subset Y\cup Z'$, and take $B'$ to be a subset of either $B\cap Y$ or $B\cap Z'$, with $|B'|=\lceil |B|/2\rceil$.

If $|A|\leq 4m/9$, then $|A'|\leq m/4$, so $|N_H(A,Y\cup Z')|\geq|N_H(A',Y)|\geq 2|A'|\geq |A|$. If $|B|\leq  4m/9$, then $|B'|\leq m/4$, so $|N_H(B,X)|\geq|N_H(B',X_1)|\geq 2|B'|\geq |B|$.
If $|A|\geq  4m/9$ and $|B|\geq 4m/9$, then $|A'|\geq m/5$ and $|B'|\geq m/5$, so $e_H(A,B)\geq e_H(A',B')>0$. Therefore, by Lemma~\ref{matchings} there is a matching between $X$ and $Y\cup Z'$.
\oof

\section{Covering vertex sets while embedding trees}\label{5embedparents}

In this section, we develop tools to embed a tree~$T$ into a graph~$G$ so that a chosen vertex subset~$X$ in the graph is covered by a chosen vertex subset~$Q$ in the tree, where $T$ has distinctly fewer vertices than~$G$. We will use a variation on a result by Glebov, Johannsen and Krivelevich~\cite[Theorem 5.1]{GJK14}, which we adapt to our particular circumstance while giving a new, more efficient, proof.

Before proving the main result of this section (Lemma~\ref{embedparentsfinal}), we will first give an approximate version (Lemma~\ref{coverlots}), where we can find a copy of a tree~$T$ in which the copy of~$Q$ \emph{mostly} covers~$X$. Our working subgraph is the graph $I(X)$ (the edgeless graph with vertex set $X$) along with a copy of the subtree of~$T$ that we are building. We start building our copy of~$T$, and extend the working subgraph by repeatedly adding a path with a copy of a vertex in~$Q$ at the end (in an order determined by Lemma~\ref{order}), all while maintaining a $(d,m)$-extendability property. If possible, we always let an uncovered vertex from~$X$ be the copy of the vertex from~$Q$.  Where this is not possible, we find a copy of the path with only one endvertex in our working graph, using Corollary~\ref{treebuild}. We then analyse this process and show that most of the vertices in~$X$ are covered.

In this analysis, however, we wish to remove some of the paths we added, while retaining extendability. A path with both endvertices in the working subgraph cannot be removed by Lemma~\ref{extendreverse} -- we may only remove leaves while remaining extendable, and thus need a small change to the above outline. Instead of adding a path $P$ covering a new vertex in $X$, we find an edge $e$ of $P$ so that we can add the two disjoint paths in $P-e$ to the subgraph and remain extendable.
This allows us to remove the path $P-e$ in our analysis, and while our working subgraph is not quite a copy of a subtree of~$T$, the missing edges do exist in the parent graph~$G$.

The statement, and proof, of Lemma~\ref{coverlots} is slightly more complicated again than suggested above. The copy of the tree~$T$ is found appropriately attached to a $(d,m)$-extendable subgraph of $G$, so that we may use this embedding iteratively to prove Lemma~\ref{embedparentsfinal}.

\lem\label{coverlots}
Let $k,d,m\in \N$ with $d\geq 20$. Let~$G$ be an $m$-joined graph and let the subgraph $S\subset G$ satisfy $\Delta(S)\leq d/4$. Suppose $X\subset V(G)\setminus V(S)$ is such that $S\cup I(X)$ is $(d,m)$-extendable in~$G$.

Let~$T$ be a tree, with $\Delta(T)\leq d/4$, which satisfies $|S|+|X|+|T|\leq |G|-10dm-2k$. Suppose $Q\subset V(T)$ is a $(4k+4)$-separated set in~$T$ which satisfies $|Q|\geq 3|X|$.

Let $t\in V(T)$ and $s\in V(S)$. Then, there is a copy, $T'$ say, of~$T$ in $G-(V(S)\setminus\{s\})$ so that~$t$ is copied to $s$, $S\cup I(X)\cup T'$ is $(d,m)$-extendable in~$G$, $|X\setminus V(T')|\leq 2m/(d-1)^k$, and all the vertices in $X\cap V(T')$ have a vertex from~$Q$ copied to them.
\ma
\pr Let $q=|Q|$. Using Lemma~\ref{order}, take subtrees $T_1,\ldots,T_q$ of~$T$ so that $t\in V(T_1)$ and, for each~$i$, $2\leq i\leq q$, $T_{i}$ is formed from $T_{i-1}$ by adding a bare path with length $l_i\geq 2k+2$ with endvertices $w_i\in V(T_{i-1})$ and $q_i\in Q$.

In the applications of Corollary~\ref{treebuild} below, our working subgraph will be the tree $S\cup I(X)$ combined with some copy of a subgraph of~$T$. Therefore the working subgraph will always have maximum degree at most $\Delta(T)+\Delta(S)\leq d/2$, and at each application the subgraph we are aiming to achieve will contain at most $|S|+|X|+|T|\leq |G|-10dm$ vertices. If the working subgraph is $(d,m)$-extendable, then the conditions to apply Corollary~\ref{treebuild} will therefore be satisfied, which will allow us to attach a further tree to our working subgraph while remaining $(d,m)$-extendable.

Using Corollary~\ref{treebuild}, then, find a subtree $S_1$ in $G-(X\cup (V(S)\setminus\{s\}))$ which is a copy of $T_1$ so that~$t$ is copied to $s$ and $S\cup I(X)\cup S_1$ is $(d,m)$-extendable in~$G$.

\stepcounter{capitalcounter}
\setcounter{capitalcounterbackup}{\value{capitalcounter}}

Let $I_1=\emptyset$ and $X_1=X$. We say we have a \emph{stage $i$ situation}, where $1\leq i\leq q$, if we have a subgraph $S_{i}$, sets $I_{i}\subset [q]$ and $X_{i}\subset X$, and edges $e_{j}\in E(S_{i})$, $j\in I_{i}$, where $S_{i}$ is a copy of~$T_{i}$ with~$t$ copied to $s$, $(S\cup I(X_{i})\cup S_{i})-(\cup_{j\in I_{i}}e_{j})$ is $(d,m)$-extendable, $X_{i}$ is disjoint from $S_{i}$ and $X\setminus X_{i}$ is a subset of the copy of~$Q$ in $S_{i}$. Note that we currently have a stage 1 situation (see Figure~\ref{spanpic1}).

For each $i$, $2\leq i\leq q$, carry out the following Step~{\bfseries \Alph{capitalcounterbackup}$_i$}, which takes a stage $i-1$ situation and produces a stage $i$ situation. Starting with our stage 1 situation, we will therefore reach a stage $q$ situation. This process is depicted in Figure~\ref{spanpic1}.

\begin{itemize}[label =\bfseries \Alph{capitalcounterbackup}$_i$ ]
\item Let $v_i$ be the copy of $w_i$ in $S_{i-1}$. Create a stage $i$ situation as follows.
\begin{enumerate}[label =(\roman{enumi})]
\item Suppose we can find a vertex $x_i\in X_{i-1}$, and a $v_i,x_i$-path $P_{i}$, with length $l_i$ and interior vertices in $V(G)\setminus (V(S)\cup X_{i-1}\cup V(S_{i-1}))$, and an edge $e_i\in E(P_i)$, so that $(S\cup I(X_{i-1})\cup S_{i-1})-(\cup_{j\in I_{i-1}}e_{j})+P_i-e_i$ is $(d,m)$-extendable. Then, let $S_i=S_{i-1}+P_i$, $I_{i}=I_{i-1}\cup\{i\}$ and $X_i=X_{i-1}\setminus \{x_{i}\}$, so that $S_i$ is a copy of $T_i$ in which $q_i$ has been copied to $x_i$. In this case, we call Step~{\bfseries \Alph{capitalcounterbackup}$_i$} a \emph{good} step.
\item If no such path exists, then use Corollary~\ref{treebuild} to attach a path $P_i$, with length $l_i$ and vertices in $(V(G)\setminus (V(S)\cup X_{i-1}\cup V(S_{i-1})))\cup\{v_i\}$ to $v_i$, so that $(S\cup I(X_{i-1})\cup S_{i-1})-(\cup_{j\in I_{i-1}}e_{j})+P_i$ is $(d,m)$-extendable. Let $S_i=S_{i-1}+P_i$, $X_i=X_{i-1}$, and $I_{i}=I_{i-1}$, so that $S_i$ is a copy of $T_i$. In this case, we call Step~{\bfseries \Alph{capitalcounterbackup}$_i$} a \emph{neutral} step.
\end{enumerate}
\end{itemize}

Having reached a stage $q$ situation, we will now show that the resulting tree $S_q$ must contain most of the vertices in $X$. The methods in this proof are depicted in Figure~\ref{spanpic2}.

\begin{figure}
\centering
    \begin{subfigure}[b]{0.5\textwidth}
        \centering
        \resizebox{\linewidth}{!}{
            \input{spanpic1A}
        }
\label{spanpicA}
    \end{subfigure}
\hspace{0.25cm}
\parbox[b]{.4\linewidth}{%
\subcaption{The initial starting stage 1 situation. The tree $S$ is represented as a triangle, and the subgraph $S\cup I(X)$ is $(d,m)$-extendable. To simplify the picture, we will assume that $S_1=I(\{s\})$.
\\
\textcolor{white}{.}}}

    \begin{subfigure}[b]{0.5\textwidth}
        \centering
        \resizebox{\linewidth}{!}{
            \input{spanpic1B}
        }
\label{spanpicB}
    \end{subfigure}
\hspace{0.25cm}
\parbox[b]{.4\linewidth}{%
\subcaption{After Step~{\bfseries \Alph{capitalcounterbackup}$_2$}. We have found the $v_2,x_2$-path~$P_2$ and the edge $e_2\in E(P_2)$, so that $S+P_2-e_2$ is $(d,m)$-extendable and $P_2$ embeds the vertex $q_2$ onto $x_2$. Step~{\bfseries \Alph{capitalcounterbackup}$_2$} is therefore a good step.\\
\textcolor{white}{.}}}

    \begin{subfigure}[b]{0.5\textwidth}
        \centering
        \resizebox{\linewidth}{!}{
            \input{spanpic1C}
        }
\label{spanpicC}
    \end{subfigure}
\hspace{0.25cm}
\parbox[b]{.4\linewidth}{%
\subcaption{After Step~{\bfseries \Alph{capitalcounterbackup}$_3$}. There was no path~$P_3$ and edge $e_3$ satisfying the conditions for Step~{\bfseries \Alph{capitalcounterbackup}$_3$} to be a good step. Therefore, we have found instead the path~$P_3$, so that $(S+P_2-e_2)+P_3$ is $(d,m)$-extendable. Step~{\bfseries \Alph{capitalcounterbackup}$_3$} is therefore a neutral step.\\
\textcolor{white}{.}}}

    \begin{subfigure}[b]{0.5\textwidth}
        \centering
        \resizebox{\linewidth}{!}{
            \input{spanpic1D}
        }
\label{spanpicD}
    \end{subfigure}
\hspace{0.25cm}
\parbox[b]{.4\linewidth}{%
\subcaption{After the final step, Step~{\bfseries \Alph{capitalcounterbackup}$_q$}, has been completed, where here $q=8$.
We have had in total four good steps (Steps~{\bfseries \Alph{capitalcounterbackup}$_2$},~{\bfseries \Alph{capitalcounterbackup}$_4$},~{\bfseries \Alph{capitalcounterbackup}$_6$} and~{\bfseries \Alph{capitalcounterbackup}$_7$}) and three neutral steps (Steps~{\bfseries \Alph{capitalcounterbackup}$_3$},~{\bfseries \Alph{capitalcounterbackup}$_5$} and~{\bfseries \Alph{capitalcounterbackup}$_8$}).\\
\textcolor{white}{.}}}

\caption{Developing a stage $q$ situation in (d) from a stage 1 situation in (a) by completing the Steps~{\bfseries \Alph{capitalcounterbackup}$_i$}, $2\leq i\leq q$, where here $q=8$.} \label{spanpic1}
\end{figure}

\begin{figure}
\centering
    \begin{subfigure}[b]{0.6\textwidth}
        \centering
        \resizebox{\linewidth}{!}{
            \input{spanpic2A}
        }
\label{spanpic2A}
    \end{subfigure}\parbox[b]{.4\linewidth}{%
\subcaption{The stage $q$ situation from Figure~\ref{spanpic1}, where $q=8$. The vertices in $X_q\subset X$ are grey, as are the vertices~$z_i$, $i\in I$. Corollary~\ref{extendableconnectcor} is then applied to the sets $\{z_i:i\in I\}$ and $X_q$, under the assumption that $X_q$ has at least a certain size, to reach the situation in (b).
\\
\textcolor{white}{.}}}

    \begin{subfigure}[b]{0.6\textwidth}
        \centering
        \resizebox{\linewidth}{!}{
            \input{spanpic2B}
        }
\label{spanpic2B}
    \end{subfigure}\parbox[b]{.4\linewidth}{%
\subcaption{Corollary~\ref{extendableconnectcor} has been used to find $j=5\in I$, $x\in X_q$, a $z_j,x$-path $R$ and an edge $e\in E(R)$ so that the working subgraph $(S\cup I(X_q)\cup S_q)-(\cup_{i\in I_{q}}e_{i})+R-e$ is $(d,m)$-extendable.  \\
\textcolor{white}{.}}}

    \begin{subfigure}[b]{0.6\textwidth}
        \centering
        \resizebox{\linewidth}{!}{
            \input{spanpic2C}
        }
\label{spanpic2C}
    \end{subfigure}\parbox[b]{.4\linewidth}{%
\subcaption{Lemma~\ref{extendreverse} has been used to remove $P_8$ from the working subgraph. The edge~$e_7$ is not part of the working subgraph, so that we may remove the vertices in $V(P_7)$ from the working subgraph by using Lemma~\ref{extendreverse} to remove the two disjoint paths comprising $P_7-e_7$.\\
\textcolor{white}{.}}}

    \begin{subfigure}[b]{0.6\textwidth}
        \centering
        \resizebox{\linewidth}{!}{
            \input{spanpic2D}
        }
\label{spanpic2D}
    \end{subfigure}\parbox[b]{.4\linewidth}{%
\subcaption{Lemma~\ref{extendreverse} has been used to remove all the vertices in $V(P_i)$, $i>j$, from the working subgraph, as well as the vertices on $P_j$ which are not in $P'_j$ (recalling $j=5$ here). As the depicted structure is $(d,m)$-extendable, the path $P_j'+R$ and the edge~$e$ demonstrates that Step~{\bfseries \Alph{capitalcounterbackup}$_j$} should have been a good step, giving a contradiction.\\
\textcolor{white}{.}}}

\caption{An illustration of the proof of Claim~\ref{claimy}.} \label{spanpic2}
\end{figure}

\stepcounter{capitalcounter}
\setcounter{claimcounter}{0}
\begin{claimalpha}\label{claimy}
We have $|X_q|< 2m/(d-1)^k$.
\end{claimalpha}
\pr[Proof of Claim~\ref{claimy}]
Suppose to the contrary that $|X_q|\geq 2m/(d-1)^k$, and let $I=\{2,\ldots,q\}\setminus I_q$, so that $|I|\geq q-1-|X| \geq 2|X|-1\geq |X_q|\geq 2m/(d-1)^k$. For each $i\in I$, let~$z_i$ be the vertex of~$P_i$ which is a distance $l_i-2k-1$ away from $v_i$. As $l_i-2k-1\geq 1$, for each $i\in I$, the vertices~$z_i$ are distinct, and there are thus at least $2m/(d-1)^k$ such vertices. Note that
\[
|(S\cup I(X_q)\cup S_q)-(\cup_{i\in I_{q}}e_{i})|\leq |S|+|X|+|T|\leq |G|-10dm-2k.
\]
Therefore, by Corollary~\ref{extendableconnectcor}, there is some index $j\in I$, and a vertex $x\in X_q$ for which there is a $z_j,x$-path $R$ in~$G$, with length $2k+1$ and interior vertices in $V(G)\setminus (V(S)\cup X \cup V(S_q))$, and an edge $e\in E(R)$ so that $(S\cup I(X_q)\cup S_q)-(\cup_{i\in I_{q}}e_{i})+R-e$ is $(d,m)$-extendable.

At each stage $i>j$ we added either the path $P_i$, or the two disjoint paths that comprise $P_i-e_i$, to some vertex, or  vertices, in $V(S)\cup X_{i-1}\cup V(S_{i-1})$. Therefore, by Lemma~\ref{extendreverse}, we can remove these paths from the subgraph $(S\cup I(X_q)\cup S_q)-(\cup_{i'\in I_{q}}e_{i'})+R-e$ one-by-one while maintaining $(d,m)$-extendability, working backwards through the Steps~{\bfseries \Alph{capitalcounterbackup}$_i$} with $i$ decreasing from $q$ to $j+1$. This demonstrates that $(S\cup I(X_j)\cup S_j)-(\cup_{i'\in I_{j}}e_{i'})+R-e$ is $(d,m)$-extendable.

Let $P_j'$ be the path containing exactly the initial $l_j-2k-1$ edges of $P_j$, starting from $v_j$. Using Lemma~\ref{extendreverse}, we can remove the rest of the path $P_j$, to show that $S\cup I(X_{j-1})\cup S_{j-1}-(\cup_{i\in I_{j-1}}e_{i})+P_j'+R-e$ is $(d,m)$-extendable. This contradicts the process above, as at Step~{\bfseries \Alph{capitalcounterbackup}$_j$} we should have found the $v_j,x$-path $P_j'+R$ and the edge $e$.
\oof

Therefore, we have the tree $S_q$, a copy of $T_q$ with $t$ copied to $s$, and the set $X_q\subset X$, so that the vertices of $X\setminus X_q$ are contained in the copy of~$Q$, $X_q$ is disjoint from $V(S_q)$, $S\cup I(X_q)\cup S_q$ is $(d,m)$-extendable, and, by Claim~\ref{claimy}, $|X_q|\leq 2m/(d-1)^k$. By Corollary~\ref{treebuild}, we can then extend~$S_q$ to a copy of~$T$ in $G-((V(S)\cup X_q)\setminus\{s\})$ to satisfy the lemma.
\oof

We will now divide our tree into a sequence of subtrees with decreasing size so that each subtree contains plenty of vertices from~$Q$. Taking increasingly well separated subsets of~$Q$ in these subtrees, we can use Lemma~\ref{coverlots} to embed each subtree to cover more and more of the set~$X$ until the whole set is covered.

\lem \label{embedparentsfinal} Let $\Delta\in \N$, and suppose that $n\in \N$ is sufficiently large. Suppose that $d,m\in \N$ satisfy $d\geq\log n/\log\log n$ and $m\leq n$. Let~$G$ be an $m$-joined graph containing the vertex set~$X$ and the vertex $v\in V(G)\setminus X$, so that $|X|\geq n/\log^2 n$ and $I(X\cup\{v\})$ is $(d,m)$-extendable in~$G$.

Let~$T$ be a tree, with $\Delta(T)\leq\Delta$ and $|T|\leq |G|-|X|-10dm-\log n$, which contains a vertex $t\in V(T)$ and a vertex set~$Q$ with $|Q|\geq 9|X|$, where~$Q$ is 16-separated in~$T$. Then,~$G$ contains a copy $S$ of~$T$, with~$t$ copied to $v$, so that $S$ is $(d,m)$-extendable and~$X$ is contained in the copy of~$Q$.
\ma

Note that in Lemma~\ref{embedparentsfinal} we do not require that $|G|=n$.
\pr Let $R_0=T$ and $t_0=t$. For each $i$, $1\leq i<\a:=2\log n/\log\log n$, using Proposition~\ref{divide}, divide~$R_{i-1}$ into two trees $T_i$ and $R_i$, so that $|Q\cap V(T_i)|,|Q\cap V(R_i)|\geq|Q\cap V(R_{i-1})|/3$ and $t_{i-1}\in V(T_{i})$, and let $t_i$ be the vertex shared by $T_i$ and $R_i$. Let $T_\a=R_{\a-1}$. Then, for each $i$, $|Q\cap V(T_i)|\geq |Q|/3^i$.

Let $Q_1=Q\cap V(T_1)$. For each $i$ with $2\leq i\leq \a$ we have $3^i\Delta^{2i+5}\leq 3^\a\Delta^{2\a+5}\leq n/3\log^2 n$, for sufficiently large~$n$. Therefore, as $|Q\cap V(T_i)|\geq n/3^i\log^2 n\geq 3\Delta^{2i+5}$, using Corollary~\ref{Qfinder2} with $k=2i+5$, we may let $Q_i$ be a $(4i+12)$-separated set of vertices in $Q\cap V(T_i)$ for which
\begin{equation}\label{Qbound}
|Q_i|\geq |Q|/3^i(16i+48)\Delta^{2i+5}.
\end{equation}

Note that $|Q_1|\geq |Q|/3\geq 3|X|$. As $Q$ is 16-separated, $Q_1$ is also 16-separated. Using Lemma~\ref{coverlots}, then, with $k=3$, find a copy $S_1$ of $T_1$ in which $t$ is copied to $v$, so that, letting $X_1$ be the set of vertices in~$X$ not covered by the copy of $Q\cap V(T_1)$ in $S_1$, we have that $X_1$ and $V(S_1)$ are disjoint, $S_1\cup I(X_1)$ is $(d,m)$-extendable, and $|X_1|\leq 2m/(d-1)^3$.

Now, for each $i$, $2\leq i\leq \a$, extend the working subgraph $S_{i-1}$ to get $S_i$, a copy of $\cup_{j=1}^iT_j$, so that as many of the remaining vertices in~$X$ as possible are covered by vertices in $Q_i$, and, letting~$X_i$ be the set of vertices in $X$ not in the copy of $Q\cap V(\cup_{j=1}^iT_j)$ in $S_i$,~$X_i$ is disjoint from $V(S_i)$ and $S_i\cup I(X_i)$ is $(d,m)$-extendable.

We will show, by induction, that, for each $i\in [\a]$, $|X_i|\leq 2m/(d-1)^{i+2}$. If this is true then, as $m\leq n$, for sufficiently large~$n$, $|X_\a|\leq 2m/(d-1)^\a<1$ and thus all of the vertices from~$X$ are covered by the copy of~$Q$ in $S_\a$, which is a copy of $\cup_iT_i=T$, as required. The induction statement is true for $i=1$, so suppose it is true for some $i$, $1\leq i<\a$. For sufficiently large~$n$, we have
\begin{equation}\label{nsufflarge}
n/100(3\Delta)^{9}\log^2 n\geq 6n/(d-1)^3\geq 6m/(d-1)^{3}.
\end{equation}
Then, by (\ref{Qbound}) and (\ref{nsufflarge}), as $|Q|\geq 9|X|\geq 9n/\log^2 n$, and by the induction statement for~$i$,
\[
|Q_{i+1}|\geq |Q|/(3^{i+1}(16i+64)\Delta^{2i+7})\geq n/100(3\Delta)^{2i+7}\log^2 n\geq 6m/(d-1)^{i+2}\geq 3|X_i|.
\]
Therefore, by Lemma~\ref{coverlots}, as $Q_{i+1}$ is $(4i+16)$-separated, $|X_{i+1}|\leq 2m/(d-1)^{i+3}$, as required.
\oof

\section{Case A}\label{7caseA}

We can now prove Case~A of Theorem~\ref{unithres}, where the trees have many leaves. We will remove leaves from such a tree, and find a copy of the remaining subtree. To complete the embedding, we need to attach the remaining uncovered vertices in the random graph onto the partial embedding, to replace the leaves we removed. We will attach these leaves using a matching, and to ensure its existence we embed the subtree so that the vertices that need leaves attached contain a particular set of vertices, which we will call a \emph{matchmaker set}.

\de We say that the set $X$ is a \emph{matchmaker set for $V$} in the graph~$G$ if, for every subset $U\subset V\setminus X$ with $|U|\leq |X|/8$, we have $|N(U,X)|\geq |U|$. 
\fn

We will reveal edges in our random graph with probability $\Theta(\log n/ n)$, and, as mentioned previously, take $m$ to be a small integer such that, almost surely, any two disjoint subsets in our random graph with size $m$ will have some edge between them (that is, the graph is $m$-joined). Using Proposition~\ref{generalprops}, we can take
\[
m=\Theta\left(\frac{n\log\log n}{\log n}\right).
\]
As defined in Section~\ref{casedivision}, trees in Case~A will have $\Omega(m)$ leaves that are 20-separated.

A matchmaker set $X$ for $V(G)$ in the random graph $\GG(n,\Theta(\log n/ n))$, can be found with $|X|=\Theta(m)$, using Proposition~\ref{goodset}. For Case A of Theorem~\ref{unithres}, we start by finding two disjoint matchmaker sets $X_1$ and $X_2$ for $V(G)$ with size $\Theta(m)$. We then take our tree $T\in\TT_A(n,\Delta)$, and divide it into two trees~$T_1$ and~$T_2$. The tree~$T_1$ will contain a 20-separated set of leaves, $L$, with size $\Theta(m)$. We remove the leaves in $L$ from~$T_1$ to get the tree $T_1'$. We then let $Q$ be the neighbours of the leaves in~$T_1$ which we removed, so that~$T_1$ is formed from $T_1'$ by adding a matching to $Q$. Next, we find a copy, $S'_1$, of~$T_1'$  in $G-X_2$ so that the image of $Q$ covers $X_1$ (using Lemma~\ref{embedparentsfinal}). We then use the extendability methods to attach an appropriate copy of~$T_2$, using more vertices in $V(G)\setminus X_2$. We will have $|L|-|X_2|=\Theta(m)$ spare vertices, so that, with well chosen constants, we can do this using Corollary~\ref{treebuild}. Finally, we use the properties of the matchmaker sets to find a matching from the image of $Q$, which contains $X_1$, into the set of uncovered vertices in the graph, which contains $X_2$, and thus complete the copy of~$T$.

\pr[Proof of Theorem~\ref{unithres} in Case A] 
\ifdraft
\else
We will reveal the edges within the vertex set $[n]$ in two rounds, each with probability $p=10^{12}\Delta\log n/2n$, and find certain properties in the resulting graph. In total therefore, each edge will have been revealed with probability at most $10^{12}\Delta\log n/ n$. Once the final edges are revealed and the graph~$G$ is fixed, we then take any tree $T\in\TT_A(n,\Delta)$ and embed it into~$G$. 

Let $m=10\log(np)/p$. For large $n$, $\log(np)=\log(10^{12}\Delta\log n/2)\leq 2\log\log n$, so that 
\[
m\leq n\log\log n/10^9\Delta\log n.
\]
Divide the vertex set $[n]$ into $V_1$ and $V_2$ as equally as possible. Reveal edges with probability $p$ within the vertex set $[n]$ to get the graph $G_1$. By Proposition~\ref{goodset}, there is, for each $i\in[2]$, a set $X_i\subset V_i$, with $|X_i|= 8m$ so that, if $U\subset [n]\setminus X_i$ and $|U|\leq m$, then $|N(U,X_i)|\geq |U|$. Thus, the sets~$X_1$ and $X_2$ are matchmaker sets for $V(G)$.

Let $d=n/100m$ and note that, as
\[
m=\Theta\left(\frac{n\log \log n}{\log n}\right),\;\;\;\text{ we have }\;\;\; d=\Theta\left(\frac{\log n}{\log\log n}\right).
\]
Reveal edges with probability $p$ to get the graph $G_2$ and let $G=G_1\cup G_2$. By Corollary~\ref{generalexpandcor}, applied with $X_1$ and $V(G)\setminus (X_1\cup X_2)$,
the subgraph $I(X_1)$ is almost surely $(d,m)$-extendable in $G-X_2$. By Proposition~\ref{generalprops},~$G$ is almost surely $m$-joined. This structure, and the following construction, is depicted in Figure~\ref{spanpic3}.

\begin{figure}
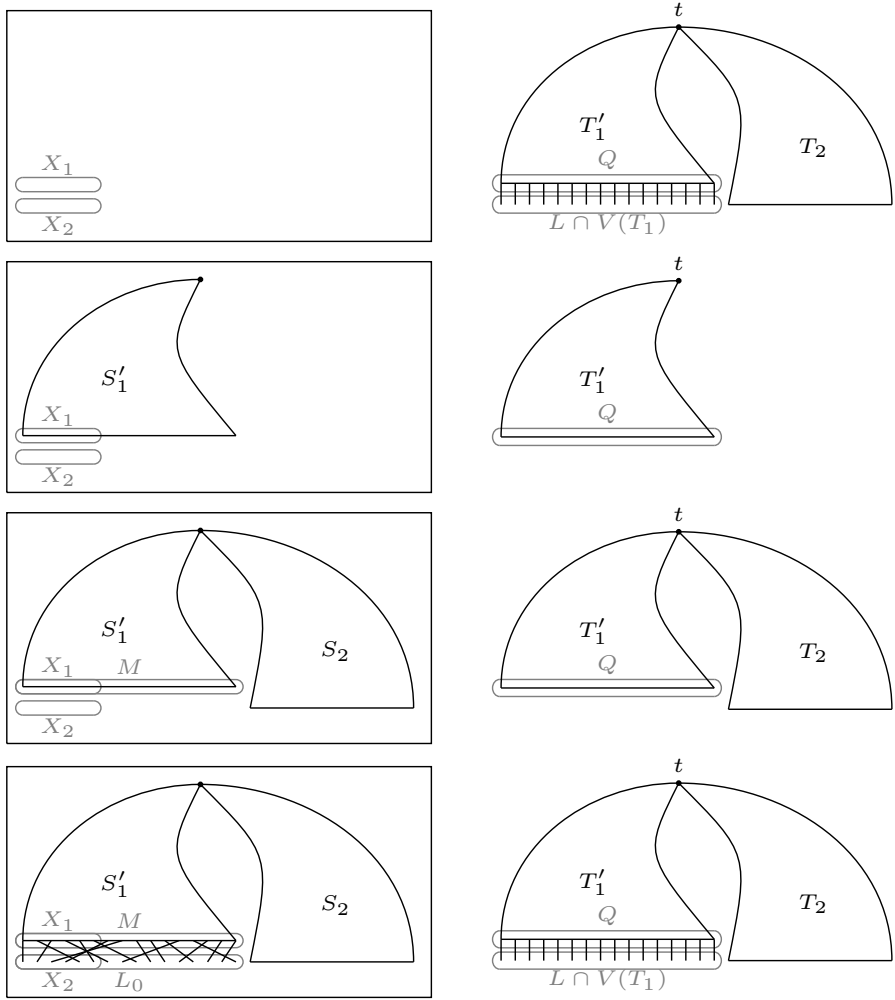

\centering
    \begin{subfigure}[b]{0.4\textwidth}
        \centering
        \resizebox{\linewidth}{!}{
            \input{spanpic3A1}
        }
\label{spanpic3A1}
    \end{subfigure}
\hspace{-0.25cm}
    \begin{subfigure}[b]{0.4\textwidth}
        \centering
        \resizebox{\linewidth}{!}{
            \input{spanpic3A2}
        }
\label{spanpic3A2}
    \end{subfigure}

\vspace{-0.4cm}


    \begin{subfigure}[b]{0.4\textwidth}
        \centering
        \resizebox{\linewidth}{!}{
            \input{spanpic3B1}
        }
\label{spanpic3B1}
    \end{subfigure}
\hspace{-0.25cm}
    \begin{subfigure}[b]{0.4\textwidth}
        \centering
        \resizebox{\linewidth}{!}{
            \input{spanpic3B2}
        }
\label{spanpic3B2}
    \end{subfigure}

\vspace{-0.4cm}


    \begin{subfigure}[b]{0.4\textwidth}
        \centering
        \resizebox{\linewidth}{!}{
            \input{spanpic3C1}
        }
\label{spanpic3C1}
    \end{subfigure}
\hspace{-0.25cm}
    \begin{subfigure}[b]{0.4\textwidth}
        \centering
        \resizebox{\linewidth}{!}{
            \input{spanpic3C2}
        }
\label{spanpic3C2}
    \end{subfigure}

\vspace{-0.4cm}

    \begin{subfigure}[b]{0.4\textwidth}
        \centering
        \resizebox{\linewidth}{!}{
            \input{spanpic3D1}
        }
\label{spanpic3D1}
    \end{subfigure}
\hspace{-0.25cm}
    \begin{subfigure}[b]{0.4\textwidth}
        \centering
        \resizebox{\linewidth}{!}{
            \input{spanpic3D2}
        }
\label{spanpic3D2}
    \end{subfigure}

\vspace{-0.4cm}



\caption{Embedding the tree in the proof of Theorem~\ref{unithres} in Case~A, depicted from top to bottom on the left, and starting with the sets found in the random graph $G$. On the right, from top to bottom, is the structure found in the tree $T$, followed by the subgraphs of $T$ embedded in the matching pictures on the left.} \label{spanpic3}
\end{figure}

Let $T$ be any tree in $\TT_A(n,\Delta)$. Let $L\subset V(T)$ be a set of 20-separated leaves of~$T$ with $|L|\geq n\log\log n/10^5\log n\geq 10^3\Delta m$. Using Proposition~\ref{divide}, find subtrees~$T_1$ and~$T_2$ which intersect only on one vertex,~$t$ say, cover~$T$ and which each contain at least $n/3$ vertices. Suppose, without loss of generality, that $|L\cap V(T_1)|\geq |L|/2\geq 500\Delta m$. This structure found in $T$ is also depicted in Figure~\ref{spanpic3}.

Let $Q=N_T(L)$, and let $T_1'=T_1-L$. Note that, if $n\geq 6$, then $t$ has a neighbour in both $T_1$ and $T_2$ and hence~$t$ is not a leaf of $T$. Thus, $t\in V(T_1')$. We have that 
\[
|T_1'|\leq |T|-|T_2|+1\leq 2n/3+1\leq (n-|X_2|)-|X_1|-10dm-\log n,
\]
$|Q\cap V(T_1')|\geq |L\cap V(T_1)|\geq 500\Delta m\geq 9|X_1|$ and $|X_1|=8m\geq n/\log^2 n$. By Lemma~\ref{embedparentsfinal}, then, there is a $(d,m)$-extendable copy of $T'_1$, $S'_1$ say, in $G-X_2$ so that $X_1$ is contained in the copy of $Q\cap V(T'_1)$. Let $M$ be the copy of $Q\cap V(T'_1)$ in $S'_1$.

The subgraph $S'_1$ is $(d,m)$-extendable in $G-X_2$, and thus, as it has maximum degree at most~$\Delta$, it is also $(\Delta,m)$-extendable. As $n=|T_1'|+|T_2|+|L\cap V(T_1)|-1$ and $|L\cap V(T_1)|\geq |X_2|+10\Delta m$, we have $n-|X_2|\geq |T-(L\cap V(T_1))|+10\Delta m$. Therefore, as $G$ is $m$-joined, using Corollary~\ref{treebuild}, we can extend the subgraph $S_1'$ by adding a copy of~$T_2$, $S_2$ say, in $G-X_2$.

We only need then to find a matching from $M$ to the vertices not in $V(S_1'\cup S_2)$ to complete the copy of~$T$. Let $L_0=V(G)\setminus V(S_1'\cup S_2)$.
By construction, $X_1\subset M$ and $X_2\subset L_0$. If $U\subset M$ and $|U|\leq m$, then $|N(U,L_0)|\geq|N(U,X_2)|\geq |U|$. If $U\subset L_0$ and $|U|\leq m$, then $|N(U,M)|\geq|N(U,X_1)|\geq |U|$. Therefore, as $G$ is $m$-joined, by Lemma~\ref{matchings}, there is a matching between $M$ and $L_0$, and thus we can complete the copy of $T$.
\fi
\oof

\section{Cycles in subgraphs}\label{3hamcycles}

As mentioned in Section~\ref{1boutline}, for Case B we need to embed a tree so that it covers a specific large subset of a graph. In essence, we cover a vertex set by finding a cycle covering that vertex set, breaking it into paths, and incorporating the paths into the embedding of the tree. In this section, we prove Lemma~\ref{hamcycleinsubgraph}, which gives a rule sufficient to guarantee certain linear-sized vertex subsets of a typical random graph $G(n,\Omega(1/n))$ are Hamiltonian.

To find Hamilton cycles in random graphs, P\'osa~\cite{posa76} introduced the celebrated rotation-extension technique, which has subsequently been instrumental in proving many results involving Hamilton cycles. The following lemma, using some brief definitions, can be easily deduced by any standard use of P\'osa's rotation-extension lemma (see, for example, Bollob\'as~\cite{bollo1}), and specifically can be taken directly from a lemma used by Krivelevich, Lubetzky and Sudakov~\cite[Lemma~2.6]{BKL14}.

\de
In a graph~$H$, we say a non-edge $e$ is a \emph{booster} if $H+e$ is Hamiltonian.
\fn
\de
Given a vertex set $U$ in a graph~$H$, we say that it has \emph{many boosters in~$H$} if there are at least $|U|^2/32$ boosters in $H[U]$.
\fn
\de
Given a graph~$H$ with $n$ vertices, we say~$H$ is an $(n,2)$-expander if, for every subset $A\subset V(H)$ with $|A|\leq n/4$, we have $|N(A)|\geq 2|A|$.
\fn

\lem\label{posa} Suppose~$H$, a graph with $n$ vertices, is an $(n,2)$-expander, and $P$ is a path of maximum length in~$H$. Then either $H[V(P)]$ is Hamiltonian, or $|P|\geq n/4$ and $V(P)$ has many boosters in~$H$.\hfill\qed
\ma

We will use this lemma through Corollary~\ref{posacor}, which we prove using the following simple proposition.

\begin{prop}\label{expconnect}
Let $n\geq 8$. If a graph~$H$, with $n$ vertices, is an $(n,2)$-expander, then it is connected.
\end{prop}
\pr Suppose to the contrary that there is some set $A\subset V(H)$ with $0<|A|\leq \lfloor n/2\rfloor $ and $N(A)=\emptyset$. If $|A|\leq n/4$, then, as~$H$ is an $(n,2)$-expander, $|N(A)|\geq 2|A|>0$, a contradiction. Therefore, $|A|>n/4$, and we can take a set $A'\subset A$ with size $\lfloor n/4 \rfloor$. Then,
\[
|N(A)|\geq |N(A')|-|A\setminus A'|\geq 2\lfloor n/4 \rfloor-(\lfloor n/2\rfloor-\lfloor n/4 \rfloor)= 3\lfloor n/4 \rfloor-\lfloor n/2 \rfloor>0,
\]
a contradiction. Thus,~$H$ is connected.
\oof

\begin{corollary}\label{posacor}
Let $n\geq 8$. Suppose~$H$, a graph with $n$ vertices, is an $(n,2)$-expander, and $P$ is a path of maximum length in~$H$. Then, either~$H$ is Hamiltonian, or $|P|\geq n/4$ and $V(P)$ has many boosters in~$H$.
\end{corollary}
\pr By Lemma~\ref{posa}, either $|P|\geq n/4$ and $V(P)$ has many boosters in~$H$, or $H[V(P)]$ is Hamiltonian. If we have the second case, then let $C$ be a cycle with the vertex set $V(P)$. If~$H$ is not Hamiltonian, then $|C|<n$. By Proposition~\ref{expconnect},~$H$ is connected. Therefore, there is some vertex outside of the cycle $C$ which has a neighbour in the cycle. Using this vertex we can find a path with $|C|+1=|P|+1$ vertices, contradicting that $P$ has maximum length. 
\oof

For the following lemma, we use a method of Sudakov and Lee~\cite[Lemma 3.5]{LS12} for showing that $\GG(n,p)$ is resiliently Hamiltonian if $p=\omega(\log n/n)$, adapting the technique for use in sparser random graphs.

\lem\label{booster}
Suppose $np\geq 10^{15}$. With probability $1-o(n^{-2})$, the following holds in the random graph $G=\GG(n,p)$. Given any set $U\subset V(G)$ with $|U|\geq n/10^{6}$, if~$H$ is a subgraph of~$G$ with $U\subset V(H)$ and $e(H)\leq  pn^2/10^{16}$ so that $U$ has many boosters in~$H$, then at least one of these boosters is an edge of $G$.
\ma
\ifdraft
\else
\pr Let $\d=10^{-16}$.
Let $\mathcal{U}$ be the set of pairs $(U,H)$ consisting of a vertex set $U\subset[n]$ and a graph~$H$ on the vertex set $[n]$, where $|U|\geq n/10^6$, $e(H)\leq \d n^2 p$ and $U$ has many boosters in~$H$. Let $B_H(U)$ be the set of boosters in $H[U]$ for each pair $(U,H)\in \mathcal{U}$, so that $|B_H(U)|\geq |U|^2/32\geq 2n^2/10^{14}$.

When the assertion in the lemma fails for the random graph $G=\GG(n,p)$, there must be some pair $(U,H)\in\mathcal{U}$ with $H\subset G$ and $B_H(U)\cap E(G)=\emptyset$. Let the probability that such a pair exists be $q$. Then
\[
q \leq \sum_{(U,H)\in\mathcal{U}}\P((H\subset G)\wedge (B_H(U)\cap E(G)=\emptyset))
=\sum_{(U,H)\in\mathcal{U}}\P(B_H(U)\cap E(G)=\emptyset|H\subset G)\P(H\subset G).
\]
Now, as $B_H(U)$ contains no edges in~$H$,
\[
\P(B_H(U)\cap E(G)=\emptyset|H\subset G)=\P(B_H(U)\cap E(G)=\emptyset)=(1-p)^{|B_H(U)|}\leq e^{-pn^2/10^{14}}.
\]
Therefore, as for each graph~$H$ there are certainly at most $2^n$ sets $U$ for which $(U,H)\in\mathcal{U}$,
\begin{align*}
q&\leq 2^ne^{-pn^2/10^{14}}\sum_{H\subset K_n,e(H)\leq \d pn^2}\P(H\subset G) \leq 2^ne^{-pn^2/10^{14}} \sum_{t=0}^{\d  pn^2}\binom{n^2}{t}p^{t} \\
&\leq 2^ne^{-pn^2/10^{14}}\sum_{t=0}^{\d  pn^2}\left(\frac{en^2p}{t}\right)^t \leq 2^ne^{-pn^2/10^{14}} n^2\left(\frac{e}{\d}\right)^{\d n^2p} \\
&\leq\exp\left(n-pn^2/10^{14}+\d\log(e/\d)n^2p\right) \leq \exp\left(-pn^2/10^{15}\right).
\end{align*}
As $pn\geq 10^{15}$, the probability $G$ does not have the property in the lemma is $o(n^{-2})$.
\oof
\fi

Putting Corollary~\ref{posacor} and Lemma~\ref{booster} together yields a condition on subsets $U$ of a random graph that guarantees the subset supports a cycle. This condition is, roughly speaking, that we can select very few edges of $G$ from those within $U$ so that the subgraph~$H$ formed by these edges is a $(|U|,2)$-expander.

\lem\label{hamcycleinsubgraph}
If $pn\geq 10^{17}$, then, with probability $1-o(n^{-2})$, the random graph $G=\GG(n,p)$ satisfies the following property for each subset $U\subset V(G)$ with $|U|\geq 4n/10^6$. If there is some subgraph $H\subset G$, with $V(H)=U$ and $e(H)\leq pn^2/10^{17}$, which is a $(|U|,2)$-expander, then the graph $G[U]$ is Hamiltonian.
\ma
\ifdraft
\else
\pr
With probability $1-o(n^{-2})$, $G=\GG(n,p)$ has the property from Lemma~\ref{booster}. Let $U\subset V(G)$ satisfy $|U|\geq 4n/10^6$, and suppose $H\subset G$ is a $(|U|,2)$-expander which satisfies $V(H)=U$ and $e(H)\leq pn^2/10^{17}$.

Let $\mu(K)$ denote the number of vertices in the longest path in a graph $K$, and let $H_0=H$. Let $l$ be the largest integer for which there exists some graph $H_l\subset G$ with $V(H_l)=U$, $H\subset H_l$, $e(H_l)=e(H)+l$ and $\mu(H_{l})\geq\mu(H)+l$. Such an $l$ exists as~$H$ satisfies the conditions with $l=0$, and, as $\mu(H_l)\leq |H|$, $l<n$. Fix such a graph $H_l$ for $l$.

We will show that in fact $\mu(H_l)=|H|$. Suppose, to the contrary, that $\mu(H_l)<|H|$. As it contains~$H$ and has the vertex set $U$, $H_l$ is a $(|U|,2)$-expander. Let~$P$ be a longest path in $H_l$. As $\mu(H_l)<|H|$, $H_l$ is not Hamiltonian, and so, by Corollary~\ref{posacor}, $|P|\geq |H_l|/4\geq n/10^6$, and $V(P)$ has many boosters in $H_l$.

As $e(H_l)\leq pn^2/10^{17} +n\leq pn^2/10^{16}$, by the property from Lemma~\ref{booster}, $G$ contains an edge~$e$ which is a booster in $H_l[V(P)]$. Let $H_{l+1}=H_l+e$. Taking a Hamilton cycle in $H_{l+1}[V(P)]$ and a neighbour to this cycle in $H_{l+1}-V(P)$, we can find a path with length $\mu(H_l)+1$ in $H_l$, which shows that $\mu(H_{l+1})\geq \mu(H_{l})+1$. This contradicts the definition of~$l$, so we must have that $\mu(H_{l})=|H_l|=|H|$.

Finally, if $H_l$ is not Hamiltonian, then, as $\mu(H_l)=|H|$, $U$ must have many boosters in $H_l$ by Lemma~\ref{posa}. As $e(H_l)\leq pn^2/10^{16}$, one of these boosters must be in~$G$, by the property from Lemma~\ref{booster}. Therefore, $G[U]$ is Hamiltonian.
\oof
\fi

\section{Almost spanning trees with long bare paths}\label{4almostspanning}

In this, one of our more technical sections, we set up the main tool we need to embed a tree in Case~B. The main result of this section is Lemma~\ref{almostspan}, which is stated so that it may also be used for Cases~C and D.\color{\newcolour}
~This section accomplishes two main tasks. Firstly, we seek to embed a tree into $G=G(n,p)$ with $p=\Theta(\log n/n)$ when this results in as few as $\Theta(n/\log n)$ uncovered vertices. In Cases~C and D we additionally want the uncovered vertices to form a cycle in $G$. Secondly, we wish to embed this tree so that the uncovered vertices lie within a small prespecified set (which we will call $W$). Achieving both of these tasks under the spectre of universality creates several technicalities. Therefore, for illustration, we will first prove a lemma capable of embedding any single bounded-degree tree with few uncovered vertices (see Lemma~\ref{newembedfew}). Our aim is to introduce how our first main task in this section can be accomplished. We do not use this illustrative lemma except for communicating the intuition behind the probabilities used for this main task (specifically, in the proof of Lemma~\ref{almostalmostspan2}).

\subsection{\textcolor{\newcolour}{Embedding a tree with few spare vertices}}
\textcolor{\newcolour}{Let us state our illustrative lemma before discussing the challenges of its proof.}

\begin{lemma}\label{newembedfew}\color{\newcolour} Let $d>0$ be fixed, let $p=10^{29}d^2\log n/n$ and let $T$ be an $n$-vertex tree with $n-n/\log n$ vertices and maximum degree at most $d$. Then, almost surely, $G=G(n,p)$ contains a copy of $T$.
\end{lemma}
\color{\newcolour}
If we try to use the most direct application of the extendability methods to prove Lemma~\ref{newembedfew}, we run into the following problem. Note that $G=G(n,p)$ from Lemma~\ref{newembedfew} will, by Proposition~\ref{generalprops}, be almost surely $m$-joined for
\begin{equation}\label{mvalue}
m=\frac{10\log (np)}{p}=\Theta\left(\frac{n\log\log n}{\log n}\right).
\end{equation}
Choosing $v\in V(G)$ arbitrarily, we have, almost surely, that $I(\{v\})$ is $(2d,m)$-extendable in $G$ (by Corollary~\ref{generalexpandcor}). Now, to use the extendability methods to embed $T$ (for example, through Corollary~\ref{treebuild}), we must have $\Omega(dm)=\Omega(n\log\log n/\log n)$ spare vertices. However, in embedding $T$ we have only $n/\log n$ spare vertices.

We get around this by observing the following. Say we carry this out and embed as much of $T$ as possible while remaining $(2d,m)$-extendable, and that, once we do, we have $n'$ vertices uncovered by the embedding. From Corollary~\ref{treebuild}, we can have $n'=\Theta(dm)=\Theta(n\log\log n/\log n)$. Let $U$ be the set of uncovered vertices, and let $U'$ be the set of vertices in the partial embedding of $T$ which still need neighbours. Note that $|U'|\leq |U|=n'$. Reveal more edges among the vertex set of $G'$ with probability $p$ to get the graph $G'$. By Proposition~\ref{generalprops},  $G'[U\cup U']$ is almost surely $m'$-joined, with
\begin{equation}\label{mprimevalue}
m'=\frac{10\log (2n'p)}{p}=\Theta\left(\frac{n\log\log\log n}{\log n}\right).
\end{equation}
We can easily deduce from the definition of the $(2d,m)$-extendability of the partial embedding of $T$ into $G$ that $I(U')$ is $(d,m')$-extendable in $G[U\cup U']$, whereupon we can use the extendability methods to embed more vertices of $T$ in $G\cup G'$ until there are only $\Theta(dm')=\Theta(n\log\log\log n/\log n)$ vertices which have not yet been embedded.

The critical point here is that we can embed all but $\Theta(dm')$ vertices instead of all but $\Theta(dm)$ vertices, and, as seen in~\eqref{mvalue} and \eqref{mprimevalue}, $m'$ is distinctly smaller than $m$.
Of course, we want to embed $T$ until there are $O(n/\log n)$ vertices not used in the embedding, so we need to repeat this further. By iterating, each time focussing on a smaller subgraph, say with $n''$ vertices, we can deal with graphs that are $m''$-joined for even smaller values of $m''$ than $m'$.
The smallest value $n''$ can take for Proposition~\ref{generalprops} to apply with edge probability $p$ is $\Theta(1/p)$, so the best possible value of $m''$ we could hope for is $\Theta(n/\log n)$. Thus, when we cover all but $\Theta(n/\log n)$ vertices we will be extracting as much from these methods as possible.

More precisely, for each $i$ from $1$ up to some $\alpha$, we reveal more edges with probability $p_i$ to get an $m_i$-joined graph, with $m_i=10\log(2n_ip_i)/p_i$, which has at most $2n_i$ vertices and in which we have a $(d_i,m_i)$-extendable subgraph, which we then extend until we have $O(d_im_i)$ uncovered vertices. The key is to choose $d_i$, $p_i$ and $n_i$, $1\leq i\leq \alpha$, so that $d_i$ decreases by $d$ each time, ending in $d$, $n_i$ decreases down to $O(n/\log n)$ and the probabilities $p_i$ are such that in total no edge will be revealed with probability more than $O(\log n/n)$. Thus, we will eventually have only $O(d_\alpha m_\alpha)=O(n/\log n)$ uncovered vertices. This we achieve in the following proof, where we first divide $T$ into subtrees to embed in each stage.

\begin{proof}[Proof of Lemma~\ref{newembedfew}]\color{\newcolour}
Let $R_0=T$ and arbitrarily pick $t_0\in V(T)$. Let $T_0=I(\{t_0\})$ and do the following process. For each $i\geq 1$, if $|R_{i-1}|\geq 10^4n/\log n$ do the following. Using Proposition~\ref{littletree}, divide~$R_{i-1}$ into two trees $T_i$ and $R_i$, so that $|R_{i-1}|/10^4 \leq |R_i|\leq |R_{i-1}|/10^3$ and $t_{i-1}\in V(T_{i})$,
and let $t_i$ be the vertex shared by $T_i$ and $R_i$. If $|R_{i-1}|< 10^4n/\log n$, then let $T_i=R_{i-1}$ and stop the process. Say that this process stops with $i=\a$, where therefore $|T_\alpha|<10^4n/\log n$ and  $|T_{\alpha-1}|\geq |R_{\alpha-2}|-|R_{\alpha-1}|\geq |R_{\alpha-2}|/2\geq 5\cdot 10^3n/\log n$.
Note that $|T_\alpha|=|R_{\alpha}|\geq |R_{\alpha-1}|/10^4\geq n/\log n$.

For each $1\leq i\leq \alpha$, we have $|T_i|\geq |R_{i-1}|/2$ and thus $|R_i|\leq |R_{i-1}|/10^3\leq |T_i|/500$. Thus, for each $1\leq i<\alpha$, $|T_i|\geq 500|R_i|\geq 500|T_{i+1}|$. Thus, $|T_{\alpha-1}|\leq (1/500)^{\alpha-2}n$. As $|T_{\alpha-1}|\geq 5\cdot 10^3n/\log n$, for sufficiently large $n$, $\alpha\leq (\log\log n)^2$.

We now define, for each $1\leq i\leq \alpha$, $d_i$, $n_i$, $p_i$ and $m_i$, as discussed just before this proof.
Let $m_0=n$. For each $1\leq i\leq \alpha$, let $d_i=(\alpha+1-i)d$,
\begin{equation}\label{neweqeq}
n_i=n-|T_0\cup\ldots\cup T_{i-1}|+1\geq |T_i|, \;\;\;
p_i=\frac{10^{28}d_i^2}{n_i}
,\;\;\;\text{ and }\;\;\;
m_i=\min\left\{m_{i-1},\frac{10\log(n_ip_i)}{p_i}\right\}.
\end{equation}
Let $n_{\alpha+1}= n-|T_0\cup\ldots\cup T_{\alpha}|+1=n-|T|+1=(n/\log n)+1$. Note that $m_i$, $1\leq i\leq \alpha$, and $d_i$, $1\leq i\leq \alpha$, are decreasing sequences.
Note further that, for each $1\leq i\leq \alpha$,
\begin{equation}\label{damnyou}
n_{i}=(n-|T|)+|T_{i}\cup\ldots\cup T_\alpha|\leq\frac{n}{\log n}+ \sum_{j=0}^{\alpha-i} \frac{|T_{i}|}{500^{j}}\leq 3|T_{i}|,
\end{equation}
where we have used that $|T_{i}|\geq |T_\alpha|\geq n/\log n$.
Therefore, for each $1\leq i<\alpha$, we have
\[
n_{i+1}\leq 3|T_{i+1}|\leq \frac{3|T_i|}{500}\leq  \frac{n_i}{100},
\]
where we have used that  $n_i\geq |T_i|$.
Thus, as $d_i\leq 2d_{i+1}$, we have that, from~\eqref{neweqeq}, $p_{i}\leq p_{i+1}/10$.

Now, for each $1\leq i<\alpha$, we have
\[
n_{i+1}\geq |T_{i+1}|\geq \frac{|R_i|}{2}\geq \frac{|R_{i-1}|}{2\cdot 10^4}\geq \frac{|T_i|}{2\cdot 10^4}\overset{\eqref{damnyou}}{\geq} \frac{n_i}{10^6},
\]
and, as $n_{\alpha+1}\geq n/\log n$ and $|T_\alpha|< 10^4n/\log n$, $n_{\alpha+1}\geq (|T_\alpha|+n/\log n)/10^6= n_\alpha/10^6$.
Therefore, for each $1\leq i\leq\alpha$, we have
\begin{equation}\label{omar}
10d_im_i\overset{\eqref{neweqeq}}{\leq}  \frac{10^2d_i\log(n_ip_i)}{p_i}=\frac{n_i\log(10^{28}d_i^2)}{10^{26}d_i}\leq \frac{n_i}{10^{10}}\leq \frac{n_{i+1}}{10}.
\end{equation}

We will first reveal an initial random graph $G_0$ to almost surely get an extendability condition. Let $p_0=10^3\log n/n$, and note that
\[
1\leq d_1=d\alpha\leq d(\log\log n)^2\leq \frac{p_0(n-1)}{240\log (np_0)}\;\;\;\text{ and }\;\;\;m_1\leq \frac{10n_1\log (10^{28}d_1)}{10^{28}d_1^2}\leq \frac{(n-1)}{8d_1}.
\]
Let $G_0=G(n,p_0)$ and $V_0=V(G_0)$. Arbitrarily pick $s_0\in V(G_0)$. By Corollary~\ref{generalexpandcor} applied with $W=V(G_0)\setminus \{s_0\}$, with probability $1-o(n^{-1})$, $I(\{s_0\})$ is $(d_1,m_1)$-extendable in $G_0$. For each $i$, $0\leq i\leq \alpha$, let $\bar{T}_{i}=T_0\cup\ldots\cup T_{i}$. Let $\bar{S}_0=S_0=I(\{s_0\})$ be a copy of $\bar{T}_0=T_0=I(\{t_0\})$.

We will reveal edges in $\alpha$ stages, indexed by $1\leq i\leq \alpha$, in each stage revealing edges within the vertex set $V(G_0)$ with probability $p_i$ to get the graph $G_i$. Recall that, for each $1\leq i<\alpha$, $p_{i}\leq p_{i+1}/10$. Therefore, when this finishes, edges in $G_0\cup\ldots\cup G_\alpha$ will have been revealed independently at random with probability at most
\[
\sum_{i=0}^\alpha p_i\leq p_0+2p_\alpha\overset{\eqref{neweqeq}}{=}  \frac{10^3\log n}{n}+10^{28}\cdot\frac{2d_\alpha^2}{n_\alpha}\leq \frac{10^3\log n}{n}+10^{28}\frac{2d^2}{|T_\alpha|}\leq \frac{10^{29}d^2\log n}{n},
\]
using that $|T_\alpha|\geq n/\log n$. Therefore, if the required copy of $T$ almost surely exists in $G_0\cup\ldots\cup G_{\alpha}$, then it almost surely exists in $G(n,10^{29}d^2\log n/n)$.

For each $i$, $1\leq i\leq \alpha$, in turn, do the following. Suppose we have revealed the random graphs $G_1,\ldots, G_{i-1}$ and have set $\bar{G}_{i-1}=G_0\cup\ldots\cup G_{i-1}$, and have $\bar{S}_{i-1}$, a copy of $\bar{T}_{i-1}$ in which $t_{i-1}$ is copied to $s_{i-1}$ such that, letting
$V_i=(V(G_0)\setminus V(\bar{S}_{i-1}))\cup\{s_{i-1}\}$, $I(\{s_{i-1}\})$ is $(d_{i},m_{i})$-extendable in $\bar{G}_{i-1}[V_i]$. (Note that we already have this situation for $i=1$.) Reveal edges independently at random with probability $p_i$  within the vertex set $V(G_0)$ to get the graph $G_i$, and let $\bar{G}_i=\bar{G}_{i-1}\cup G_i$. As $|V_i|=n-|\bar{S}_{i-1}|+1=n-|\bar{T}_{i-1}|+1= n_i$, and $\bar{G}_{i-1}[V_i]$ is $m_{i-1}$-joined, by Proposition~\ref{generalprops}, with probability $1-o(n_i^{-2})=1-o(n^{-1})$, $\bar{G}_i[V_i]$ is $m_i$-joined. Note that, as we have at most $(\log\log n)^2$ stages, we can almost surely assume this holds at each stage.

Note that
\[
|V_i|=n-|\bar{S}_{i-1}|+1=n-|\bar{T}_i|+|T_i|=n_{i+1}-1+|T_i|\overset{\eqref{omar}}{\geq} |T_i|+10d_im_i.
\]
Therefore, as $I(\{s_{i-1}\})$ is $(d_i,m_i)$-extendable in $\bar{G}_i[V_i]$, and $d_i\geq d_\alpha=d$, by Corollary~\ref{treebuild}, we can find a copy $S_i$ of $T_i$ in which $t_{i-1}$ is copied to $s_{i-1}$ which is $(d_i,m_i)$-extendable in $\bar{G}_i[V_i]$. Let $\bar{S}_i=\bar{S}_{i-1}\cup S_i$. Note that $\bar{S}_i$ is a copy of $\bar{T}_i$. When we reach $i=\alpha$, we thus will have found a copy of $\bar{T}_\alpha=T$.

If $i<\alpha$, then let $s_{i}$ be the copy of $t_{i}$ in $S_i$. Let $V_{i+1}=(V_{i}\setminus V(S_{i}))\cup \{s_{i}\}$. As $S_{i}$ is $(d_i,m_{i})$-extendable in $\bar{G}_{i}[V_i]$, $s_{i}\in V(S_{i})$, $m_{i+1}\leq m_{i}$ and $d_{i+1}\leq d_i-\Delta(S_i)$, it follows simply from the definition of $(d_i,m_{i})$-extendability that $I(\{s_{i}\})$ is $(d_{i+1},m_{i+1})$-extendable in $\bar{G}_{i}[V_i]$. Thus, we have the required situation to reveal the stage $i+1$ edges.
\end{proof}

\subsection{\color{\newcolour} Using long bare paths}
\color{\newcolour}
Let us now focus on the second main task, where we have to embed a tree $T$ in a random graph $G$ so that the uncovered vertices in $G$ all lie within a prespecified set $W$. As above, we wish to achieve this in situations where $W$ is as small as $\Theta(n/\log n)$, and therefore we use an iterative procedure where we embed more and more of $T$ while pushing the uncovered vertices into smaller and smaller sets (reducing in size by a constant multiple each time) until the whole tree is embedded and the uncovered vertices lie in $W$.

All the trees we embed are large subtrees of trees in Cases B, C or D, and therefore we can choose them to have many long bare paths. Thus, in each iteration, we can embed several long bare paths between vertex pairs in a partial embedding of $T$. We separate this iterative step as Lemma~\ref{almostalmostspan}, which embeds a collection of paths between certain pairs of vertices, so that the uncovered vertices lie in a small linear-sized subset $W$. We then iterate this for Lemma~\ref{almostalmostspan2}, which effectively achieves the same thing but with a much smaller set $W$. Finally, we use Lemma~\ref{almostalmostspan2} to prove the main lemma of this section,
Lemma~\ref{almostspan}, which combines this with an embedding of the rest of the tree $T$ (that is, not just the long bare paths) into the form that we then use later, and which is the only one of the lemmas in this section used elsewhere.

We achieve the second of our main tasks for Lemma~\ref{almostalmostspan} -- ensuring all the vertices outside of $W$ are covered by the paths -- by finding a cycle containing all the uncovered vertices outside of $W$, breaking it into paths, and then joining these paths between vertex pairs using the vertices in $W$. We do this with Lemma~\ref{hamcycleinsubgraph}, where this use relies on $W$ being a linear-sized subset.

Before stating and proving Lemma~\ref{almostalmostspan}, let us comment on two further aspects of its proof. Firstly, the length of the paths in the tree we are embedding, and, secondly, the complications that arise in the lemma statements to allow for universality.

\smallskip

\textbf{The length of paths in $T$.} When we iterate  to push the uncovered vertices into a small set $W$ for Lemma~\ref{almostalmostspan2}, we do so at each stage by finding long bare paths for the tree $T$ that we are trying to embed. These paths can have length up to $\Theta(k(T))$. Furthermore, if $p=\Theta(\log n/n)$, $G=G(n,p)$, and $W\subset V(G)$ is a large set, $G[W]$ likely has average degree around $p|W|$. Therefore, as we embed the paths with length $\Theta(k(T))$ using vertices in $W$ and Corollary~\ref{trivial} (applying it with $\log m\approx \log |W|$
 and $d\approx p|W|$), we need
\begin{equation}\label{tricky}
k(T)=\Omega(\log|W|/\log (p|W|)).
\end{equation}
This restricts how small $W$ can be, and results in the bound on the size of $W$ seen in Lemmas~\ref{almostalmostspan2} and \ref{almostspan} and the value taken for the edge probability in Lemma~\ref{almostalmostspan}. In Cases C and D, as $k(T)=\Omega(\log n)$, $W$ can be taken to be as small as $\Theta(\log n/n)$. In Case B, $W$ cannot be as small, but here our embedding will be untroubled as, in compensation, the larger $W$ must be, the more leaves the tree will have. As, in this case, we embed leaves to cover the final uncovered vertices, we can allow $W$ to be larger.

Let us remark further that the bounds on $|W|$ coming from \eqref{tricky} in Case B can be difficult to interpret. These bounds are only close to tight when $k(T)=\Theta(\log n)$ or $k(T)=\Theta(\log n/\log\log n)$. Between these values (that is, in the span of Case B), the set $W$ in the random graph that arises naturally from the embedding comfortably has high enough average degree to connect vertices within $W$ with paths of length $\Theta(k(T))$. In particular, the logarithms in~\eqref{tricky} result in plenty of room in the bounds used for $|W|$.

\smallskip

\textbf{Universality.}\color{black}
~The actual lemmas have a more technical statement than suggested by this discussion, as, with universality in mind, we gather properties which have the potential to embed many different trees, rather than actually embedding a specific tree. Given a vertex set~$W$ in a certain random graph, Lemma~\ref{almostalmostspan} shows that the graph will contain two disjoint vertex sets~$V$ and~$A\subset W$ so that, given the correct number of pairs of vertices from $V$, there exist disjoint paths of some specified length between these pairs of vertices, with interior vertices in $A$, together with a disjoint cycle, so that the cycle covers exactly the vertices in $A$ which are not in the paths.
Lemma~\ref{almostalmostspan} also contains a variation of this result where we may take $V=V(G)\setminus A$ if we expose edges with an increased probability. The structure found in the lemma is illustrated in Figure~\ref{spanpic4}.

\lem\label{almostalmostspan} Let $n,l,r,w\in \N$ satisfy $l\geq 20\log n/\log\log n$, $rl+2w\leq n/2$ and $w\geq 4n/10^6$. Take $d=10\exp(10\log n/l)$. Suppose $W\subset[n]$ satisfies $|W|\geq rl+2w$. Then, with probability $1-o(n^{-2})$, the random graph $G=\GG(n, 10^{28}d\log d/n)$ contains subsets $A\subset W$ and $V\subset V(G)\setminus A$ with $|A|=rl+w$ and $|V|\geq n/2$, so that the following property holds.

Given any collection of $r$ disjoint pairs $(x_i,y_i)$, $i\in[r]$, of vertices in $V$, we can find disjoint $x_i,y_i$-paths in~$G$, with length $l+1$ and interior vertices in $A$, together with a cycle in~$G$ through exactly those $w$ vertices in $A$ which are not covered by the $x_i,y_i$-paths.

Furthermore, with probability $1-o(n^{-2})$, the random graph $G=\GG(n,10^{28}\log n/n)$ contains a set $A\subset W$, with $|A|=rl+w$, so that $A$ has this property with $V=V(G)\setminus A$.
\ma

\ifdraft
\else
\pr
Note that $d\geq 10$ and, as $l\geq 20\log n/\log\log n$, $d\leq 10\sqrt{\log n}$. For the first part of the lemma, let $p=10^{10}d\log d/ n$ and
\[
m=\frac{10\log(np)}{p}=\frac{10n\log(10^{10}d\log d)}{10^{10}d\log d}\leq \frac{n\log(10^{10})\log d^2}{10^{9}d\log d}\leq \frac{n}{10^7d}\leq \frac{w}{10d}.
\]
Reveal edges within the vertex set $[n]$ with probability $p$ to get the graph~$G_1$. By Proposition~\ref{alledges}, with probability $1-o(n^{-2})$, $G_1$ has at most $pn^2$ edges. By Proposition~\ref{generalprops}, with probability $1-o(n^{-2})$, $G_1$ is $m$-joined.

\begin{figure}[b!]
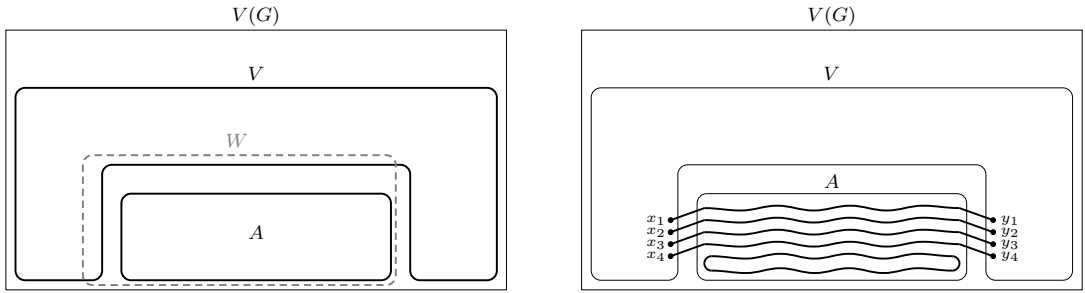

\centering
    \begin{subfigure}[b]{0.45\textwidth}
        \centering
        \resizebox{\linewidth}{!}{
            \input{spanpic4A1}
        }
\label{spanpic4A1}
    \end{subfigure}
\hspace{0.25cm}
    \begin{subfigure}[b]{0.45\textwidth}
        \centering
        \resizebox{\linewidth}{!}{
            \input{spanpic4A2}
        }
\label{spanpic4A2}
    \end{subfigure}

\vspace{-0.4cm}

\caption{A depiction of Lemma~\ref{almostalmostspan}. Given a set $W\subset [n]$ on the left, there are likely to be two disjoint sets $V$ and $A$ in $G$, with the property on the right. That is, given any disjoint pairs of vertices $(x_i,y_i)$, $i\in [r]$, in $V$ (here $r=4$), we can find disjoint $x_i,y_i$-paths with interior vertices in~$A$ and a fixed length, such that the remaining vertices in~$A$ can be covered exactly by a cycle.} \label{spanpic4}
\end{figure}

Take a subset $A_0\subset W$ with $|A_0|=rl+w$. As $|A_0|\geq w\geq 10dm$, by Proposition~\ref{neat}, we can take $B\subset V(G_1)$ to be a subset, with $|B|\leq m$, satisfying the following property.
\stepcounter{capitalcounter}
\begin{enumerate}[label=\bfseries \Alph{capitalcounter}\arabic*]\addtocounter{enumi}{0}
\item If $U\subset [n]\setminus B$ and $|U|\leq 2m$, then $|N_{G_1}(U,A_0\setminus B)|\geq d|U|$.\label{AA2}
\end{enumerate}

Choose $A\subset W\setminus B$ to satisfy $A_0\setminus B\subset A$ and $|A|=rl+w$. Let $V=V(G_1)\setminus (A\cup B)$, so that $|V|\geq n-rl-w-m\geq n-rl-2w\geq n/2$. Reveal more edges with probability $10^{17}p$ to get the graph~$G_2$, and let $G=G_1\cup G_2$. By Lemma~\ref{hamcycleinsubgraph}, with probability $1-o(n^{-2})$, we have that if a subgraph $H\subset G$ is a $(|H|,2)$-expander, with $e(H)\leq pn^2$ and $|H|\geq 4n/10^6$, then the graph $G[V(H)]$ is Hamiltonian. If $H\subset G_1$, then $e(H)\leq e(G_1)\leq pn^2$, so, as $w\geq 4n/10^6$, we have the following property.
\begin{enumerate}[label=\bfseries \Alph{capitalcounter}\arabic*]\addtocounter{enumi}{1}
\item If $H\subset G_1$ with $|H|=w$ is a $(w,2)$-expander then there is a cycle in~$G$ supported by $V(H)$.\label{AA3}
\end{enumerate}

We will now show that the property in the lemma holds for the sets $A$ and $V$. To see this, take any collection of $r$ disjoint pairs $(x_i,y_i)$, $i\in[r]$, of vertices in $V$. Let $X=\{x_i,y_i:i\in[r]\}$. By Property~\ref{AA2}, for any subset $U\subset X\cup A$, with $|U|\leq 2m$, we have that $|N_{G_1}(U, A)|\geq |N_{G_1}(U,A_0\setminus B)|\geq d|U|$. Therefore, the subgraph $I(X)$ (with vertex set $X$ and no edges, and thus maximum degree 0) is, by Proposition~\ref{uttriv}, $(d,m)$-extendable in $G_1[X\cup A]$.

Using this extendability, and as $G_1$ is $m$-joined, we can apply Corollary~\ref{trivial} repeatedly (with details checked below) to find disjoint $x_i,y_i$-paths $P_i$, $i\in[r]$, with length $l+1$ and interior vertices in $A$, so that $I(X)\cup (\cup_{i\in[r]}P_i)=\cup_{i\in[r]}P_i$ is $(d,m)$-extendable in $G_1[X\cup A]$. This is possible because, firstly, as $d-1\geq \exp(10\log n/l)$, we have that
\[
2\lceil \log (2m)/\log (d-1)\rceil+1\leq 3\log n/\log (d-1)\leq 3l/10\leq l,
\]
and, secondly, because each such path found adds $l$ vertices to the working subgraph, so that we will always have at least $|X\cup A|-(|X|+(r-1)l)=w+l\geq 10dm+l$ spare vertices at each application of Corollary~\ref{trivial}. Therefore, both the length condition and the size condition hold for these applications of Corollary~\ref{trivial} and we can find the paths as described. It is left then just to cover exactly the remaining vertices in~$A$ with a cycle.


Let $H=G_1[A\setminus V(\cup_{i\in[r]}P_i)]$, so that $|H|=w$. If $U\subset V(H)$ and $|U|\leq m$, then, as $\cup_{i\in[r]}P_i$ is $(d,m)$-extendable in $G_1[X\cup A]$ and $U\cap V(\cup_{i\in [r]}P_i)=\emptyset$, by the definition of $(d,m)$-extendability (Definition~\ref{extendabledefn}), we have that $|N'_{G_1[A\cup X]}(U)\setminus V(\cup_{i\in[r]}P_i)|\geq(d-1)|U|$. Therefore, if $U\subset V(H)$ and $|U|\leq m$, then
\[
|N_H(U)|\geq |N'_H(U)|-|U|\geq |N'_{G_1[A\cup X]}(U)\setminus V(\cup_{i\in[r]}P_i)|-|U|\geq(d-1)|U|-|U|\geq 2|U|,
\]
where we have used that $d\geq 10$. If $U\subset V(H)$ and $m\leq |U|\leq w/4$, then, as there are no edges in $H$ between $U$ and $V(H)\setminus (U\cup N_H(U))$, and the graph $G_1$ is $m$-joined, we have that $|V(H)\setminus (U\cup N_H(U))|\leq m$, and, hence, $|N_H(U)|\geq |H|-|U|-m\geq w/2\geq 2|U|$.
Therefore,~$H$ is a $(w,2)$-expander. By Property~\ref{AA3}, then, there is a cycle in~$G$ with vertex set $(X\cup A)\setminus V(\cup_{i\in[r]}P_i)$. This cycle, along with the disjoint paths $P_i$, $i\in[r]$, is the structure we require. Therefore, the sets $A$ and $V$ have the property in the lemma.

For the second part of the lemma, first recall that $d\leq 10\sqrt{\log n}$. Let $p=10^{10}\log n/n$ and $m=10\log(np)/p$, so that, for sufficiently large $n$,
\[
m=\frac{10\log(10^{10}\log n)}{p}\leq \frac{n\log\log n}{\log n}\leq \frac{n}{10^8d}\leq \frac{w}{10^2d}.
\]
Let $A\subset W$ be any set with $|A|=rl+w$ and reveal edges with probability $p$, to get a graph~$G_1$ which, with probability $1-o(n^{-2})$, has at most $pn^2$ edges and is $m$-joined. Note that $p|A|\geq pw\geq pn/10^6\geq 200\log n$, $d\leq w/24m=pw/240\log(np)$, and $2m\leq w/4d$. Therefore, by Lemma~\ref{generalexpand}, with probability $1-o(n^{-2})$, Property~\ref{AA2} holds with $A_0=A$ and $B=\emptyset$ (this is the crucial difference from the argument for the first part of the lemma). Revealing more edges with probability $10^{17}p$ to get~$G_2$, and letting $G=G_1\cup G_2$, by Lemma~\ref{hamcycleinsubgraph} we get, with probability $1-o(n^{-2})$, Property~\ref{AA3} for the new graphs $G_1$ and $G$. The argument for the first part of the lemma then shows that $V=V(G)\setminus (A\cup B)=V(G)\setminus A$ and $A$ have the property in the lemma, as required.
\oof
\fi

Lemma~\ref{almostalmostspan} allows us to embed some paths while forcing the remaining vertices into the set~$W$, where they can then be covered exactly by a cycle. Lemma~\ref{almostalmostspan} requires that $|W|\geq 2w \geq 8n/10^6$, which is too large for some of our applications, where we may have $|W|=\Theta(n/\log n)$. Roughly speaking, for a smaller set~$W$, by applying Lemma~\ref{almostalmostspan} repeatedly and finding the paths in stages we can force the uncovered vertices into smaller and smaller sets until they eventually lie within~$W$. This allows us to prove the following version of Lemma~\ref{almostalmostspan}, in which the set~$W$ can be smaller but one half of the vertex pairs are preselected. The lower bound for $w$ in Lemma~\ref{almostalmostspan2} is similar to the size of the set~$W$ in the discussion at the start of this section if $l=\Theta(k(T))$, and, similarly, it is important that we may expect our random graph induced on~$W$ to have sufficient average degree to allow us to form paths of length $l$ between many different vertex pairs.

Due to the desired universality (which will demand a choice of vertex pairs), we again construct a sequence of sets with increasing size and certain properties. Then, given any suitable collection of vertex pairs, we find the paths iteratively, while forcing the uncovered vertices into the sets in order of decreasing size. These sets will decrease by a constant multiple (of 10) each time, just like in the proof of Lemma~\ref{newembedfew}, while the probabilities we use in the iteration will similarly increase by a constant multiple each time, up to a final edge probability that is $O(\log n/n)$.

\lem\label{almostalmostspan2} Let $l,n,w\in \N$ satisfy $20\log n/\log\log n\leq l\leq 20\log n$ and
\[
\frac{n}{10^{20}\log n}\exp\left(\frac{20\log n}{l}\right)\leq w\leq \frac{n}{10^5}.
\]
Suppose $W,B\subset [n]$ are disjoint sets satisfying $|W|\geq 100w$ and $|B|\leq n/4$, and let $r= 4\lfloor n/4(10^4l)\rfloor$. With probability $1-o(n^{-1})$, the random graph $G=\GG(n,10^{50}\log n/2n)$ contains a subset $A\subset V(G)\setminus B$ with $|A|=3rl+w$, and distinct pairs $(x_i,y_i)$, $i\in[r/2]$, of vertices from $V(G)\setminus (B\cup A)$ so that the following property holds with $X=\cup_{i=1}^{r/2}\{x_i,y_i\}$.

Given any collection of $r/2$ disjoint pairs $(x_i,y_i)$, $i\in[r]\setminus[r/2]$, of vertices in $V(G)\setminus (A\cup X)$, we can find disjoint $x_i,y_i$-paths, $i\in[r]$, in~$G$, with length $3l+1$ and interior vertices in $A$, so that the $w$ vertices in $A$ which are not covered by the paths lie in~$W$ and there is a cycle in~$G$ through exactly those $w$ uncovered vertices in $A$.
\ma
\pr Let $\a\geq 0$ be the integer for which $n/10^5<5l\lfloor w/l\rfloor 10^{\a}/2\leq n/10^4$. We will apply the first part of Lemma~\ref{almostalmostspan} $\a$ times on sets with increasing size, before applying the second part of the lemma once. This is depicted in Figure~\ref{spanpic5A}, with the final sets we form depicted in Figure~\ref{spanpic5B}, while parts of the argument demonstrating these final sets satisfy the lemma are depicted in Figures~\ref{spanpic5C} and~\ref{spanpic5D}. For simplicity we assume in these figures that $B=\emptyset$; in general the subsets depicted in the figures are disjoint from $B$.

Let $w_0=w/l$, and, for each~$i\in[\a+1]$, let $w_i=\lfloor w/l\rfloor 10^{i}$. Let $W_0\subset W$ satisfy $|W_0|=10w_1l$. With probability $1-o(n^{-1})$, we will find disjoint sets $W_1,\ldots,W_\a\subset [n]$ and, for each $j\in[\a]$, a set $A_j\subset W_{j-1}$, so that $|A_j|=2w_{j}l+w_{j-1}l$, $|W_j|=10w_{j+1}l$, and the following property holds.

\stepcounter{capitalcounter}
\begin{enumerate}[label=\bfseries \Alph{capitalcounter}\arabic*]
\item Given any disjoint pairs $(a_i,b_i)$, $i\in[2w_j]$, of vertices from $W_{j}$, there are disjoint $a_i,b_i$-paths, with length $l+1$ and interior vertices in $A_{j}$, so that the $w_{j-1}l$ uncovered vertices in $A_{j}$ support a cycle in $G_j$, a subgraph of our final random graph $G$.
\label{B1}
\end{enumerate}

\setcounter{capitalcounterbackup}{\value{capitalcounter}}
\stepcounter{capitalcounterbackup}

To find these sets, let $d=10\exp(10\log n /l)\geq 10$ and carry out the following Step~{\bfseries \Alph{capitalcounterbackup}$_j$} for each~$j$,~$1\leq j\leq \a$. These steps are depicted in Figure~\ref{spanpic5A}.

\begin{figure}[t!]
\centering
    \begin{subfigure}[b]{0.6\textwidth}
        \centering
        \resizebox{\linewidth}{!}{
            \input{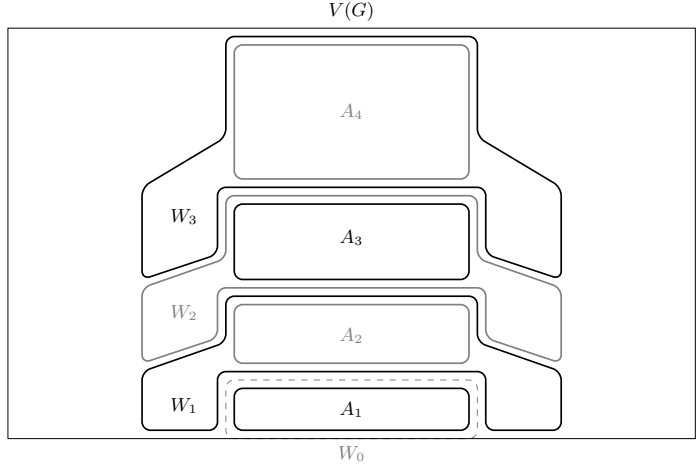}
        }
\label{spanpic4A1}
    \end{subfigure}

\vspace{-0.7cm}

\caption{The sets found in the random graph during the proof of Lemma~\ref{almostalmostspan2}, where here $\a$ is very small, with $\a=3$. Starting with the set $W_0\subset W$, we carry out Step~{\bfseries \Alph{capitalcounterbackup}}$_1$, to find the sets~$A_1$ and $W_1$ so that, given the appropriate number of disjoint pairs of vertices from $W_1$, we can connect these pairs with disjoint paths with length $l+1$ so that the remaining vertices in $A_1$ can be covered exactly with a cycle. Carrying out Steps~{\bfseries \Alph{capitalcounterbackup}}$_2$  and~{\bfseries \Alph{capitalcounterbackup}}$_3$, we similarly find $A_2$ and $W_2$, and then $A_3$ and~$W_3$. Finally, using Lemma~\ref{almostalmostspan}, we find the set $A_{\a+1}=A_4$ so that Property~\ref{B2} holds. That is,~$A_{\a+1}$ can similarly be covered by a cycle and paths connecting disjoint pairs of vertices, but the pairs of vertices may come from anywhere outside of $A_{\a+1}$.
} \label{spanpic5A}

\end{figure}

\begin{itemize}[label =\bfseries \Alph{capitalcounterbackup}$_j$ ]
\item Noting $\sum_{i=1}^{j-1}|W_i|= \sum_{i=1}^{j-1}10w_{i+1}l\leq 20w_j l\leq n/10^3$, pick $V_{j}\subset [n]\setminus(B\cup (\cup_{i< j-1}W_i))$ with $W_{j-1}\subset V_j$ and $|V_j|=10^3w_{j}l$. Reveal edges with probability $10^{28}d\log d/|V_{j}|$ to get the graph~$G_j$. Note that $w_{j-1}l\geq 4|V_{j}|/10^6$ and $|W_{j-1}|=10w_jl\geq 2w_jl+2w_{j-1}l$, so we may apply the first part of Lemma~\ref{almostalmostspan} to $G_j[V_{j}]$ and $W_{j-1}$ with values $(n,l,r,w)=(|V_{j}|,l,2w_j,w_{j-1}l)$. With probability $1-o(n^{-2})$, we can then find disjoint sets $A_{j}\subset W_{j-1}$ and $W_{j}\subset V_{j}\setminus W_{j-1}$ so that $|A_{j}|=2w_{j}l+w_{j-1}l$, $|W_{j}|=10w_{j+1}l\leq |V_{j}|/2-|W_{j-1}|$ and Property~\ref{B1} holds in~$G_j$. \label{Sj}
\end{itemize}

Note that, for each $j\in [\a]$, $W_{j}\subset V_{j}\subset V(G)\setminus B$, and recall that $W_0\subset W\subset V(G)\setminus B$. Eventually we will choose the desired set~$A$ as a subset of $\cup_{j=0}^\a W_j$, and thus we will have that $A\subset V(G)\setminus B$.

Note that the probability that the required set $A_j$ could be found for each $j$, $1\leq j\leq \a$, is $1-o(n^{-2}\a)=1-o(n^{-1})$.
Having completed these steps, let
\[
k'=\sum^{\a}_{i=1}w_j=\left\lfloor\frac wl\right\rfloor \sum_{i=1}^\a 10^i\leq \frac{5}{4}\cdot 10^\a\left\lfloor\frac wl\right\rfloor\leq \frac{n}{10^4(2l)}\leq \frac r2.
\]
Reveal edges between vertices in $[n]$ with probability $10^{28}\log n/n$ to get the graph $G_{\a+1}$ and let $G=\cup_{i=1}^{\a+1}G_i$. As $w_{\a}l=10^\a \lfloor w/l\rfloor l> 2n/5(10^5)=4n/10^6$ and
\begin{equation}\label{halften}
|W_{\a}|=10w_{\a+1} l> 200n/5(10^5)=4n/10^4\geq 3rl\geq 3(r-k')l+2w_{\a}l,
\end{equation}
we can apply the second part of Lemma~\ref{almostalmostspan} to $G_{\a+1}$ and $W_{\a}$ using $(n,l,r,w)=(n,3l,r-k',w_{\a}l)$. With probability $1-o(n^{-1})$, we get a set $A_{\a+1}\subset W_{\a}$ with $|A_{\a+1}|=3l(r-k')+w_{\a}l$ so that the following property holds.

\begin{enumerate}[label=\bfseries \Alph{capitalcounter}\arabic*]\addtocounter{enumi}{1}
\item Given any $r-k'$ disjoint pairs $(a_i,b_i)$, $i\in[r-k']$, of vertices from $V(G_{\a+1})\setminus A_{\a+1}$, there are disjoint $a_i,b_i$-paths, with length $3l+1$ and interior vertices in $A_{\a+1}$, so that the $w_{\a}l$ uncovered vertices in $A_{\a+1}$ support a cycle in $G_{\a+1}$.\label{B2}
\end{enumerate}

In total we have revealed edges with probability at most
\begin{align*}
10^{28}\frac{\log n}{n} +\sum_{i=1}^{\a} \frac{10^{28}d\log d}{|V_i|}&\leq 10^{28}\frac{\log n}{ n}+10^{28}\left(\frac{2}{|V_1|}\right)d^2\\
&\leq 10^{28}\frac{\log n}{n} + 10^{28}\left(\frac{1}{10^3w}\right)10^2\exp\left(\frac{20\log n}{l}\right)
\leq 10^{50}\frac{\log n}{2n},
\end{align*}
where we have used the lower bound for $w$ in the statement of the lemma.

For each $j\in[\a-1]$, note that $|W_j\setminus A_{j+1}|=10w_{j+1}l-(2w_{j+1}l+w_{j}l)\geq 2w_{j}$ and choose disjoint pairs $(u_{j,i},v_{j,i})$, $i\in [w_{j}]$, of vertices from $W_{j}\setminus A_{j+1}$. By~(\ref{halften}), $|W_\a|\geq 3l(r-k')+2w_\a l=|A_{\a+1}|+w_\a l$, so we may pick disjoint pairs $(u_{\a,i},v_{\a,i})$, $i\in [w_{\a}]$, of vertices from $W_{\a}\setminus A_{\a+1}\subset V(G)\setminus B$. Let $A=\cup_{i\leq\a+1}A_i$, noting that $|A|\leq 10w_\a l+3rl\leq n/2$. Choose disjoint pairs $(u_{\a+1,i},v_{\a+1,i})$, $i\in [r/2-k']$, of vertices from $V(G)\setminus (B\cup A \cup(\cup_{j\leq \a,i\in[w_j]} \{u_{j,i},v_{j,i}\}))$ (see Figure~\ref{spanpic5B}).
Label the vertices in the $r/2$ vertex pairs in both $\{(u_{j,i},v_{j,i}): j\in[\a],i\in [w_{j}]\}$ and $\{(u_{\a+1,i},v_{\a+1,i}):i\in [r/2-k']\}$, so that together we have the vertex pairs $\{(x_i,y_i):i\in[r/2]\}$.

\begin{figure}[t!]
\centering

\centering
    \begin{subfigure}[b]{0.6\textwidth}
        \centering
        \resizebox{\linewidth}{!}{
            \input{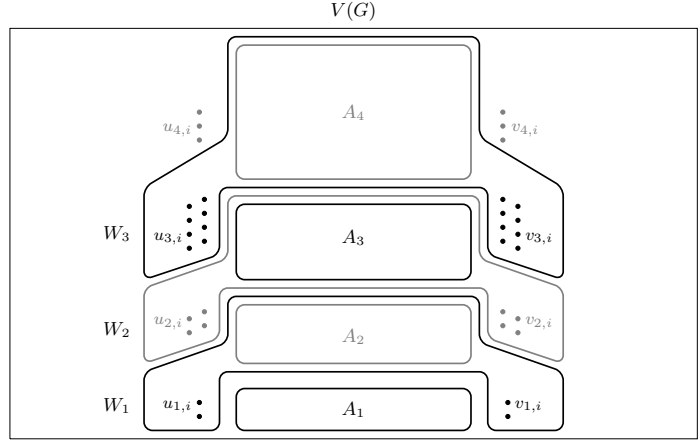}
        }
\label{spanpic4A1}
    \end{subfigure}

\vspace{-0.4cm}

\caption{Continuing with the structure found in Figure~\ref{spanpic5A} and with $\a=3$, for each $j\in[\a]$ we pick disjoint pairs $(u_{j,i},v_{j,i})$, $i\in [w_j]$, of vertices from $W_j\setminus A_{j+1}$, and, further disjoint pairs $(u_{\a+1,i},v_{\a+1,i})$, $i\in [r/2-k']$ in $V(G)\setminus (\cup_{i\in [\a+1]}A_i)$. We let $X$ be the set of all 
these vertices. We relabel these vertices as $(x_i,y_i)$, $i\in [r/2]$, for the statement of the lemma, but use the first labels in Figures~\ref{spanpic5C} and~\ref{spanpic5D}.
We then let $A=\cup_{i\in[\a+1]}A_i$, which, with the pairs of vertices in $X$, will satisfy the requirements of Lemma~\ref{almostspan}.
} \label{spanpic5B}
\end{figure}

\stepcounter{capitalcounter}
\stepcounter{capitalcounter}

\setcounter{STEPJcounter}{\value{capitalcounter}}

We claim the conclusion of the lemma holds with the vertex pairs $(x_i,y_i)$, $i\in[r/2]$, and the set~$A$. Indeed, suppose we have $r/2$ disjoint vertex pairs $(x_i,y_i)$, $i\in [r]\setminus[r/2]$ in $V(G)\setminus (A\cup X)$, where $X=\cup_{i=1}^{r/2}\{x_i,y_i\}$.
By Property~\ref{B2}, we can find disjoint $x_i,y_i$-paths, $i\in [r]\setminus[k']$, with length $3l+1$ and internal vertices in $A_{\a+1}$, so that the $w_{\a}l$ uncovered vertices in $A_{\a+1}$ support a cycle, $Q_{\a+1}$ say.
Now, for~$j$ decreasing from $\a$ down to 1, carry out the following Step~$\mathbf{\Alph{capitalcounter}}_j$ which, given a cycle $Q_{j+1}$ in~$A_{j+1}$, with $|Q_{j+1}|=w_jl$, finds disjoint $u_{j,i},v_{j,i}$-paths, $i\in[w_j]$, with length $3l+1$ and  interior vertices in $V(Q_{j+1})\cup A_j$ so that the $w_{j-1}l$ uncovered vertices lie in $A_j$ and support a cycle, which we call $Q_j$. A typical Step~{\bfseries \Alph{capitalcounter}$_j$} is depicted in Figures~\ref{spanpic5C} and~\ref{spanpic5D}.


\begin{itemize}[label =\bfseries \Alph{capitalcounter}$_j$ ]
\item Take the cycle $Q_{j+1}$ and break it into $w_j$ disjoint paths with length $l-1$. Let the pairs of endvertices of these paths be $(u'_{j,i},v'_{j,i})$, $i\in[w_j]$. These vertices lie in $A_{j+1}\subset W_j$, and the vertices $u_{j,i}$, $v_{j,i}$, $i\in[w_j]$ lie in $W_j\setminus A_{j+1}$. By Property~\ref{B1}, we can find disjoint $u_{j,i},u'_{j,i}$-paths and disjoint $v'_{j,i},v_{j,i}$-paths with length $l+1$ and interior vertices in $A_{j}$ so that the uncovered vertices in $A_j$ form a cycle, $Q_{j}$ say, with $|Q_j|=w_{j-1}l$. For each $i\in[w_j]$, by combining the $u_{j,i},u'_{j,i}$-path and the $v'_{j,i},v_{j,i}$-path with the $u'_{i,j},v'_{i,j}$-path from the cycle $Q_{j+1}$, we have a $u_{j,i},v_{j,i}$-path with length $3l+1$.
\end{itemize}


When all the steps are completed, we have a disjoint set of $x_i,y_i$-paths, $i\in [r]$, with length $3l+1$ and internal vertices in $A$, and the final cycle $Q_1$ covers the $w_0l=w$ unused vertices in $A$, which are all in $A_1\subset W_0\subset W$.
\oof

\begin{figure}[p]
\vspace{-0.5cm}

\centering
    \begin{subfigure}[b]{0.75\textwidth}
        \centering
        \resizebox{\linewidth}{!}{
            \input{spanpic5C1}
        }
\label{spanpic5C1}
    \end{subfigure}

\vspace{-0.5cm}

\caption{To show that the set $A=\cup_{i\in[\a]}A_i$ and vertex pairs in $X$, as found in the proof of Lemma~\ref{almostalmostspan2}, satisfy the conditions of that lemma, where here $\a=3$, we pick arbitrary disjoint pairs $(x_i,y_i)$, $i\in [r]\setminus [r/2]$, of vertices from $V(G)\setminus (A\cup X)$. Adding to these pairs the vertex pairs $(u_{\a+1,i},v_{\a+1,i})$, $i\in [r/2-k']$, from $X$, here shown in grey, we use Property~\ref{B2} to find paths connecting these vertex pairs using vertices in $A_{\a+1}$, and cover the remaining vertices in $A_{\a+1}$ by a cycle,~$Q_{\a+1}$. We are now ready to start the Steps~$\mathbf{\Alph{STEPJcounter}}_j$ for $j$ from $\a$ to $1$, as depicted in Figure~\ref{spanpic5D}.
} \label{spanpic5C}

\vspace{0.5cm}

\centering
    \begin{subfigure}[b]{0.75\textwidth}
        \centering
        \resizebox{\linewidth}{!}{
            \input{spanpic5D1}
        }
\label{spanpic5D1}
    \end{subfigure}

\vspace{-0.5cm}

\caption{A typical Step~$\mathbf{\Alph{STEPJcounter}}_j$ in the proof of Lemma~\ref{almostalmostspan2}, here shown for $j=3$. We take the cycle $Q_{j+1}$, as shown in Figure~\ref{spanpic5C}, and divide it into subpaths with new endvertices giving the pairs $(u'_{j,i},v'_{j,i})$, $i\in [w_j]$. Recalling the set $W_j$ from Figure~\ref{spanpic5B} we see that it contains these new vertices, as well as the vertices $u_{j,i}$, $v_{j,i}$, $i\in [w_j]$. Therefore, we can use Property~\ref{B1} to find the $u_{j,i},v_{j,i}$-paths, $i\in [w_j]$, using vertices from $A_j$ and the vertices in $Q_{j+1}$, while covering the other vertices in $A_j$ by the cycle $Q_j$. The paths found so far then cover $A_{j+1}$, and all of $A_j$ except for the cycle $Q_j$, which we use to start the next step. When all the Steps~$\mathbf{\Alph{capitalcounter}}_j$ from $j=\a$ to $j=1$ have been completed, all of $A=\cup_iA_i$ will be covered with paths, except for the final cycle $Q_1$, which lies in $A_1$, and hence~$W_0$, as required.  
} \label{spanpic5D}
\end{figure}

We can now prove the key result of this section. Given a subset~$W$ and a tree with many bare paths, we embed the tree so that it covers all the vertices in~$W$ while the uncovered vertices support a cycle. The statement of the lemma, again, caters to our universality requirements.

\lem\label{almostspan} Let $l,n,w,\Delta\in\N$ satisfy $20\log n/\log\log n\leq l\leq 20\log n$, and
\[
\frac{n}{10^{20}\log n}\exp\left(\frac{20\log n}{l}\right)\leq w\leq \frac{n}{10^5}.
\]
Suppose $W,B\subset [n]$ are disjoint sets satisfying $|W|\geq 100w$ and $|B|\leq n/4$. In the random graph $G=\GG(n,10^{50}\log n/n)$, there is, with probability $1-o(n^{-1})$, a subset $Z\subset V(G)\setminus B$ with $|Z|\leq 5n/10^4$ with the following property.

Let~$T$ be a tree with $\Delta(T)\leq \Delta$, which contains at least $n/2(10^4l)$ vertex disjoint bare paths with length $10l+1$. Suppose $V\subset V(G)$ satisfies $Z\subset V$ and $|V|=|T|+w$. Let $v\in V\setminus Z$ and $t\in V(T)$. Then there is a copy $S$ of~$T$ in $G[V]$ with~$t$ copied to $v$ so that $(V\setminus V(S))\subset W$ and  $G[V\setminus V(S)]$ is Hamiltonian.
\ma
\begin{rmk}
The Hamiltonicity of $G[V\setminus V(S)]$ in Lemma~\ref{almostspan} is only used in Cases C and D where after this lemma is used the vertices corresponding to $V\setminus V(S)$ are embedded as sections of bare paths. In Case B, these vertices are embedded as leaves, for which such Hamiltonicity is not used.
\end{rmk}
\pr[Proof of Lemma~\ref{almostspan}] 
\textbf{Revealing the random graph.}
Let $r=4\lfloor n/4(10^4l)\rfloor$. Reveal edges within the vertex set $[n]$ with probability $10^{50}\log n/2n$ to get the graph $G_1$. By Lemma~\ref{almostalmostspan2}, with probability $1-o(n^{-1})$ we can take a subset $A\subset [n]\setminus B$ and disjoint pairs $(x_i,y_i)$, $i\in[r/2]$, of vertices in $[n]\setminus (B\cup A)$, so that $|A|=3rl+w$ and the following holds with the set $X=\cup_{i=1}^{r/2}\{x_i,y_i\}$.
\stepcounter{capitalcounter}
\begin{enumerate}[label =\bfseries \Alph{capitalcounter}\arabic*]
\item Given any collection of $r/2$ disjoint pairs $(x_i,y_i)$, $i\in[r]\setminus[r/2]$, of vertices in $[n]\setminus (A\cup X)$, we can find vertex disjoint $x_i,y_i$-paths, $i\in[r]$, in $G_1$, with length $3l+1$ and interior vertices in $A$, so that the $w$ vertices in $A$ which are not covered by the paths lie in~$W$ and there is a cycle in~$G_1$ through exactly those $w$ uncovered vertices in $A$.
\label{P1}
\end{enumerate}

Take a set $Y\subset V(G_1)\setminus (B\cup A\cup X)$ with $|Y|=n/10^4$. Let $d=\log n/10^6\log\log n$ and $m=n/10^8d$. Reveal edges with probability $p=10^{20}\log n/n$ to get the graph~$G_2$. By Propositions~\ref{alledges} and~\ref{generalprops} and Lemma~\ref{generalexpand}, with probability $1-o(n^{-1})$,~$G_2$ is $m$-joined, has at most $pn^2$ edges and the following property holds.
\begin{enumerate}[label=\bfseries \Alph{capitalcounter}\arabic*]\addtocounter{enumi}{1}
\item If $U\subset V(G)$ and $|U|\leq 2m$, then $|N_{G_2}(U,Y)|\geq d|U|$. \label{P2}
\end{enumerate}

Reveal more edges with probability $10^{17}p$, and let~$G$ be the union of all the revealed edges. As $e(G_{2})\leq pn^2$, by Lemma~\ref{hamcycleinsubgraph}, with probability $1-o(n^{-1})$ the following property holds.
\begin{enumerate}[label=\bfseries \Alph{capitalcounter}\arabic*]\addtocounter{enumi}{2}
\item If $H\subset G_{2}$ is a $(|H|,2)$-expander, and $|H|\geq 4n/10^{6}$, then $V(H)$ supports a cycle in~$G$. \label{P4}
\end{enumerate}

Note that each edge has been revealed with probability at most $10^{50}\log n/2n+(10^{17}+1)p$, which, as required, is at most $10^{50}\log n/n$.

\textbf{Embedding the tree.}
Let $Z=Y\cup A\cup X\subset V(G)\setminus B$. We will show that the set $Z$ satisfies the lemma. Firstly,
\[
|Z|=\frac{n}{10^4}+(3rl+w)+r\leq \frac{n}{10^4}+\frac{3n}{10^4}+12l+w+r\leq\frac{5n}{10^4},
\]
as required. Suppose then that~$T$ is a tree, with $\Delta(T)\leq \Delta$, which has at least $n/2(10^4l)\geq r/2$ vertex disjoint bare paths with length $10l+1$, and let $V\subset V(G)$ be a set satisfying $Z\subset V$ and $|V|=|T|+w$. Suppose further that we have vertices $v\in V\setminus Z$ and $t\in V(T)$.

Disjointly, pick $r/2$ bare paths with length $5l-1$ and $r/4$ bare paths with length $10l+1$ from the tree~$T$, and remove them to get the forest $T'$. Replace these paths with dummy edges to get a tree,~$T''$ say, with maximum degree at most $\Delta$. Note that $|T|-|T''|=r(5l-2)/2+10rl/4=5rl-r$. By Property~\ref{P2}, and Proposition~\ref{uttriv}, $I(\{v\}\cup X)$ is $(d,m)$-extendable in $G_{2}[V\setminus A]$. Furthermore,
\[
|V\setminus A|-|T''|-|X|=|T|+w-|T''|-|A|-|X|= 5rl-r-3rl-r\geq rl\geq n/2(10^4)\geq 10dm.
\]
As~$G_2$ is $m$-joined, by Corollary~\ref{treebuild} applied to the subgraph $I(\{v\}\cup X)$ in the graph $G_2[V\setminus A]$, there is a copy $S''$ of $T''$ in $G_{2}[V\setminus(A\cup X)]$ so that $v$ is a copy of~$t$ and $S''\cup I(X)$ is $(d,m)$-extendable in $G_2[V\setminus A]$.

Remove from $S''$ the copies of the dummy edges in $T''$, to get a copy $S'$ of $T'$. By the definition of extendability, and as each vertex was in at most one of the removed paths, and hence at most one dummy edge, $S'\cup I(X)$ is $(d-1,m)$-extendable in $G_{2}[V\setminus A]$. Label vertices within $S''$ appropriately as $a_i$, $b_i$, $i\in [r/2]$, and $u_i$, $v_i$, $i\in[r/4]$, so that to extend $S'$ into a copy of~$T$ we need to find vertex disjoint $a_i,b_i$-paths, $i\in[r/2]$, of length $5l-1$ and vertex disjoint $u_i,v_i$-paths, $i\in[r/4]$, of length $10l+1$.

For each $i\in[r/2]$, use Corollary~\ref{trivial} (once the conditions are checked) to extend the subgraph $S'\cup I(X)$ by adding disjointly an $a_i,x_i$-path and a $y_i,b_i$-path in $G_{2}[V\setminus A]$, both with length $l-1$ and interior vertices in $V\setminus(A\cup V(S')\cup X)$, so that the final subgraph, $S$ say, is $(d-1,m)$-extendable in $G_{2}[V\setminus A]$. To check the conditions, observe that, as the paths have length $l-1\geq 20\log n/\log\log n-1\geq 2\lceil\log n/\log (d-2)\rceil+1$, the length condition holds in the applications of Corollary~\ref{trivial}. The working subgraphs to which we apply Corollary~\ref{trivial} have size at most
\[
|S|=|S'|+|X|+r(l-2)=|S'|+r(l-1)=|T|-(5rl-r)+r(l-1)=|T|-4rl.
\]
Therefore, the size condition for each application of Corollary~\ref{trivial} holds as
\begin{align}
|V\setminus A|-|S|= |T|+w-|A|-(|T|-4rl)=4rl-(|A|-w)=rl\geq 10dm+l.\label{tenforty}
\end{align}
As~$G_2$ is $m$-joined, we have thus all the conditions we need for the applications of Corollary~\ref{trivial}.

By~\eqref{tenforty}, we have $|V\setminus (V(S)\cup A)|=rl$. Using the definition of the $(d-1,m)$-extendability of~$S$ in $G_2[V\setminus A]$, for every set $U\subset V\setminus (V(S)\cup A)$ with $|U|\leq m$, we have
\[
|N_{G_2}(U,V\setminus (V(S)\cup A))|\geq |N'_{G_2[V\setminus A]}(U)\setminus V(S)|-|U|\geq (d-1)|U|-|U|\geq 2|U|,
\]
as $d\geq 10$. If $U\subset V\setminus (V(S)\cup A)$ with $m\leq |U|\leq rl/4$, then, as $G_2$ is $m$-joined,
\[
|N_{G_2}(U,V\setminus (V(S)\cup A))|\geq |V\setminus (V(S)\cup A)|-|U|-m\geq rl-rl/4-m\geq rl/2\geq 2|U|.
\]
Therefore, $G_{2}[V\setminus (V(S)\cup A)]$ is an $(rl,2)$-expander. By Property~\ref{P4}, there is a cycle, $Q$ say, in~$G$ with $V(Q)=V\setminus (V(S)\cup A)$.

Our final step closely resembles a typical Step~{\bfseries \Alph{STEPJcounter}$_j$}, which was used in the proof of Lemma~\ref{almostalmostspan2} and depicted in Figure~\ref{spanpic5D}. Take the cycle $Q$ and break it into $r/4$ vertex disjoint paths of length $4l-1$. Call the endvertices of each resulting path $u'_{i}$ and $v'_{i}$, $i\in[r/4]$. Take the $r/2$ disjoint pairs of vertices $(u_{i},u'_{i})$ and $(v_{i},v'_{i})$, $i\in[r/4]$, and, using Property~\ref{P1}, connect these pairs, and the pairs $(x_i,y_i)$, $i\in [r/2]$, with vertex disjoint paths with length $3l+1$ and interior vertices in $A$ so that the unused vertices in $A$ lie in~$W$ and form a cycle. For each $i\in[r/2]$, the $x_i,y_i$-path can be used, along with the $a_i,x_i$-path and the $y_i,b_i$-path we found, to extend the copy of $T'$ by replacing a deleted path of length $5l-1$. For each $i\in[r/4]$, the $u_i,u'_{i}$-path, the $u'_i,v'_i$-path from $Q$, and the $v_i',v_i$-path can be combined to extend the copy of $T'$ by replacing a deleted path of length $10l+1$. \nopagebreak This completes the copy of $T$ with the required properties.
\oof

\pagebreak[2]

\section{Case B}\label{8caseB}

The embedding for Case A works because we can find a matchmaker set for $V(G)$ (defined in Section~\ref{7caseA}) with size comfortably smaller than $\lambda(T)$, for any tree $T\in\TT_A(n,\Delta)$. Then, by the definition of $\l(T)$ we could find a large 20-separated collection of leaves, and embed part of~$T$ so that the neighbours of these leaves covered the matchmaker set. As remarked in Section~\ref{1boutline}, the minimum size for a matchmaker set in $V(G)$ is closely linked to the parameter $m$, the minimum integer $m$ such that $G$ is $m$-joined. Using calculations similar to those in Proposition~\ref{generalprops}, we can see that in $G=\GG(n,p)$, with $p=\Theta(\log n/ n)$, we almost surely have $m=\Theta(n\log\log n/\log n)$. We thus would need at least $\Theta(n\log\log n/\log n)$ neighbours of well separated leaves to cover a matchmaker set for $V(G)$. In Case~B, we can only guarantee at least $\Theta(n/k(T))$ well separated leaves. 

If we take a fixed subset $Z_1\subset [n]$, and reveal edges with probability $p=\Theta(\log n/n)$, then, by Proposition~\ref{generalprops}, we may expect an edge between any two subsets of $Z_1$ with size $\Theta(\log (|Z_1|p)/p)$. Thus, if $|Z_1|=(n/\log n)\exp(\Theta(\log n/k(T)))$, then we may expect an edge between any two subsets of $Z_1$ with size $\Theta(n/k(T))$. This will allow us to find a set $Z_1$ with such a size, along with two sets~$M$ and $L_0$, with $|M|=\Theta(n/k(T))$ and $|L_0|=\Theta(n/k(T))$, which function as our matchmaker sets for $Z_1$ (although, technically, by our definition~$M$ will be a matchmaker set for $Z_1\cup L_0$, while~$L_0$ will be a matchmaker set for~$M$). Therefore, with well chosen constants, for any tree in Case~B, we will be able to use Lemma~\ref{embedparentsfinal} to show that we can cover~$M$ with the neighbours of well separated leaves in~$T$. Importantly, the set $Z_1$ is large enough to allow us to apply Lemma~\ref{almostspan} with $l=\Theta(k(T))$.


More precisely, once we have found these sets $M$ and $L_0$, we will divide a tree $T\in\TT_B(n,\Delta)$ into three subtrees, $T_1$, $T_2$ and $T_3$, so that $|T_1|=\Theta(n)$, $|T_2|=\Theta(n)$, $|T_3|=\Theta(|Z_1|)$, $T_1$ has many disjoint bare paths with length $k=k(T)$ and~$T_2$ has a large 20-separated collection of leaves. In Stage 1, we remove some of these leaves from $T_2$, to get $T_2'$, and then embed $T_2'$ in such a way that the matchmaker set~$M$ is contained by the image of vertices in $T_2'$ which need leaves added in order to extend $T_2'$ to~$T_2$. In Stage 2, we use Lemma~\ref{almostspan} to extend the embedding to $T_1$ so that it covers the unused vertices not in $Z_1\cup L_0$. In Stage 3, we use vertices from $Z_1$ to embed $T_3$. In Stage 4, we complete the embedding of~$T_2$, and hence of~$T$, by attaching the final uncovered vertices as leaves. As previously, we reveal edges to find sets and properties we will need, working backwards through these stages, before taking a tree and carrying out the appropriate embedding.


\pr[Proof of Theorem~\ref{unithres} in Case B] For each tree $T\in \TT_B(n,\Delta)$, by definition, we have $10^3\log n/\log\log n\leq k(T)< 10^2\log n$. We will further divide the trees in this case based on the value of~$k(T)$, showing that, for each~$k$ in this range, with probability $1-o(n^{-1})$, all of the trees $T\in \TT(n,\Delta)$ with $k(T)=k$ can be found simultaneously in the random graph $\GG(n,10^{52}\Delta^2\log n/ n)$. Therefore, almost surely, all of the trees in $\TT_B(n,\Delta)$ can be found simultaneously in the random graph $\GG(n,10^{52}\Delta^2\log n/ n)$.

Let $k$ then satisfy $10^3\log n/\log\log n\leq k<10^2\log n$. We will reveal the edges within the vertex set $[n]$ in several rounds, with each possible edge present, in total, with probability at most $10^{52}\Delta^2\log n/ n$. In what follows we will use various different subsets while developing different properties. The construction of some of these sets are depicted in Figures~\ref{spanpic6A0} and~\ref{spanpic6AA}. The only sets that we need for the final embedding are depicted in Figure~\ref{spanpic6A}. 


\textbf{Revealing the random graph.}
Let $l=\lfloor (k-1)/10 \rfloor$, so that $50\log n/\log\log n\leq l\leq 10\log n$. Let
\[
w=\left\lceil \frac{n}{10^8\log n}\exp\left(\frac{20\log n}{l}\right)\right\rceil\leq \frac{2n}{10^8\log n}\exp\left(\frac{\log\log n}{2}\right)
= \frac{2n}{10^8\sqrt{\log n}}.
\]
Let $\lambda=n/10^{10}k$, so that, as $e^x\geq x$ for all $x\in \R$,
\begin{equation}\label{tenfourfive}
w\geq\frac{n}{10^8\log n}\exp\left(\frac{20\log n}{l}\right)\geq \frac{n}{10^8\log n}\left(\frac{20\log n}{l}\right)=\frac{2n}{10^7l}\geq 10\lambda.
\end{equation}
Take a subset $Z_0\subset [n]$ with $|Z_0|=250w$. Reveal edges within the vertex set $[n]$ with probability $p=10^{16}\Delta^2\log n/ n$ to get the graph $G_1$. Note that $wp\leq 10^9\Delta^2\exp(20\log n/l)$. Therefore, $\log(|Z_0|p)=\log(250wp)\leq \log(10^{12}\Delta^2)+ 20\log n/l\leq 40(\Delta+\log n/l)$. Let
\begin{equation}\label{tenstar}
m_{\l}=\frac{10}{p}\log(|Z_0|p)\leq\frac{400n}{10^{16}\Delta^2 \log n}\left(\Delta+\frac{\log n}{l}\right)\leq \frac{n}{10^{12}k\Delta}=\frac{\l}{10^2\Delta}.
\end{equation}
By Proposition~\ref{generalprops}, with probability $1-o(n^{-1})$ we have the following property.
\stepcounter{capitalcounter}
\begin{enumerate}[label =\bfseries \Alph{capitalcounter}\arabic*]
\item Any two disjoint subsets in $Z_0$ with size $m_{\l}$ have some edge between them in $G_1$.\label{notQ1}
\end{enumerate}
\textbf{Properties for Stage 4.}
Recalling (\ref{tenfourfive}), take disjoint subsets $A_1$, $A_2\subset Z_0$ so that $|A_1|=\l$ and $|A_2|=\l/2$. Consider the bipartite subgraph $H$ induced in $G_1$ by the vertex classes $A_1$ and $Z_0\setminus A_1$. Note that, by (\ref{tenstar}), $|A_1|,|A_2|\geq 50m_\l$. Using Property~\ref{notQ1}, and Proposition~\ref{neat2}, take $B_1\subset Z_0$ to be a subset, with $|B_1|\leq 2m_{\l}$, so that if $U\subset Z_0\setminus B_1$ and $|U|\leq m_\lambda$, then $|N_H(U, (A_1\cup A_2)\setminus B_1)|\geq|U|$. Let $M=A_1\setminus B_1$ and $L_0=A_2\setminus B_1$, so that we have the following property .
\begin{enumerate}[label=\bfseries \Alph{capitalcounter}\arabic*]\addtocounter{enumi}{1}
\item If $U\subset Z_0\setminus B_1$ and $|U|\leq m_\lambda$, then $|N_H(U, M\cup L_0)|\geq|U|$.\label{notQ2either}
\end{enumerate}
Note that $|M|-|L_0|\geq |A_1|-m_{\l}-|A_2|\geq \lambda/4$.

Let $Z_1=Z_0\setminus(A_1\cup A_2 \cup B_1)$, so that we have completed the construction shown in Figure~\ref{spanpic6A0}.
Suppose we have a subset $U_0\subset Z_1$ with $|U_0|=|M|-|L_0|$. If $U\subset U_0\cup L_0$ and $|U|\leq m_\lambda$, then, by Property~\ref{notQ2either}, $|N_{G_1}(U,M)|=|N_H(U,M\cup L_0)|\geq |U|$. If $U\subset M$ and $|U|\leq m_\lambda$, then, by Property~\ref{notQ2either},
\[
|N_{G_1}(U,U_0\cup L_0)|\geq |N_{G_1}(U,L_0)|=|N_H(U,M\cup L_0)|\geq |U|.
\]
Therefore, by Property~\ref{notQ1} and Lemma~\ref{matchings}, there is a matching between $U_0\cup L_0$ and~$M$. Thus, we have the following property.
\begin{enumerate}[label=\bfseries \Alph{capitalcounter}\arabic*]\addtocounter{enumi}{2}
\item If $U\subset Z_1$ and $|U|=|M|-|L_0|$, then there is a matching between $U\cup L_0$ and~$M$.\label{notQ2}
\end{enumerate}

\begin{figure}[t]
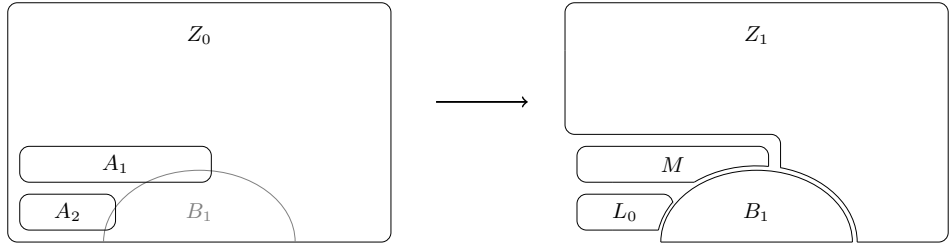

\centering
    \begin{subfigure}[b]{0.35\textwidth}
        \centering
        \resizebox{\linewidth}{!}{
            \input{spanpic6A01}
        }
\label{spanpic6A01}
    \end{subfigure}
\begin{subfigure}[b]{0.1\textwidth}
        \centering
        \resizebox{\linewidth}{!}{
{\scalefont{0.4}
\begin{tikzpicture}[scale=0.7]
\draw [white] (0,0) -- (0,6) -- (2,6) -- (2,0);
\draw [thick,->] (0,3.75) -- (2,3.75);
\end{tikzpicture}
}
}
    \end{subfigure}
    \begin{subfigure}[b]{0.35\textwidth}
        \centering
        \resizebox{\linewidth}{!}{
            \input{spanpic6A02}
        }
\label{spanpic6A02}
    \end{subfigure}

\vspace{-0.5cm}

\caption{Finding subsets $Z_1$, $M$, $L_0$ and $B_1$ of $Z_0$ in the proof of Theorem~\ref{unithres} in Case~B.} \label{spanpic6A0}
\end{figure}

Later, when we embed a tree, the set~$M$ will be covered by the neighbours of leaves. We will take vertices from our tree which need leaves attached to them and use them to cover~$M$. The set $L_0$ will be reserved until the end of our embedding and used, along with the final uncovered vertices, to attach leaves to the set~$M$ using Property~\ref{notQ2}.


\textbf{Properties for Stage 3.}
Note that, as $Z_1=Z_0\setminus(A_1\cup A_2 \cup B_1)$, we have
\[
|Z_1|\geq |Z_0|-|A_1|-|A_2|-|B_1|\geq 250w-2\l\geq 240w.
\]
Take a set $Z_2\subset Z_1$ with $|Z_2|=20w\geq 200\l\geq  10^4\Delta m_\l$, using (\ref{tenfourfive}) and~(\ref{tenstar}). Using Property~\ref{notQ1} and Proposition~\ref{neat}, take a subset $B_2\subset Z_1$, with $|B_2|\leq m_\l$, so that, for each set $U\subset Z_1\setminus B_2$ with $|U|\leq 2m_{\l}$, we have that $|N(U,Z_2\setminus B_2)|\geq 2\Delta|U|$. Let $Z_3=Z_2\setminus B_2$, so that $|Z_3|\leq 20w$. Let $Z_4=Z_1\setminus (Z_3\cup B_2)=Z_1\setminus (Z_2\cup B_2)$, and pick $v_0\in Z_3$. 

Suppose $U_0\subset Z_4$. For any subset $U\subset U_0\cup Z_3$ with $1\leq |U|\leq 2m_\lambda$, 
\[
|N(U,U_0\cup Z_3)\setminus\{v_0\}|\geq |N(U,Z_3)\setminus\{v_0\}|\geq 2\Delta|U|-1\geq \Delta|U|.
\]
Therefore, by Proposition~\ref{uttriv}, we have the following property.
\begin{enumerate}[label=\bfseries \Alph{capitalcounter}\arabic*]\addtocounter{enumi}{3}
\item If $U_0\subset Z_4$, then $I(\{v_0\})$ is $(\Delta,m_\lambda)$-extendable in $G_1[U_0\cup Z_3]$.\label{Q3}
\end{enumerate}

\textbf{Properties for Stage 2.}
Note that $|Z_4|= |Z_1|-|Z_3|-|B_2|\geq 200w$. Divide the set $Z_4$ into $Z_5$ and $Z_6$, so that $|Z_5|,|Z_6|\geq 100w$ and we have completed the construction shown in Figure~\ref{spanpic6AA}.
Let $Y=Z_3\cup Z_5\cup M\cup L_0\subset Z_0$, and note that
\begin{equation}\label{Ybound}
|Y|\leq |Z_0|= 250w=o(n).
\end{equation}
Reveal more edges among the vertex set $[n]$ with probability $10^{50}\log n/ n$ to get the graph $G_2$. With probability $1-o(n^{-1})$, by Lemma~\ref{almostspan} applied to the graph $G_2$ with the sets $W=Z_6\subset Z_4$ and~$B=Y$, there is a set $W_0\subset V(G_2)\setminus Y$ with $|W_0|\leq 5n/10^4$ which satisfies the following property.
\begin{enumerate}[label=\bfseries \Alph{capitalcounter}\arabic*]\addtocounter{enumi}{4}
\item If $V\subset [n]$ and $v\in V\setminus W_0$, with $W_0\subset V$, then, given any tree~$T$ with $|T|=|V|-w$ which contains at least $n/2(10^4)l$ bare paths of length $10l+1$ and a chosen vertex $t\in V(T)$, there is a copy $S$ of~$T$ in $G_2[V]$ so that $(V\setminus V(S))\subset Z_4$, and $t$ is copied to $v$.\label{Q4}
\end{enumerate}

\begin{figure}[b]
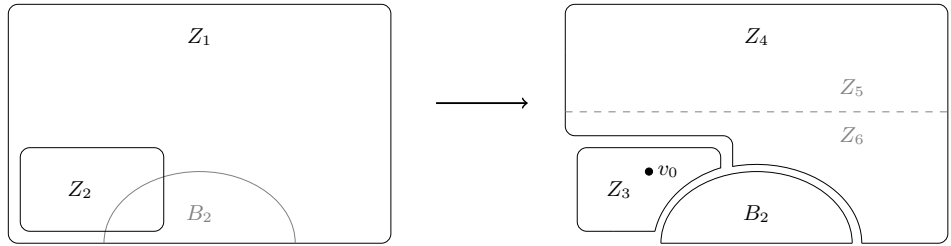

\centering
    \begin{subfigure}[b]{0.35\textwidth}
        \centering
        \resizebox{\linewidth}{!}{
            \input{spanpic6AA1}
        }
\label{spanpic6AA1}
    \end{subfigure}
\begin{subfigure}[b]{0.1\textwidth}
        \centering
        \resizebox{\linewidth}{!}{
{\scalefont{0.4}
\begin{tikzpicture}[scale=0.7]
\draw [white] (0,0) -- (0,6) -- (2,6) -- (2,0);
\draw [thick,->] (0,3.75) -- (2,3.75);
\end{tikzpicture}
}
}
    \end{subfigure}
\begin{subfigure}[b]{0.35\textwidth}
        \centering
        \resizebox{\linewidth}{!}{
            \input{spanpic6AA2}
        }
\label{spanpic6AA2}
    \end{subfigure}

\vspace{-0.5cm}

\caption{Finding subsets $Z_3$, $Z_4$, $Z_5$ and $Z_6$ of $Z_1$ in the proof of Theorem~\ref{unithres} in Case~B.} \label{spanpic6AA}
\end{figure}

\textbf{Properties for Stage 1.}
Let $X=[n]\setminus(W_0\cup Y)$, noting that, by (\ref{Ybound}), $|X|\geq n-5n/10^4-o(n)$. Let
\[
d=\log n/\log\log n\;\;\;\text{ and }\;\;\;m=n/10^6d.
\]
Reveal more edges among the vertex set $[n]$ with probability $10^{20}\log n/ n$ to get the graph $G_3$. By Corollary~\ref{generalexpandcor}, with probability $1-o(n^{-1})$ we have the following property.
\begin{enumerate}[label=\bfseries \Alph{capitalcounter}\arabic*]\addtocounter{enumi}{5}
\item The subgraph $I(M\cup\{v_0\})$ is $(d,m)$-extendable in $G_3[X\cup M\cup \{v_0\}]$. \label{Q5}
\end{enumerate}
 By Proposition~\ref{generalprops}, with probability $1-o(n^{-1})$ we have the following property.
\begin{enumerate}[label=\bfseries \Alph{capitalcounter}\arabic*]\addtocounter{enumi}{6}
\item The graph $G_3$ is $m$-joined.\label{Q6}
\end{enumerate}
Let $G=G_1\cup G_2\cup G_3$, noting that the properties we have accumulated continue to hold under the addition of edges, and, hence, hold for the graph $G$. Each possible edge has been revealed with total probability at most $10^{52}\Delta^2\log n/ n$. We are now ready to embed any tree $T\in\TT_B(n,\Delta)$ with $k(T)=k$. 

\begin{figure}[t!]
\centering
    \begin{subfigure}[b]{0.55\textwidth}
        \centering
        \resizebox{\linewidth}{!}{
            \input{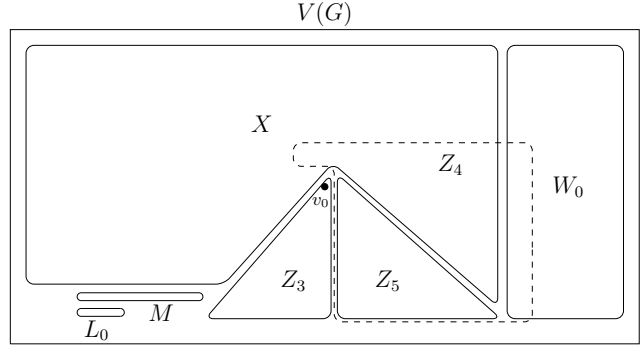}
        }
\label{spanpic6A1}
    \end{subfigure}

\vspace{-0.4cm}

\caption{Important sets found in the random graph $G$ in the proof of Theorem~\ref{unithres} in Case~B, along with the vertex $v_0\in Z_3$. The sets $L_0$, $M$, $Z_3$, $Z_5$, $W_0$ and $X$ partition $V(G)$. As well as keeping the set $Z_4$ in mind, we remember that $Z_3\cup Z_4\subset Z_1$.} \label{spanpic6A}
\end{figure}

\medskip

\textbf{Embedding the tree: split the tree.}
Let $T\in \TT_B(n,\Delta)$ satisfy $k(T)=k$, recalling that
\[
10^3\log n/\log\log n\leq k<10^2\log n.
\] 
By the definition of $k(T)$,~$T$ has at least $n/90k$ vertex disjoint bare paths of length $k$. Using Corollary~\ref{dividepath}, divide~$T$ into two trees~$T_1$ and~$R_1$ intersecting on a single vertex $t_1$ so that~$T_1$ and~$R_1$ each contain at least $n/270k-1\geq n/2(10^4l)$ vertex disjoint bare paths of length $k$. Note that, due to the paths they contain, $|T_1|,|R_1|\geq n/300$. Say, without loss of generality, that $|R_1|\geq n/2$. Using Proposition~\ref{littletree}, divide $R_1$ into two trees $T_2$ and $T_3$ which intersect on a single vertex $t_2$ so that $t_1\in V(T_2)$ and~$30w\leq |T_3|\leq 90w$. The structure found in $T$ is depicted in Figure~\ref{spanpic6B}, and the following embedding is depicted in Figure~\ref{spanpic6C}.

\begin{figure}[b]
\centering
    \begin{subfigure}[b]{0.5\textwidth}
        \centering
        \resizebox{\linewidth}{!}{
            \input{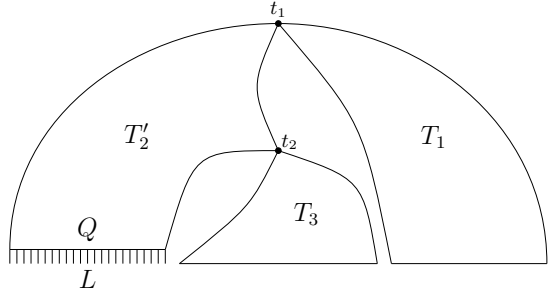}
        }
\label{spanpic6B1}
    \end{subfigure}

\vspace{-0.6cm}

\caption{The structure found in the tree $T\in \TT_B(n,\Delta)$.} \label{spanpic6B}
\end{figure}

\textbf{Stage 1.} As $k(T)=k$,~$T$ does not contain $n/90(k+1)$ vertex disjoint bare paths with length $k+1$. Therefore, $T_2$ does not contain $n/90(k+1)+1$ vertex disjoint bare paths with length $k+1$. Note that $|T_2|/40(k+1)\geq (n/2 -90w)/40(k+1)>n/90(k+1)+1$. If $T_2$ has at least $|T_2|/100$ leaves, then we can find a 20-separated collection of $n/400\Delta^{20}\geq 10\l$ leaves, for sufficiently large~$n$, because of the maximum degree of $T_2$.  If $T_2$ does not have at least $|T_2|/100$ leaves, then, by Lemma~\ref{farleaves}, $T_2$ must contain a 20-separated collection of at least $n/40(k+1)\geq 10\l$ leaves. Therefore, in either case, we can let $L$ be a 20-separated set of $9\l$ leaves of $T_2$ which does not contain $t_1$ or $t_2$. Let $Q=N_{T_2}(L)$, and let $T_2'=T_2-L$. We have then $|Q|=9\l\geq 9|M|$. As $X=V(G)\setminus (M\cup L_0\cup Z_3\cup Z_5\cup W_0)$ and $|M\cup L_0\cup Z_3\cup Z_5\cup W_0|\leq |Z_0|+|W_0|\leq n/10^3$, we have $|X|\geq n-n/10^3$, and thus
\begin{equation}\label{elevenalmost}
|T_2|\leq n-|T_1|\leq n-n/300\leq |X|-10dm-\log n.
\end{equation}
Therefore, using Lemma~\ref{embedparentsfinal} and Properties~\ref{Q5} and~\ref{Q6}, we can find a $(d,m)$-extendable copy of $T_2'$ in $G[X\cup M\cup \{v_0\}]$, so that the copy of $Q$ contains~$M$ and $t_2$ is copied to $v_0$. Using (\ref{elevenalmost}), Property~\ref{Q6}, and Corollary~\ref{mextend}, for each leaf $t\in L$ adjacent to a vertex not copied onto~$M$, extend the copy of $T_2'$ to cover $t$, while maintaining the $(d,m)$-extendability. Call the resulting subgraph $S'_2$, so that to make $S'_2$ into a copy of $T_2$ we need to attach a leaf to each of the vertices in~$M$.

\textbf{Stage 2.} Let $X'=X\setminus V(S'_2)$ and let $v_1$ be the copy of $t_1$ in $S'_2$. The vertices in $S'_2$ cover~$M$, and $|S'_2|=|T_2|-|M|$. As $V(G)=V(S'_2)\cup L_0\cup Z_3\cup Z_5\cup X'\cup W_0$, we have
\[ 
|X'\cup W_0|=n-(|T_2|-|M|)-|L_0|-|Z_3|-|Z_5|= |T_1|+|T_3|-2+(|M|-|L_0|)-|Z_3|-|Z_5|.
\]
Thus, as $|T_3|\geq 30w$ and $|Z_3|\leq 20w$, we have $|X'\cup W_0|+|Z_5|\geq |T_1|+30w-2-20w\geq |T_1|+w$. Furthermore, as $|Z_5|\geq 100w$ and $|T_3|\leq 90w$, $|X'\cup W_0|\leq |T_1|+90w+\l-100w\leq |T_1|$. Therefore, we may take vertices from~$Z_5$ and add them to $X'\cup W_0\cup \{v_1\}$ to get $V$ so that $|V|=|T_1|+w$.
The tree $T_1$ contains at least $n/2(10^4l)$ vertex disjoint bare paths of length $k\geq 10l+1$. Thus, by Property~\ref{Q4} there is a copy of~$T_1$,~$S_1$ say, in $G[V]$ so that $(V\setminus V(S_1))\subset Z_4$ and $t_1$ is copied to~$v_1$. The tree $S_1\cup S_2'$ covers every vertex in $V(G)$ except for the vertices in $L_0$, $Z_3\setminus\{v_0\}$, $Z_5\setminus V$ and $V\setminus V(S_1)$, and needs a leaf added to each vertex in $M$ to make it into a copy of $T_1\cup T_2$.

\textbf{Stage 3.} Let $U_0=(Z_5\setminus V)\cup (V\setminus V(S_1))\subset Z_4$. By Property~\ref{Q3}, $I(\{v_0\})$ is $(\Delta,m_\lambda)$-extendable in the graph $G[U_0\cup Z_3]$. The subgraph $S_1\cup S'_2$ needs a copy of $T_3$ and $|M|$ leaves added appropriately to get a copy of~$T$. Therefore, 
\begin{equation}\label{obvious}
|S_1\cup S'_2|=n-|T_3|-|M|+1.
\end{equation}
The set $U_0\cup Z_3$ contains the vertex $v_0\in V(S_1\cup S'_2)$ and all the vertices in $V(G)$ which are not in~$S_1$ or~$S'_2$, except for those in $L_0$. Therefore, using (\ref{obvious}), and then (\ref{tenstar}),
\begin{align}
|U_0\cup Z_3|&=1+(n-|S_1\cup S_2'|)-|L_0|=|T_3|+|M|-|L_0|\label{elevenLAB}
\\
&\geq|T_3|+ \l/4\geq |T_3|+10\Delta m_\lambda.\nonumber
\end{align}
Using Property~\ref{notQ1} and Corollary~\ref{treebuild}, find a copy $S_3$ of $T_3$ in $G[U_0\cup Z_3]$ with $t_2$ copied to $v_0$. 

\textbf{Stage 4.} Finally, let $U_1=(U_0\cup Z_3)\setminus V(S_3)$, so that $U_1\subset Z_4\cup Z_3\subset Z_1$ and, using (\ref{elevenLAB}), $|U_1|=|M|-|L_0|$. By Property~\ref{notQ2}, we can find a matching from~$M$ into $U_1\cup L_0$. Use this matching to add a leaf to each vertex in~$M$, making $S'_2$ into a copy of $T_2$. Together with~$S_1$ and~$S_3$, this gives a copy of~$T$.
\oof

\begin{figure}[p]
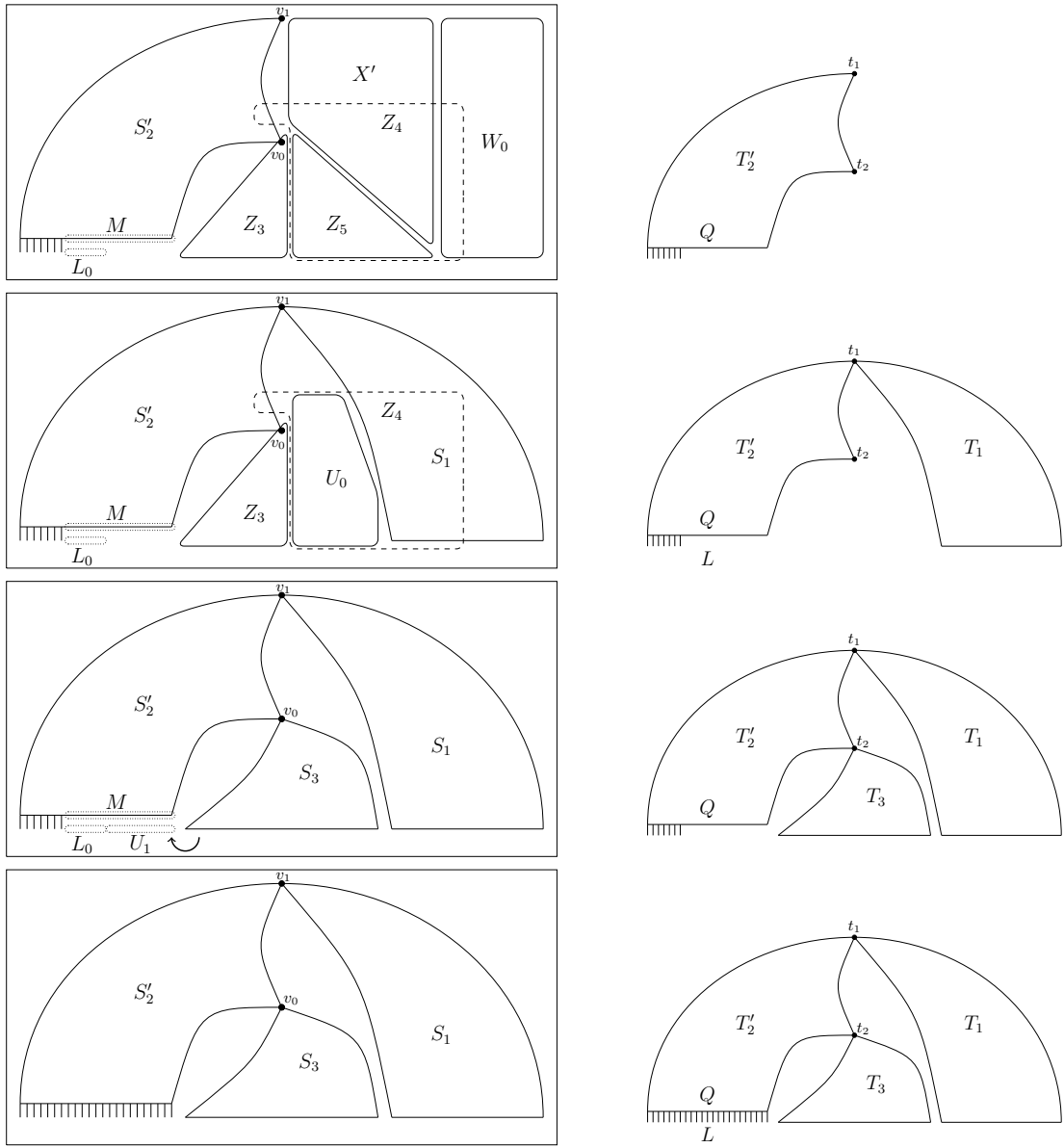

\centering
    \begin{subfigure}[b]{0.5\textwidth}
        \centering
        \resizebox{\linewidth}{!}{
            \input{new6c1}
        }
\label{new6c1}
    \end{subfigure}
\hspace{0.5cm}
    \begin{subfigure}[b]{0.4\textwidth}
        \centering
        \resizebox{\linewidth}{!}{
            \input{spanpic6D1}
        }
\label{spanpic6D1}
    \end{subfigure}

\vspace{-0.3cm}

    \begin{subfigure}[b]{0.5\textwidth}
        \centering
        \resizebox{\linewidth}{!}{
            \input{new6c2}
        }
\label{new6c2}
    \end{subfigure}
\hspace{0.5cm}
    \begin{subfigure}[b]{0.4\textwidth}
        \centering
        \resizebox{\linewidth}{!}{
            \input{spanpic6D2}
        }
\label{spanpic36B2}
    \end{subfigure}

\vspace{-0.3cm}


    \begin{subfigure}[b]{0.5\textwidth}
        \centering
        \resizebox{\linewidth}{!}{
            \input{new6c3}
        }
\label{new6c3}
    \end{subfigure}
\hspace{0.5cm}
    \begin{subfigure}[b]{0.4\textwidth}
        \centering
        \resizebox{\linewidth}{!}{
            \input{spanpic6D3}
        }
\label{spanpic36C2}
    \end{subfigure}

\vspace{-0.3cm}

    \begin{subfigure}[b]{0.5\textwidth}
        \centering
        \resizebox{\linewidth}{!}{
            \input{new6c4}
        }
\label{new6c4}
    \end{subfigure}
\hspace{0.5cm}
    \begin{subfigure}[b]{0.4\textwidth}
        \centering
        \resizebox{\linewidth}{!}{
            \input{spanpic6D4}
        }
\label{spanpic36D2}
    \end{subfigure}

\vspace{-0.5cm}

\caption{On the left, from top to bottom, following on from the structure in Figure~\ref{spanpic6A}, we show the situation in the embedding for a tree $T$ in Case~$B$ immediately after each stage, from Stage~1 to Stage~4. 
On the right are the subgraphs of $T$, as shown in Figure~\ref{spanpic6B}, that have been embedded in the matching pictures on the left. The arrow in the third picture on the left represents the vertices in $(U_0\cup Z_3)\setminus V(S_3)$ being used to form $U_1$.}
\label{spanpic6C}
\end{figure}

\section{Connecting vertex pairs in subsets}\label{6connectionlemmas}

As remarked in Section~\ref{1boutline}, for Cases C and D we need to solve several problems to develop a new embedding. Building a reservoir, defined precisely, and built, in Section~\ref{9absorbC}, is an important part of this. In this section, we choose a set to build into a reservoir, finding a small set in which we can connect many different pairs of vertices (chosen arbitrarily from a further subset) with disjoint paths of some specified length (in our use, length $6\log n$). We will want such paths to cover disjointly a positive proportion of the vertices in the reservoir (in our use, at least a fraction $1/10^9$ of the vertices). The subset used for our reservoir is shown to exist in Lemma~\ref{connectlowdense}, the key result of this section, for which we use the following lemma. Given only a simple expansion property in a graph, Lemma~\ref{connectlog} finds a subset with the kind of connection property just discussed.

\lem \label{connectlog} Let $n\in \N$ and $m=n/200$, and suppose $G$ is an $m$-joined graph with $n$ vertices. Then, for sufficiently large $n$, there is some set $Z\subset V(G)$ with $|Z|\geq n/4$, such that, for any collection of disjoint pairs $(x_i,y_i)$, $i\in [n/8\log n]$, of vertices in $Z$, there exist disjoint $x_i,y_i$-paths in $G$, with length $2\log n$ and internal vertices in $V(G)\setminus Z$.
\ma
\ifdraft
\else
\pr
Let $W$ be a set of $n/2=100m$ vertices in $G$. By Proposition~\ref{neat}, as $G$ is $m$-joined, we can take $B\subset V(G)$ to be a subset, with $|B|\leq m$, so that, for every set $U\subset V(G)\setminus B$ with $|U|\leq 2m$, we have $|N(U,W\setminus B)|\geq 4|U|$. We will show that $Z=V(G)\setminus(W\cup B)$ satisfies the requirements in the lemma. Firstly, note that $|Z|\geq n-|W\cup B|\geq n-n/2-m\geq n/4$, as required.

Suppose then we have disjoint pairs $(x_i,y_i)$, $i\in[n/8\log n]$, of vertices in $Z$. Let $X$ be the set consisting of the vertices in these pairs, and let $H=G[X\cup (W\setminus B)]$. If $U\subset V(H)$ and $|U|\leq 2m$, then $|N_H(U)\setminus X|\geq |N_G(U,W\setminus B)|\geq 4|U|$. As $I(X)$ has maximum degree $0$, by Proposition~\ref{uttriv}, $I(X)$ is a $(4,m)$-extendable subgraph of~$H$.

Repeatedly applying Corollary~\ref{trivial} (once we have checked the details), we can find the disjoint required $x_i,y_i$-paths by starting with the subgraph $I(X)$ and adding, for each $i$, an $x_i,y_i$-path with length $2\log n$ while keeping the subgraph $(4,m)$-extendable in~$H$. These paths have internal vertices in $V(H)\setminus X\subset V(G)\setminus Z$. As the paths we create have length $2\log n\geq 2\lceil \log(2m)/\log 3\rceil -1$, for sufficiently large $n$, the length condition holds. The final subgraph with all the paths added has size at most $(2\log n+1)(n/8\log n)\leq n/4+n/\log n\leq |H|-40m-2\log n$, for sufficiently large~$n$, so therefore the size condition also holds in each application of Corollary~\ref{trivial}. 
\oof
\fi

Lemma~\ref{connectlog} is a good start, but we want a subset with size $n'=\Theta(n/k(T))$ in a random graph $\GG(n,\Theta(\log n/n))$ with good connection properties, which may be as small as $n'=\Theta(n\log\log n/\log^2 n)$ in Case D.
Large graphs satisfying the condition in Lemma~\ref{connectlog} can be seen to have at least $50n$ edges.
The random graph $\GG(n,\Theta(\log n/n))$ will, almost surely, contain no vertex set with size $n'=\Theta(n\log\log n/\log^2 n)$ which contains more than $3n'/2$ edges. Hence, we can typically find no subgraph with~$n'$ vertices to which we can apply Lemma~\ref{connectlog}. 
Instead, we will find within $\GG(n,\Theta(\log n/n))$ a subgraph with $n'$ vertices, which looks like a graph satisfying the conditions of Lemma~\ref{connectlog}, but in which each edge has been replaced by a disjoint path of length 3 (i.e.\ a 2-subdivision of such a graph). Replacing the edges by these paths lowers the average degree of the graph and increases the girth, making the graph easier to find in our random graph.
Importantly, if we are careful, then adding the extra vertices in the middle of paths will only increase the number of vertices by a constant factor (as the graph we subdivide can have constant average degree).

\lem\label{connectlowdense} Let $n,\lambda \in \N$ satisfy $10^6n/\log^3 n\leq \l\leq n/10\log n$. 
With probability $1-o(n^{-1})$, the random graph $G=\GG(n,100\log n/ n)$ contains subsets $Z_2\subset Z_1\subset V(G)$ with $|Z_1|\leq 10^9\l$ and $|Z_2|\geq \l$ so that the following holds. Given any $\l/2\log n$ disjoint pairs $(x_i,y_i)$, $i\in [\l/2\log n]$, of vertices in~$Z_2$ there exist disjoint $x_i,y_i$-paths in $G[Z_1]$, with length $6\log n$ and internal vertices in $Z_1\setminus Z_2$.
\ma
\ifdraft
\else
\pr 
Take distinct vertices $a_1,\ldots,a_{4\l}$ in $ V(G)$. Reveal edges among the vertex set $[n]$ with probability $90\log n/ n$ to get the graph $G_1$. By Lemma~\ref{logmatching}, with probability $1-o(n^{-1})$, we can find disjoint sets $A_1,\ldots, A_{4\l}$ in $V(G_1)\setminus (\cup_{i\in[4\lambda]}\{a_i\})$, so that $A_i\subset N(a_i)$ and $|A_i|=\log n$, for each $i\in[4\l]$.

Reveal more edges with probability $p=10^7/\l\log^2 n\leq 10\log n/n$ to get the graph $G_2$, and let $G=G_1\cup G_2$. Consider the auxillary graph~$H$ on the vertex set $[4\l]$, where $ij$ is an edge in~$H$ if there is an edge between $A_i$ and $A_j$ in $G_2$. The graph~$H$ has edges present independently with probability $q=1-(1-p)^{\log^2 n}$, so that $p\log^2n/2\leq q\leq p\log^2 n= 10^7/\l$.
Therefore, by Proposition~\ref{AKS1}, with probability $1-o(|H|^{-2})=1-o(n^{-1})$, there are at least $\l$ edges between any two disjoint sets with size $\l/100$ in~$H$. By Proposition~\ref{alledges}, with probability $1-o(|H|^{-2})=1-o(n^{-1})$, the total number of edges in~$H$ is at most $10^7(16\l)$.
By Corollary~\ref{fewpathslengthtwo}, with probability $1-o((\l\log n)^{-2})=1-o(n^{-1})$, the graph $G_2[\cup_{i\in [4\l]}A_i]$ contains at most $(4\l\log n)^3p^2\leq 10^{16} \l/\log n\leq \l/2$ paths with length two. 

Delete an edge from each path with length two in $G_2[\cup_{i\in [4\l]}A_i]$ from~$G_2$ to get the graph~$G_2'$. Let $H'\subset H$ have an edge $ij$ exactly when there is an edge between $A_i$ and $A_j$ in $G'_2$.
Given any two disjoint sets $A,B\subset[4\l]$, each containing at least $\l/100$ vertices, as $e(H)-e(H')\leq \l/2$, there must be at least $\l-\l/2>0$ edges between $A$ and $B$ in $H'$. For sufficiently large~$n$, by applying Lemma~\ref{connectlog}, we can find $Z\subset [4\l]$ with $|Z|\geq \l$, such that, for any collection of disjoint pairs $(x_i,y_i)$, $i\in [\l/2\log n]$, of vertices in~$Z$, there are vertex disjoint $x_i,y_i$-paths in $H'$ with length $2\log n$ and internal vertices in $V(H')\setminus Z$.

Now, for each edge $e=ij$ in $H'$, take a path $P_e$ of length three between $a_i$ and $a_j$, say $a_ic_ed_ea_j$, so that $c_e\in A_i$, $d_e\in A_j$ and $c_ed_e\in E(G_2')$. This is possible for each edge $e$ by the definition of $H'$. As $G_2'$ has no paths with length two, these paths $P_e$, $e\in V(H')$, are internally vertex disjoint. Let $Z_1=\{a_i:i\in[4\l]\}\cup(\cup_{e\in E(H')}V(P_e))$. As $e(H')\leq e(H)\leq 10^7(16\l)$, we have $|Z_1|\leq 10^9\l$. Let $Z_2=\{a_i:i\in Z\}$, so that $|Z_2|\geq \l$.

The sets $Z_1$ and $Z_2$ have the properties required by the lemma. Indeed, given any $\l/2\log n$ disjoint pairs of vertices in $Z_2$, we can find vertex disjoint paths with length $2\log n$ between the corresponding vertices in $H'$. Replacing each edge $e$ in these paths by the path $P_e$ gives the required vertex disjoint paths with length $6\log n$ in $G$.
\oof
\fi

\section{The reservoir in Cases C and D}\label{9absorbC}

In this section, we develop the tools used to make a set into a reservoir in both Cases C and~D, in the sense of the following definition.

\de Given a graph~$G$, a vertex set $R\subset V(G)$, a tree $T$ and vertices $v\in V(G)\setminus R$ and $t\in V(T)$, we say $R$ is \emph{made into a reservoir} by $(G,v,T,t)$ if $|G|-|T|=|R|/2$ and, for every subset $U\subset R$ with $|U|=|R|/2$, there is a copy of $T$ in $G-U$ with $t$ copied to $v$.
\fn

That is, we may delete one half of the vertices in $R$ and find a copy of $T$ with $t$ copied to $v$ which covers exactly the remaining vertices in~$G$. We think of this as partially embedding $T$ in $G-R$ so that, given any such set $U\subset R$, we can \emph{absorb} the vertices in $R\setminus U$ into the partial embedding to create a copy of $T$ in $G-U$.
 In order to make sets into reservoirs we use the following structures, which we call \emph{$\l$-devices}. 

\de Let $n,\lambda\in \N$ satisfy $n\log\log n/\log^2 n\leq \l\leq n/100\log n$. Suppose we have a graph~$G$ with $n$ vertices and a set $R\subset V(G)$ with $|R|=6\l$. 

\begin{itemize}
\item If $\l> n\log\log n/\log^2 n$, then a \emph{$\l$-device for $R$ in~$G$} is a pair of subsets $(W_0,W_1)$ in $V(G)\setminus R$ with $9\l\leq |W_0|\leq 10^3\l$ and $|W_1|=|W_0|-3\l$ such that, for every subset $U\subset R$ with $|U|=|R|/2=3\l$, there is a matching between $W_0$ and $W_1\cup U$ in~$G$.
\item  If $\l=n\log\log n/\log^2 n$, then a \emph{$\l$-device for $R$ in~$G$} is a pair $(X,\{(x_i,y_i):i\in[9\l]\})$ consisting of a vertex set $X\subset V(G)\setminus R$ and disjoint pairs $(x_i,y_i)$, $i\in[9\l]$, of vertices in~$X$, such that the following holds with $l=10^4\log^2 n/(\log\log n)^2$. We have $|X|=9\l l-3\l$ and, given any subset $U\subset R$ with $|U|=|R|/2=3\l$, there is a set of $9\l$ vertex disjoint $x_i,y_i$-paths in $X\cup U$ with length $l-1$. In fact, therefore, these paths cover the set $X\cup U$.
\end{itemize}
Given a $\l$-device $D$ for $R$, if $\l>n\log\log n/\log^2 n$ and $D=(W_0,W_1)$, then let $V(D)=W_0\cup W_1$, and, if $\l=n\log\log n/\log^2 n$ and $D=(X,\{(x_i,y_i):i\in[9\l]\})$, then let $V(D)=X$.
\fn

Note that when $D$ is a $\l$-device, and $n\log\log n/\log^2 n\leq \l\leq n/100\log n$, we have
\begin{equation}\label{VDsize}
|V(D)|\leq \max\{2(10^3\l),9(n\log\log n/\log^2 n)\cdot (10^4\log^2 n/(\log\log n)^2)\}=o(n).
\end{equation}

Given a $\l$-device for a set $R$, we can use it to make the set $R$ into a reservoir under certain conditions. How a $\l$-device is found, and how it is used to make a set into a reservoir, is the only difference in our embeddings in Case~C and Case~D. The following lemma has two instances. The first (when $\l>n\log\log n/\log^2 n$) uses the leaves of a tree to make a set into a reservoir, and is used in Case~C. The second (when $\l=n\log\log n/\log^2 n$) uses the bare paths in a tree to make a set into a reservoir, and is used in Case~D.

\lem\label{reservoirfromdevice} Let $\Delta\in \N$, and let $n\in \N$ be sufficiently large. Let $d=2\log n/\log\log n$ and $m\leq n/100d$. Let $\lambda\in \N$, with $n\log\log n/\log^2 n\leq \l\leq n/100\log n$. Suppose a graph~$G$ contains a $\l$-device $D$ for $R\subset V(G)$, where $|R|=6\lambda$ and $I(R\cup V(D))$ is $(d,m)$-extendable in~$G$, and suppose that~$G$ is $m$-joined.

Let $T$ be a tree, with $|T|\leq |G|-|V(D)|-|R|-10dm-\log n$ and $\Delta(T)\leq \Delta$, let $t\in V(T)$ and suppose that the following properties hold.
\begin{itemize}
\item If $\l> n\log\log n/\log^2n$, then $\lambda(T)\geq 10^4\lambda$.
\item If $\l=n\log\log n/\log^2n$, then $T$ contains a collection of at least $9\l+1$ vertex disjoint bare paths with length $10^4\log^2 n/(\log\log n)^2+2\log n$.
\end{itemize}
Then, there exists a subgraph $H\subset G$ with $R\subset V(H)$ and a vertex $v\in V(H)\setminus R$ so that $R$ is made into a reservoir by $(H,v,T,t)$.
\ma
Note that in Lemma~\ref{reservoirfromdevice} we do not require that $|G|=n$.

\pr Firstly, suppose $\l> n\log\log n/\log^2n$. Let the $\l$-device $D$ for $R$ be $(W_0,W_1)$, and pick some $v\in V(G)\setminus (R\cup V(D))$. Let $L$ be a 20-separated set of leaves in $T$ which does not contain $t$, with $|L|$ maximised. By the definition of $\lambda(T)$, $|L|\geq\l(T)-1\geq 10^4\l-1\geq 9|W_0|$. As $I(R\cup V(D))$ is $(d,m)$-extendable in~$G$, if $U\subset V(G)$ and $0<|U|\leq 2m$, then
\[
|N'(U)\setminus (R\cup W_0\cup W_1\cup\{v\})|\geq (d-1)|U|-1\geq (d-2)|U|.
\]
Therefore, by Definition~\ref{extendabledefn}, we have that $I(W_0\cup\{v\})$ is $(d-2,m)$-extendable in $G-(W_1\cup R)$. As $|W_0|\geq 9\l> n/\log^2 n$, $|N_T(L)|=|L|\geq 9|W_0|$, and 
\begin{equation}\label{deer}
|T|\leq |G|-|W_1|-|W_0|-|R|-10dm-\log n=|G-(W_1\cup R)|-|W_0|-10dm-\log n,
\end{equation}
we can apply Lemma~\ref{embedparentsfinal} to $G-(W_1\cup R)$. That is, we can find a copy of $T-L$, $S$ say, so that~$W_0$ is in the copy of $N_T(L)$,~$S$ is $(d-2,m)$-extendable in $G-(W_1\cup R)$ and $t$ is copied to $v$.

Using Corollary~\ref{mextend} repeatedly, extend the subgraph $S$ in $G-(W_1\cup R)$ by adding a leaf to every vertex in the copy of $N_T(L)$ which is not in the set $W_0$, to get a tree, $S'$ say, which is $(d-2,m)$-extendable in $G-(W_1\cup R)$. Note that, for each application of Corollary~\ref{mextend} here we use a subgraph of $S'$, which, using~\eqref{deer}, has size at most $|T|\leq |G|-|W_1\cup R|-10dm$. 

Let $H=G[V(S')\cup W_1\cup R]$. We claim that $R$ is made into a reservoir by $(H,v,T,t)$. Indeed, given any subset $U\subset R$ with $|U|=|R|/2=3\l$, find a matching in the $\l$-device between $W_0$ and $W_1\cup(R\setminus U)$. Use this matching to add the final leaves to~$S'$ to get a copy of $T$ with the vertex set $V(H)\setminus U$. This completes the proof of the lemma in the case that $\l> n/\log^2n$.

Suppose secondly, then, that $\l=n\log\log n/\log^2n$ and $(X,\{(x_i,y_i):i\in [9\l]\})$ is the $\l$-device for~$R$. Let $l=10^4\log^2 n/(\log\log n)^2$. The tree $T$ contains at least $9\l$ vertex disjoint bare paths, of length $l-1+2\log n$, which do not contain the vertex $t$. Call these paths $P_i$, $i\in [9\l]$, and remove them from the tree to get $T'$. Pick a vertex $v\in V(G)\setminus (X\cup R)$. As $I(X\cup R)$ is $(d,m)$-extendable in~$G$, we have, by the definition of extendability, that $I(X)$ is $(d,m)$-extendable in $G-R$. Therefore, by the definition of extendability, $I(\{v\}\cup X)$ is $(d-2,m)$-extendable in $G-R$. As $|X|+|T'|\leq |T|\leq |G|-|R|-10dm$, by Corollary~\ref{treebuild} we can extend $\{v\}$ to get a copy $S$ of $T'$ in $V(G)\setminus (X\cup R)$, with $t$ copied to $v$, so that $S\cup I(X)$ is $(d-3,m)$-extendable in $G-R$ (using dummy edges to replace the deleted paths, as we did in the proof of Lemma~\ref{almostspan}).

Let $(a_i,b_i)$, $i\in [9\l]$, be the pairs of vertices in $S$ which need to be connected by paths with length $l-1+2\log n$ to make $S$ into a copy of $T$. Using Corollary~\ref{trivial} repeatedly (once we have checked the details), we can find vertex disjoint $a_i,x_i$-paths and $b_i,y_i$-paths, $i\in [9\l]$, with length $\log n$ and interior vertices in $V(G)\setminus(V(S)\cup X\cup R)$ so that if $S'$ is the subgraph gained by adding the paths to $S$, then $S'\cup I(X)$ is still $(d-3,m)$-extendable in $G-R$. We will now check the details for the applications of Corollary~\ref{trivial}. As $|S'\cup I(X)|\leq |T|+|V(D)|\leq |G|-|R|-10dm-\log n$, the size condition for Corollary~\ref{trivial} holds for each application. As $\log n\geq 2\lceil \log (2m)/\log(d-4)\rceil+1$, the length condition for Corollary~\ref{trivial} holds for each application.

Letting $H=G[V(S')\cup X\cup R]$, we claim $R$ is made into a reservoir by $(H,v,T,t)$. Indeed, if $U\subset R$ and $|U|=|R|/2=3\l$, then, by the definition of a $\l$-device, there are vertex disjoint $x_i,y_i$-paths with length $l-1$ covering $X\cup (R\setminus U)$ exactly. Adding these paths to $S'$ completes the required copy of $T$ with the vertex set $V(S')\cup X\cup (R\setminus U)$.
\oof

Thus, we have reduced the task of making a set into a reservoir to the task of finding an appropriate $\l$-device. Such a $\lambda$-device can be found using the following lemma, which is proved for $\l=n\log\log n/\log^2n$ and $\l>n\log\log n/\log^2 n$ in Section~\ref{deviceD} and Section~\ref{deviceC} respectively.

\lem\label{finddevice}
Let $\l,n,n'\in\N$ satisfy $n\log\log n/\log^2 n\leq \l\leq n/100\log n$ and $n'\geq 3n/4$. Let $R\subset [n']$ satisfy $|R|=6\l$. Then, with probability $1-o(n^{-1})$, the random graph $G=\GG(n',10^{5}\log n /n)$ contains a $\l$-device for $R$.
\ma

\subsection{Building the $\l$-device when $\l=n\log\log n/\log^2 n$}\label{deviceD}
We will first construct \emph{absorbers} in Lemma~\ref{absorbone}, which are capable of absorbing single vertices into a path to give a path which contains that additional vertex, but which has the same endvertices. We will then use these absorbers to build up a global absorption property to create a $\lambda$-device with $\l=n\log\log n/\log^2 n$.
\ifdiss
\else
\fi

\de In a graph~$G$, we say $(A,x,y)$ is an \emph{absorber} for a vertex $v\in V(G)$ if $x\neq y$, $\{x,y\}\subset A\subset V(G)$, $v\notin A$, and there is both an $x,y$-path in~$G$ with the vertex set $A$ and an $x,y$-path in~$G$ with the vertex set $A\cup\{v\}$. We say that $|A|$ is the \emph{size} of the absorber, and call the vertices~$x$ and $y$ the \emph{ends} of the absorber.
\fn

\lem\label{absorbone} Let $n\in\N$ be sufficiently large, and take $d=\log n/\log\log n$ and $m=n/200d$.
Suppose~$G$ is an $m$-joined graph, with at least $n/2$ vertices, containing the set $W\subset V(G)$. Suppose that $I(W)$ is $(d,m)$-extendable in~$G$ and $|W|\leq n(\log\log n)^2/10^4\log^2 n$. Then there is a set of 100 absorbers in $V(G)\setminus W$ with size $16\log^2n/(\log\log n)^2+2$ for each vertex $v\in W$, so that all of these $100|W|$ absorbers are disjoint.
\ma
\ifdraft
\else
\pr Noting that $200|W|=o(|G|)$, and using Corollary~\ref{treebuild} repeatedly, attach $200$ leaves to each vertex $v\in V(I(W))$ to get the star $S_v$, so that the stars $S_v$ are disjoint and $H=\cup_{v\in W}S_v$ is $(d,m)$-extendable in~$G$. In each star $S_v$ divide the leaves into 100 pairs of vertices.

We wish to find $100$ absorbers for each vertex $v\in W$, each using a different pair of leaves in~$S_v$. We will do this by adding vertex disjoint paths with length $2k+1$ or $k-1$ several times between different vertex pairs in the subgraph $H$, where $k=4\log n/\log\log n\geq 2\lceil\log 2m/\log (d-1)\rceil+2$. We~will~add~$k^2$ vertices~for~each~vertex pair, so the subgraph will always have size at most $n/4\leq |G|-10dm-\log n$. The subgraph created will always have maximum degree at most 100. Thus, by Corollary~\ref{trivial}, we will be able to add the described paths so that the subgraph remains $(d,m)$-extendable. We will describe the construction of one absorber for $v\in W$ using one pair of leaves $x_0,y_1$ in $S_v$, but, using Corollary~\ref{trivial}, all the absorbers can be constructed simultaneously and disjointly.

Find an $x_0,y_1$-path $Q$ with length $2k+1$ and label it as $x_0x_1x_2\ldots x_{k}y_0y_{k}y_{k-1}\ldots y_2y_1$. Find $k$ vertex disjoint $x_i,y_i$-paths $P_i$, $i\in[k]$ with length $k-1$ (see Figures~\ref{absorberpictureQ} and~\ref{absorberpicture}, where the heavy lines are paths and the light lines are edges).

\input{absorberpictureevenQ}

Let $A=\{x_0,y_0\}\cup(\cup_{i\in[k]}V(P_i))$, so that $|A|=k^2+2$. When $k$ is even, the following two $x_0,y_0$-paths have vertex sets $A$ and $A\cup\{v\}$ respectively. 
\[
x_0x_1P_1y_1y_2P_2x_2x_3P_3y_3y_4\ldots y_{k}P_{k}x_{k}y_0
\]
\[
x_0vy_1P_1x_1x_2P_2y_2y_3P_3x_3\ldots x_{k}P_{k}y_{k}y_0
\]
When $k$ is odd the following two $x_0,y_0$-paths have vertex sets $A$ and $A\cup\{v\}$ respectively.
\[
x_0x_1P_1y_1y_2P_2x_2x_3P_3y_3y_4\ldots x_{k}P_{k}y_{k}y_0
\]
\[
x_0vy_1P_1x_1x_2P_2y_2y_3P_3x_3\ldots y_{k}P_{k}x_{k}y_0
\]
Thus, $(A,x_0,y_0)$ is an absorber for $v$, with $|A|=k^2+2$, as required.
\oof
\fi

We will now use our absorbers for single vertices to build up a system of paths with a global absorption property. Suppose we have two absorbers, $(A,x,y)$ and $(A',x',y')$, for the vertices $v$ and $v'$ respectively. Given a $y,x'$-path with interior vertices not in $A\cup A'\cup\{v,v'\}$, suppose we add its vertices to $A\cup A'$ to get the set $A''$. It can be easily seen that $(A'',x,y')$ is an absorber both for $v$ and for $v'$. Similarly, given several absorbers with paths linking their endpoints in some order, we can merge them into a single absorber.

We will take absorbers, produced by Lemma~\ref{absorbone}, and divide them into groups of up to~100, connecting them as described above. This creates absorbers capable of absorbing up to~100 different vertices. When we come to absorb vertices, each path will absorb exactly one vertex so that the length of the resulting path is controlled. By controlling precisely which absorbers we group together, we can build up a $\l$-device.
Before doing this, we first show the following lemma, which finds a graph which will guide us in the grouping of absorbers.

\lem \label{flexiblematching}
There is a constant $h_0\in\N$ such that, for every $h\geq h_0$ with $3|h$, there exists a bipartite graph $H$ with maximum degree at most 100 and vertex classes $X$ and $Y\cup Z$, with $|X|=h$, and $|Y|=|Z|=2h/3$, so that the following is true. If $Z'\subset Z$ and $|Z'|=h/3$, then there is a matching between $X$ and $Y\cup Z'$.
\ma
\ifdraft
\else
\pr
Take disjoint vertex sets $X$, $Y$, $Z$ with $|X|=h$ and $|Y|=|Z|=2h/3$. Furthermore, let $X_1,X_2\subset X$ be subsets with size $2h/3$ so that $X=X_1\cup X_2$.

Independently, place 25 random matchings between each of the pairs of sets $(X_1,Y)$, $(X_2,Y)$, $(X_1,Z)$, and $(X_2,Z)$, and let $H$ be the union of the resulting edges. As each vertex is in at most 100 matchings, the maximum degree of $H$ is at most 100.
By Lemma~\ref{lotsrandmatch}, with probability at least $1-o(h^{-2})$, for every set $Z'\subset Z$ with $|Z'|=h/3$, there is a matching between $X$ and $Y\cup Z'$. Therefore, for sufficiently large $h$, there must be some graph $H$ which satisfies the properties in the lemma.
\oof
\fi

\pr[Proof of Lemma~\ref{finddevice} when $\l=n\log\log n/\log^2 n$.] Let $d=\log n/\log\log n$ and $m=n/100d$, and take $l_0=16\log^2n/(\log\log n)^2+2$ and $l=10^4\log^2n/(\log\log n)^2$. Pick disjoint subsets $Y\subset [n']\setminus R$ and $R'$ with $|R'|=|R|=6\lambda$ and $|Y|=18\l$ and label the vertices in $Y$ so that $Y=\{x_i,y_i:i\in[9\l]\}$.

Reveal edges within the vertex set $[n']$ with probability $10^4\log n/n$ to get the graph~$G_1$. By Corollary~\ref{generalexpandcor}, with probability $1-o(n'^{-1})=1-o(n^{-1})$, the subgraph $I(R\cup R')$ is $(d,m)$-extendable in $G_1-Y$. By Proposition~\ref{generalprops}, with probability $1-o(n^{-1})$, the graph $G_1$ is $m$-joined. Therefore, by Lemma~\ref{absorbone}, we can find disjoint absorbers $(A_{v,j},a_{v,j},b_{v,j})$, $v\in R\cup R'$, $j\in[100]$, so that each absorber~$A_{v,j}$ has size $l_0$ and can absorb $v$. Let $Y'=\{a_{v,j},b_{v,j}:v\in R\cup R', j\in[100]\}$ and take $W=(\cup_{v,j}(A_{v,j}\setminus\{a_{v,j},b_{v,j}\}))\cup R'\cup R$, so that $|W|\leq 12\lambda(100l_0+1)=o(n)$.

Reveal more edges within the vertex set $[n']$ with probability $10^4\log n/ n$, and let the graph containing all the edges revealed be~$G$.  By Corollary~\ref{generalexpandcor}, with probability $1-o(n^{-1})$, the subgraph $I(Y'\cup Y)$ is $(d,m)$-extendable in $G-W$. Using Corollary~\ref{trivial}, we will create disjoint paths between vertices in $I(Y'\cup Y)$ while maintaining this extendability in $G-W$. As we create paths with length at least $\log n\geq 2\lceil\log(2m)/\log(d-1)\rceil+1$, the length condition will hold in each application of Corollary~\ref{trivial}. The subgraph will have always have maximum degree at most 2 and will contain in total at most $9\l l\leq n/100\leq |G-W|-10dm-l$ vertices, so the size condition will hold in each application of Corollary~\ref{trivial}. For each pair of vertices $(x_i,y_i)$, we will pick up to 100 absorbers and link them end to end with paths between $x_i$ and $y_i$ to create an absorber with size $l-1$ and ends $x_i$ and $y_i$, as depicted in Figure~\ref{spanpic8}. By allocating each pair $(x_i,y_i)$ a disjoint set of absorbers, and using Corollary~\ref{trivial} to create the paths, we can ensure the large absorbers created are disjoint.

To describe how to allocate absorbers to the vertex pairs we refer to an auxillary bipartite graph $H$ with the following properties, provided by Lemma~\ref{flexiblematching} with $h=9\l$, for sufficiently large~$n$, and hence~$h$. The bipartite graph $H$ has maximum degree $100$ and vertex classes $[9\l]$ and $R\cup R'$, such that, for any subset $U\subset R$ with $|U|=3\l$, there is a matching between $[9\l]$ and $R'\cup U$ in $H$. For each vertex $v\in R\cup R'$, let $c_v:N_H(v)\to [100]$ be an injective function numbering the neighbours of $v$ in $H$.

For each $i\in[9\l]$, carry out the following. Take the absorbers $A_{v,c_v(i)}$, $v\in N_H(i)$. Take the set of vertices $\{a_{v,c_v(i)},b_{v,c_v(i)}:v\in N_H(i)\}\cup\{x_i,y_i\}\subset Y'\cup Y$. Using Corollary~\ref{trivial}, join pairs from these vertices by paths in $G-W$ with length at least $\log n$ in such a way as to create an absorber $(B_i,x_i,y_i)$, with $|B_i|=l-1$, which is an absorber for each vertex $v\in N_H(i)$. One of these absorbers, and an example of how it absorbs a vertex, is depicted in Figures~\ref{spanpic8} and~\ref{spanpic82}.
Note that, because the functions~$c_v$ are injective, the absorbers $B_i$, $i\in [9\l]$, can be kept disjoint.

\begin{figure}[b!]

\vspace{0.25cm}

\centering
    \begin{subfigure}{0.95\textwidth}
        \centering
        \resizebox{\linewidth}{!}{
            \input{spanpic8A1}
        }
    \end{subfigure}
\label{spanpic8A1}
\caption{Joining three absorbers $A_{v_j,c_{v_j}(i)}$, $j\in [3]$, together using paths labelled as $P_j$, $j\in[4]$, to create an absorber capable of absorbing any of the vertices $v_1$, $v_2$ and $v_3$. In the proof of Lemma~\ref{finddevice} when $\l=n\log\log n/\log^2 n$, if $N_H(i)=\{v_1,v_2,v_3\}$, then the structure depicted, except for the vertices~$v_1$,~$v_2$ and~$v_3$, represents the absorber $(B_i,x_i,y_i)$. As an example, a path through this absorber and the vertex~$v_2$ is depicted in Figure~\ref{spanpic82}.} \label{spanpic8}
    \begin{subfigure}{0.95\textwidth}
        \centering
        \resizebox{\linewidth}{!}{
            \input{spanpic8A2}
        }
    \end{subfigure}\label{spanpic8A2}
\caption{An $x_i,y_i$-path through the absorber $(B_i,x_i,y_i)$ and the vertex~$v_2$ (cf.\  Figure~\ref{spanpic8}). Similarly, we can find $x_i,y_i$-paths through $(B_i,x_i,y_i)$ and either the vertex~$v_1$ or the vertex~$v_3$.} \label{spanpic82}
\end{figure}

Let $X=(\cup_{i\in[9\lambda]}B_i)\cup R'$. We claim that $(X,\{(x_i,y_i):i\in[9\l]\})$ is a $\l$-device for $R$. Indeed, suppose $U\subset R$ satisfies $|U|=3\l$. From the property of the graph $H$, we can find a matching $M$ in $H$ between $[9\l]$ and $R'\cup U$. For each $i\in [9\l]$, find the vertex $v\in R'\cup U$ matched to $i$ in $M$ and take an $x_i,y_i$-path with length $l-1$ through $B_i\cup\{v\}$. These paths disjointly cover $X\cup U$, and have length $l-1$, as required.
\oof

\subsection{Building the $\l$-device when $\l>n\log\log n/\log^2 n$}\label{deviceC}
When $\l>n\log\log n/\log^2 n$, the $\l$-device bears some similarity to the graphs that were shown to exist in Lemma~\ref{flexiblematching}. In fact, for sufficiently large $\l$, by Lemma~\ref{flexiblematching} there is a graph $H$ with vertex classes~$X$ and $Y\cup Z$ with $|X|=9\l$ and $|Y|=|Z|=6\l$ which has the matching property stated in that lemma. Then we can observe that $(X,Y)$ is a $\l$-device for $Z$ in $H$. When $\l=o(n/\log n)$, such a graph is too dense to appear as a subgraph of our random graph. Instead, as we did in Lemma~\ref{connectlowdense}, we find a subgraph of our random graph which resembles such a graph~$H$, but with each edge replaced by a path of length $3$, so that these short paths are internally vertex disjoint. Subdividing the edges in this manner decreases the density so that such a subgraph will typically appear in our random graph, and, as given in detail below, such a subdivided graph can also function as a $\l$-device. Furthermore, where $R$ is the set we wish to make into a reservoir, we will find such a subgraph with $Z=R$. 


\ifdraft
\else
\pr[Proof of Lemma~\ref{finddevice} when $\l>n/\log^2 n$.]
Take disjoint vertex sets $X, Y\subset [n']\setminus R$ satisfying $|X|=9\l$ and $|Y|=6\l$ . Pick subsets $X_1,X_2\subset X$ with $|X_1|=|X_2|=6\l$ and $X=X_1\cup X_2$.

We will find a bipartite auxillary graph $H$, with maximum degree 100, and vertex classes $X$ and $Y\cup R$, so that the following property holds.

\stepcounter{capitalcounter}
\begin{enumerate}[label =\bfseries \Alph{capitalcounter}\arabic*]
\item If $U\subset R$ and $|U|=|R|/2=3\l$, then there is a matching between $X$ and $Y\cup U$ in $H$. \label{ZZZZZ}
\end{enumerate}
We will also find an accompanying set of disjoint edges $F=\{e_f=v_fw_f:f\in E(H)\}$ in $G-(X\cup Y \cup R)$, so that, if $f=ab\in E(H)$, with $a\in X$ and $b\in Y\cup R$, then $av_fw_fb$ is a path of length 3 in~$G$. Such a graph, with its auxillary graph, is depicted in Figure~\ref{lamb1}.

If we can find such a graph $H$ and an edge set $F$, then the sets $W_0=X\cup(\cup_{f\in E(H)}w_f)$ and $W_1=Y\cup (\cup_{f\in E(H)}v_f)$ form a $\l$-device for $R$. Indeed, firstly, as $H$ has maximum degree~100, $|W_0|\leq |X|+100|X|\leq 10^3\l$. Secondly, given any set $U\subset R$ with $|U|=3\l$, by Property~\ref{ZZZZZ}, we can find a matching between $Y\cup U$ and $X$ in $H$ and call it $M$. Then 
\[
M'=\{e_f:f\notin M\}\cup\{av_f,w_fb:ab=f\in M,a\in X, b\in Y\cup U\}
\]
is a matching between $W_0$ and $W_1\cup U$ in~$G$, as depicted in Figure~\ref{lamb2}. Indeed, none of the edges in~$M'$ can share a vertex in $X$ or $Y\cup U$ because $M$ is a matching, none of the edges in $M'$ can share a vertex in $(W_0\cup W_1)\setminus (X\cup Y)$ from the choice of $F$, and each vertex in $W_0\cup W_1\cup U$ appears in some edge in~$M'$.

It is left then to find such a graph $H$ and an edge set $F$. Reveal edges within the vertex set~$[n']$ with probability $10^3\log n/ n$ to get the graph $G_1$. By Lemma~\ref{logmatching}, with probability $1-o(n'^{-1})=1-o(n^{-1})$, we can find disjoint sets~$A_x$, $x\in X\cup Y\cup R$, in $[n']\setminus (X\cup Y\cup R)$ so that $A_x\subset N(x)$ and $|A_x|=\log n$ for each vertex $x\in X\cup Y\cup R$.

Let $H$ be the bipartite graph with no edges, and the vertex classes $X$ and $Y\cup R$. Reveal more edges with probability $p=100\log n/ n$ within the vertex set $[n']$ to get the graph $G_2$. Consider the auxillary bipartite graph~$K$ with vertex classes $X$ and $Y\cup R$ and an edge $xy\in E(K)$ with $x\in X$ and $y\in Y\cup R$ exactly if there is some edge between $A_x$ and $A_y$ in $G_2$. The graph $K$ has edges present independently with probability
\[
1-(1-p)^{\log^2 n}\geq 50\log^3 n/ n> 300\log n/6\l.
\]
Therefore, by Lemma~\ref{matchingthres}, with probability $1-o(\l^{-2})=1-o(n^{-1})$, there is some matching in the graph $K$ between $X_1$ and $Y$. Pick such a matching uniformly at random from all such matchings, to get~$M_1$ say, and add it to $H$. As the sets $A_x$ and $A_y$ with $x\in X$ and $y\in Y\cup R$ all have the same size, each edge in $K$ is present with exactly the same probability. By symmetry then, in revealing edges and selecting the matching $M_1$ we must have been equally likely to end up with any of the possible matchings between $X$ and $Y\cup R$.

For each edge $f=xy\in M_1$, with $x\in X_1$ and $y\in Y$, pick vertices $v_f\in A_x$ and $w_f\in A_y$ so that $v_fw_f\in E(G_2)$. This is possible by the definition of the graph $K$. Remove the vertices $v_f$ and $w_f$ from~$A_x$ and~$A_y$ respectively, and remove a vertex arbitrarily from $A_z$ for each $z\in (X\setminus X_1)\cup R$.

In summary, by revealing edges in the graph with probability $100\log n/ n$, we have added a random matching $M_1$ to $H$ between $X_1$ and $Y$ while finding a set of disjoint edges $\{v_fw_f:f\in M_1\}$ so that if $f=xy$ then $xv_fw_fy$ is a path in $G_1\cup G_2$. We then removed one vertex from each set~$A_x$, $x\in X\cup Y\cup R$, so that the vertices $v_f$ and $w_f$, $f\in M_1$, were all removed. This ensures that all the sets $A_x$, $x\in X\cup Y\cup R$, still have the same size, so that by running a similar process again we can find another uniformly random matching. It also ensures that any further matchings found in this way will contain edges which are disjoint from the edges in $\{v_fw_f:f\in M_1\}$. Note that the size of the sets $A_x$, $x\in X\cup Y\cup R$, has decreased slightly, but there is enough room in the above calculations that we may repeat the process and find another matching.

Revealing 99 more sets of edges with probability $100\log n/ n$, with probability $1-o(n^{-1})$, we can add in this manner 24 more matchings in $H$ between $X_1$ and $Y$ in $H$, chosen uniformly at random from all possible such matchings, and 25 matchings chosen uniformly at random between each of the pairs of sets $(X_2,Y)$, $(X_1,R)$ and $(X_2,R)$. By Lemma~\ref{lotsrandmatch}, with probability $1-o(\l^{-2})=1-o(n^{-1})$, Property~\ref{ZZZZZ} holds. Therefore, as detailed above, the required sets~$W_0$ and~$W_1$ can be found. In total, we revealed edges with probability at most $10^5\log n/n$.
\oof
\fi

\input{lambdadevice}

\section{Cases C and D}\label{11caseCD}

The final tool that we need for the proof of Cases~C and~D of Theorem~\ref{unithres} is the notion of $(l,\gamma)$-connectors. 
We define these graphs and construct them in Section~\ref{connectors}, before proving Cases~C and D of Theorem~\ref{unithres} in Section~\ref{CDproof}.

\subsection{$(l,\gamma)$-connectors}\label{connectors}
For Cases~C and~D of Theorem~\ref{unithres}, we will use $(l,\gamma)$-connectors, defined as follows.
\de\label{connectordefn}
A graph $H$ is an \emph{$(l,\gamma)$-connector} if $H$ has $l$ vertices and there are two disjoint subsets $H^+,H^-\subset V(H)$, with $|H^+|=|H^-|= \lceil\gamma l\rceil$, so that, given any pair of vertices $x\in H^+$ and $y\in H^-$, there is an $x,y$-path in $H$ with the vertex set $V(H)$.
\fn

In practice, when we find an $(l,\gamma)$-connector $H$, we will implicitly fix sets $H^+$ and $H^-$ which demonstrate that it is such a connector. In our embedding for a tree $T$ in Case C or D, we will wish to find an edge from a vertex $v$ into the reservoir $R$, where $|R|=\Theta(n/k(T))$, so that we can then form a path in the reservoir to connect~$v$ into the embedding. However, if $k(T)\gg \log n$, then such an edge will typically not exist in our random graph $\GG(n,\Theta(\log n/n))$. Instead, we will find a $(\Theta(k(T)),\log\log n/16\log n)$-connector $H$, so that there is a neighbour, $v_1$ say, of $v$ in $H^+$. We then find an edge, $v_2w$ say, with $v_2\in H^-$ and $w\in R$. By the definition of a connector, there is a $v_1,v_2$-path with the vertex set $V(H)$, and, with the edges $xv_1$ and $v_2w$, this gives a path from~$v$ into the reservoir $R$, which we can then join into the embedding using vertices in $R$.


In fact, we will find a collection~$\HH$ of $\Theta(n/k(T))$ disjoint $(l,\gamma)$-connectors, where $l=\Theta(k(T))$ and $\gamma=\log\log n/16\log n$. We will wish, for most vertices $v$, to be able to find $\Theta(\log\log n)$ such connectors in~$\HH$ which can be connected to $v$ in this way. This will give enough flexibility to take any sufficiently small subset of these vertices and connect each of them to a different connector. Fortunately, when $H\in \HH$ and $v\in V(G)$, we expect to find
\[
\Theta(p|H^+||\HH|)= \Theta(\log n/n)\cdot\Theta(k(T)\log\log n/16\log n)\cdot\Theta(n/k(T))=\Theta(\log\log n)
\]
connectors which could be connected to $v$ in this manner. Of course, for some vertices we will not find enough such neighbours in the connectors, but we will manipulate our embedding to cover these vertices earlier with part of the tree.

To find the collection of connectors $\HH$, we will use the following lemma. The construction of $(l,\gamma)$-connectors uses P\'osa rotation (cf.\ Section~\ref{3hamcycles}). Essentially, to get an $(l,\gamma)$-connector, we construct a path with additional edges so that the path can be rotated many times around its endvertices.

\lem \label{connectorfind} Let $n,l\in \N$ satisfy $\log n\leq l\leq \log^2 n/\log\log n$, and let $\gamma=\log\log n/16\log n$. With probability $1-o(n^{-1})$, in the random graph $G=\GG(n,10^{10}\log n/n)$ there are $n/2l$ disjoint $(l,\gamma)$-connectors.
\ma
\pr Let $r=n/2l$, $k=4\log n/\log\log n$, and $s=\lfloor (l-k-1)/2k\rfloor\geq l/4k=\gamma l$. Note that, as $l\leq \log^2 n/\log\log n$, we have that $s\leq \log n$. Let $X\subset [n]$ be a set of $2r$ vertices, and reveal edges within the vertex set $[n]$ with probability $100\log n/n$ to get the graph $G_1$. By Lemma~\ref{logmatching}, there is an $s$-matching from $X$ into $V(G_1)\setminus X$. Let $Y$ be the set of vertices in the image of this matching, and note that~$|X\cup Y|= 2r(s+1)\leq n/k$.

Let $d=\log n/\log\log n$ and $m=n/100d$. Reveal more edges with probability $10^{9}\log n/n$ to get the graph $G_2$, and let $G=G_1\cup G_2$. By Corollary~\ref{generalexpandcor}, with probability $1-o(n^{-1})$, $I(X\cup Y)$ is $(d,m)$-extendable in $G$. By Proposition~\ref{generalprops}, with probability $1-o(n^{-1})$, the graph $G$ is $m$-joined.

Given two vertices $x_0,y_0\in X$, we will describe how to add paths with length at least $k$ between the vertices in the image of $\{x_0,y_0\}$ under the matching, using Corollary~\ref{trivial}, so that the vertices~$x_0$ and~$y_0$, their adjacent edges in the matching and these additional paths form an $(l,\gamma)$-connector. We will describe the construction of just one $(l,\gamma)$-connector for simplicity, but, using Corollary~\ref{trivial} (once the details are checked), it will be clear that we can divide the set $X$ into $r$ pairs of vertices and carry this out for each vertex pair to simultaneously create $r$ disjoint $(l,\gamma)$-connectors. In total, the subgraphs we create will contain $rl=n/2\leq n-10dm-l$ vertices, so the size condition for the applications of Corollary~\ref{trivial} will hold. We will create paths which have length at least $k\geq 2\lceil\log (2m)/\log(d-1)\rceil+1$ and interior vertices outside of the working subgraph, so the length condition for the applications of Corollary~\ref{trivial} will hold as well.

Take then the vertices $x_0,y_0\in X$, and label the vertices in their image under the matching as $x_1,\ldots,x_s$, and $y_1,\ldots, y_s$ respectively. For each $i\in[s]$, find an $x_{i-1},x_i$-path, with length $k$, and label it as $x_{i-1}P_iv_ix_i$, where $v_i$ is a vertex and $P_i$ is an $x_{i-1},v_i$-path with length $k-1$. For each $i\in[s]$, find a $y_{i-1},y_i$-path with length $k$, and label it as $y_{i-1}Q_iw_iy_i$, where $w_i$ is a vertex and $Q_i$ is a $y_{i-1},w_i$-path with length $k-1$. Find an $x_s,y_s$-path, $R$ say, with length $l-2ks-1\geq k$. Finally, let $H=G[V(R)\cup(\cup_{i\in[s]}(V(P_i)\cup V(Q_i)))]$ and note that $|H|=(l-2ks)+2ks=l$ (see Figure~\ref{connectorpic}).

\input{connectorpicture}

We claim that the sets $H^+=\{v_i:i\in[s]\}$ and $H^-=\{w_i:i \in[s]\}$ demonstrate that $H$ is an $(l,\gamma)$-connector. We have, firstly, that $|H^+|=|H^-|=s\geq \gamma l$. Take then any integers $i,j\in[s]$. The following path is a $v_i,w_j$-path that covers exactly the vertices in $H$ (see Figure~\ref{connectorpic}, where the heavy lines are paths and the light lines are edges).
\[
v_iP_ix_{i-1}v_{i-1}P_{i-1}x_{i-2}\ldots P_1x_0x_iP_{i+1}v_{i+1}\ldots x_sRy_sw_sQ_sy_{s-1}\ldots Q_{j+1}y_jy_0Q_1w_1y_1\ldots Q_jw_j
\]
Therefore, $H$ satisfies all the requirements to be an $(l,\gamma)$-connector.
\oof


When we use an $(l,\gamma)$-connector $H$ in our construction, we will often be interested in the edges that may exist between vertices elsewhere in the graph and the sets $H^-$ and $H^+$. To describe the properties of these edges we will use the following definitions.

\de
Given a collection $\mathcal{A}$ of $(l,\gamma)$-connectors, we let $\mathcal{A}^+=\{P^+:P\in \mathcal{A}\}$ and $\mathcal{A}^-=\{P^-:P\in \mathcal{A}\}$.
\fn
\de
Given a collection of subsets $\mathcal{A}$ in the graph $G$, the \emph{grouped graph on the set $\mathcal{A}$ with respect to $G$} is the graph $H$ with the vertex set $\mathcal{A}$ which contains an edge between $A$ and~$B$ exactly when there is an edge between $A$ and $B$ in $G$.
\fn
In fact, we will use bipartite grouped graphs with two vertex classes $\mathcal{A}$ and $\mathcal{B}$, where we only consider edges between the classes. Abusing our notation slightly, when $\mathcal{A}$ is a collection of subsets of $V(G)$ and $V\subset V(G)$, we say the \emph{bipartite grouped graph with vertex classes $\mathcal{A}$ and $V$} is such a graph on vertex classes $\mathcal{A}$ and $\{\{v\}:v\in V\}$.

\subsection{Cases C and D of Theorem~\ref{unithres}}\label{CDproof}
For the moment we will ignore the role of the $(l,\gamma)$-connectors and recap the proof outline for Cases~C and~D given in Section~\ref{1boutline}, while giving more details. Given a tree~$T$ in Case C or Case D, we split~$T$ into three subtrees~$T_1$, $T_2$ and $T_3$ with certain properties. In Stage 1 of the embedding, we start by using $T_1$ to build a reservoir out of a set~$R$ (using Section~\ref{9absorbC}), where we have selected~$R$ so that $|R|=\Theta(n/k(T))$ and~$R$ has good connection properties (using Section~\ref{6connectionlemmas}). The tree $T_2$ will contain many bare paths with length $k(T)$. In Stage 2, we remove $\Theta(n/k(T)\log n)$ such paths from $T_2$ to get $T_2'$ and extend the embedding to cover~$T_2'$. In Stage~3, we embed $T_3$ (using Section~\ref{4almostspanning}), and find a cycle supported by the remaining $\Theta(n/\log n)$ vertices outside of the reservoir,~$R$. Dividing this cycle into $\Theta(n/k(T)\log n)$ subpaths with length $\Theta(k(T))$, in Stage 4, we join these sections into the embedding of $T_2'$ by using shorter paths of length $\Theta(\log n)$, with interior vertices in the reservoir,~$R$. These new paths will allow us to embed the $\Theta(n/k(T)\log n)$ paths from $T_2$ that we removed. Crucially, the paths that we ask for in~$R$ will use in total $\Theta(n/k(T)\log n)\cdot\Theta(\log n)=\Theta(|R|)$ vertices from the reservoir. The set~$R$, which becomes the reservoir, is chosen carefully so that we may ask it to connect any collection of disjoint pairs of vertices with disjoint paths of length $6\log n$, so long as the pairs of vertices are chosen from a certain smaller subset of~$R$ (which has itself size $\Theta(|R|)$) and the paths found cover only a small linear proportion of the vertices in~$R$. Therefore, with suitably chosen constants, the embedding we describe above does not ask for too many paths to be formed using vertices in~$R$.

\smallskip\color{\newcolour}

\textbf{Connecting one endvertex into the reservoir.} The one problem remaining is that the endvertices of the paths from the divided cycle may not have any neighbours in~$R$.
We use $(l,\gamma)$-connectors, with $l=\Theta(k(T))$ and $\gamma=\Theta(\log\log n/\log n)$, to solve this problem. For illustration, let us sketch how we can connect one endvertex into the reservoir, before discussing how we do this for all the endvertices of the paths simultaneously. We find a set $\HH$ containing $\Theta(n/l)$ disjoint $(l,\gamma)$-connectors and hold them in reserve (discussed below). When we break the cycle into pieces, let $v$ be an endvertex of one of the paths. Our issue is that $v$ may not have a neighbour in $R$, as this set has size $\Theta(n/k(T))$, where $k(T)$ can be much larger than $\log n$. However, the sets $H^+$, across each $H\in \HH$, have a union with size at least $\Theta(n/l)\cdot \gamma l=\Theta(n\log\log n/\log n)$, so it is likely that $v$ has some neighbour in $\cup_{H\in \HH}H^+$ in our random graph $G$. In fact, most vertices in $G$ will have a neighbour in $\cup_{H\in \HH}H^+$, and we can make sure during our embedding that any vertex which does not is already covered by the embedding.

Now, an edge between $v$ and some $v'\in H^+$, and the properties of an $(l,\gamma)$-connector allow us to find a path with vertex set $\{v\}\cup V(H)$ from $v$ to any vertex in $H^-$. Our task now is to find an edge between $H^-$ and the reservoir, $R$. As $|H^-|\geq \gamma l=\Theta(\log\log n/k(T)\log n)$ and $|R|=\Theta(n/k(T))$, we expect $\Theta(\log\log n)$ edges between $H^-$ and $R$. Indeed, by only taking connectors in $\HH$ for which this is true, we can ensure there is some such edge, say between $w\in R$ and $w'\in H^-$. Then, taking a $v',w'$-path with vertex set $V(H)$ in $H$ (which exists as $H$ is an $(l,\gamma)$-connector, $v'\in H^+$ and $w'\in H^-$) and adding the edges $vv'$ and $w'w$ gives a path from $v$ into the reservoir, using exactly the vertices in one connector in $\HH$.

\smallskip

\textbf{Connecting all the endvertices into the reservoir.} To connect all the  $\Theta(n/k(T)\log n)$ endvertices when we break the cycle into pieces, we carry out the procedure for the vertex $v$ above, but for every endvertex simultaneously. Say the set of endvertices is $E$. We find a matching from $E$ into $\HH$
so that, if an endvertex~$v$ is matched to $H\in \HH$, then there is an edge in~$G$ between~$v$ and $H^+$. Taking the matched $(l,\gamma)$-connectors, we, similarly, match them into the reservoir using the sets $H^-$, for each $H\in \HH$. Using these matchings and the properties of the connectors, we can then find, disjointly for each $v$, a path from $v$ into the reservoir using exactly the vertices in one connector in $\HH$.

\emph{Why} such matchings should exist, comes, essentially, from comparison to the following property which is likely to hold for a set $W$ of $\Theta(n\log\log n/\log n)$ vertices in $G'=G(n,\log n/n)$. Given any set $U$ of vertices outside of $W$, each with $\Theta(\log \log n)$ neighbours in $G'$ in $W$, then, if $|U|\leq |W|/2$, we can match $U$ into $W$.
Here, $\Theta(\log\log n)$ possible matches for each vertex in $U$ are needed because we are allowed to choose the vertices in $U$ from a set of vertices which is almost $\log n$ times as large as $W$. This property follows from Lemma~\ref{mindegexp3}. Essentially, the matchings described above exist because the relevant grouped graph we find them in behaves like a subgraph of the random graph $G'$.

\smallskip

\textbf{Connecting in the paths from the cycle.} Using connectors, we have now extended each path from the cycle using paths leading into the reservoir. We will have manipulated the embedding of~$T_2'$ so that the paths we removed from $T_2$ have endvertices embedded into the reservoir. Thus we can use the connection properties of the reservoir to connect each extended path from the cycle into the embedding of $T_2'$ in place of a missing path.



\smallskip

\textbf{Holding connectors in reserve and replacing used connectors.} We will hold the connectors in $\HH$ in reserve using the tree $T_2$. In Stage 2, when we embed $T_2$ with some paths missing, we will in fact embed $T_2$ with a further $|\HH|=\Theta(n/l)$ missing paths, each of length $l+1$. We will do this so that each missing path could be replaced by finding a path through a corresponding connector in $\HH$. At the end of the embedding, for each connector in $\HH$ that we do not use to connect vertices to the reservoir, we find a path through that connector to embed the corresponding path from~$T_2$, and call the connector \emph{unused}. We will then need to embed the deleted paths corresponding to the connectors that \emph{are} used. In our embedding we make sure that embedding each such deleted path will require two vertices from the reservoir to be connected by a path of the appropriate length. However, we cannot just use vertices from the reservoir to do this. Indeed, we will use $\Theta(n/k(T)\log n)$ of the connectors, so this would use $\Theta(n/\log n)$ vertices from the reservoir, which is far too many when $k(T)=\omega(\log n)$.

Therefore, when gathering the properties we need for Stage 4, we set aside paths of length $\Theta(l)$ to form the majority of the paths replacing the used connectors. We note that we do not know in advance which connectors we will use, but we \emph{do} know how many we will use. Thus, we can set aside the correct number of paths -- these paths are denoted~$P_{i,j}$, for varying $i$ and $j$, and we ensure their endvertices lie in the reservoir. When we find paths through the reservoir, we will by then have identified which connectors we are going to use, and we can also find paths through the reservoir to connect the paths $P_{i,j}$ in place of the used connectors. By choosing the paths $P_{i,j}$ to have length close to the size of the connectors, connecting in each path $P_{i,j}$ will use only $\Theta(\log n)$ vertices. As we use $\Theta(n/k(T)\log n)$ connectors, connecting in the paths $P_{i,j}$ requires only $\Theta(n/k(T))=\Theta(|R|)$ vertices from the reservoir.\color{black}



\pr[Proof of Theorem~\ref{unithres} in Cases C and D] For each $k$ with $10^2\log n\leq k\leq 10^{-6}\log^2 n/\log\log n$, we will show that, with probability $1-o(n^{-1})$, we can embed all the trees $T\in\TT(n,\Delta)$ with $k(T)=k$ into the random graph $\GG(n,10^{52}\log n/n)$. Therefore, almost surely, $\GG(n,10^{52}\log n/n)$ will contain a copy of every tree $T\in \TT(n,\Delta)$ with $k(T)$ in this range.
Let
\begin{equation}\label{halfeleven}
\l=n/10^{15}k,\;\;\;\; l=k/4\geq 24\log n,\;\;\;\;r=\lambda/12\log n\;\;\;\;\text{and}\;\;\;\;\gamma=\log\log n/16\log n.
\end{equation}
Let
\begin{equation}\label{dmdefn}
d=2\log n/\log\log n \;\text{ and }\; m=n/10^8d=5n\log\log n/10^9\log n.
\end{equation}
Recall that
\[
10^2\log n\leq k\leq 10^{-6}\log^2 n/\log\log n.
\]

Note that, in Case D, $10^9\l=n\log\log n/\log^2 n$.

\medskip

\textbf{Revealing the random graph.}
We will reveal edges within the vertex set $[n]$ in rounds with total edge probability at most $10^{52}\log n/n$. The final structure that we find with probability $1-o(n^{-1})$ is depicted in Figure~\ref{spanpic7A}.

\textbf{Stage 4 properties.}
Reveal edges with probability $100\log n/ n$ within the vertex set $[n]$ to get the graph $G_1$. By Lemma~\ref{connectlowdense}, there are subsets $R_0\subset R\subset [n]$ with $|R|=10^9\l$ and $|R_0|=\l$ so that the following property holds.
\stepcounter{capitalcounter}
\begin{enumerate}[label =\bfseries \Alph{capitalcounter}\arabic*]
\item Given any $6r$ disjoint pairs of vertices $(a_i,b_i)$, $i\in[6r]$, in $R_0$ there are disjoint $a_i,b_i$-paths in $G_1[R]$, with length $6\log n$ and internal vertices in $R\setminus R_0$.\label{R1}
\end{enumerate}

Reveal edges with probability $10^{20}\log n/ n$ to get the graph $G_2$. Recall the definition of $d$ and~$m$ from~(\ref{dmdefn}).
By Proposition~\ref{generalprops}, with probability $1-o(n^{-1})$, the graph $G_2$ is~$m$-joined. By Corollary~\ref{generalexpandcor}, with probability $1-o(n^{-1})$, $I(R)$ is $(d,m)$-extendable in~$G_2$.

Divide $R_0$ into two sets, $R_1$ and~$R_2$, with $|R_1|=6r$ and $|R_2|=|R_0|-6r$, and label the vertices of~$R_1$ so that $R_1=\{x_i,y_i,u_{i,1},v_{i,1},u_{i,2},v_{i,2}:i\in[r]\}$.
Note that $l-12\log n+1\geq 12\log n\geq 2\lceil \log(2m)/\log (d-1)\rceil+1$, and that $2rl=o(n)$. Start with the subgraph $I(R)$ and, by repeatedly applying Corollary~\ref{trivial}, for each $i\in[r]$ and $j\in[2]$, add a disjoint $u_{i,j},v_{i,j}$-path, $P_{i,j}$ say, with length $l-12\log n+1$ and internal vertices in $[n]\setminus R$, so that all the created paths are disjoint and the resulting subgraph is still $(d,m)$-extendable.

Reveal edges with probability $10^{12}\log n/ n$ among the vertex set $[n]$ to get the graph $G_3$. By Lemma~\ref{connectorfind}, there are, almost surely, $\l/6$ disjoint $(l,\gamma)$-connectors in $G_3-R-\cup_{i,j}V(P_{i,j})$, where~$l$ and~$\gamma$ are defined in (\ref{halfeleven}). Let $\HH$ be a set of such connectors.


Reveal more edges with probability $p=10^{24}\log n/ n$ to get the graph $G_4$. Consider the bipartite grouped graph $K$ with vertex classes $\HH^{\pm}=\mathcal{H}^+\cup \mathcal{H}^-$ and~$R_2$, with respect to $G_4$. Each set in~$\HH^{\pm}$ has size at least $l\gamma$ by the definition of an $(l,\gamma)$-connector. Therefore, the probability an edge is present between $U\in \HH^{\pm}$ and $v\in R_2$ in $K$ is at least
\[
1-(1-p)^{l\gamma}\geq pl\gamma/2\geq 10^6\log\log n/\l.
\]
As all the sets in $\HH^{\pm}$, and the set~$R_2$, are disjoint, each potential edge is present, or not, in $K$ independently.
Therefore, as $|\HH^{\pm}|+|R_2|=4\l/3-6r\leq n$, the graph $K$ can be viewed as a subgraph of some random graph $\GG(n, 10^6\log\log n/\l)$. If $q=10^6\log\log n/\l$, then
\[
\frac{10\log (nq)}{q}=\frac{\l\log(10^6n\log\log n/\l)}{10^5\log\log n}= \frac{\l\log(10^{21}k\log\log n)}{10^5\log\log n}\leq \frac{\l\log(10^{21}\log^3n)}{10^5\log\log n}\leq \frac{\l}{10^3},
\]
for large $n$.
By Proposition~\ref{generalprops}, with probability $1-o(n^{-1})$, any two subsets from $\HH^\pm$ and~$R_2$ respectively with size $m_\l:=\l/10^3$ have an edge between them. Therefore, as $|R_2|=\lambda-6r \geq 100m_\l$, using Proposition~\ref{neat3}, we can take a subset $\BB\subset\HH^{\pm}$, with $|\BB|\leq m_\l$, so that, if $\UU\subset \HH^{\pm}\setminus \BB$ and $|\UU|\leq m_\l$, then $|N_K(\UU)|\geq 2|\UU|$. If $\UU\subset  \HH^{\pm}\setminus \BB$ and $|\UU|\geq m_\l$, then, as there are no edges between~$\UU$ and $R_2\setminus N_K(\UU)$ in $K$, we must have that $|R_2\setminus N_K(\UU)|\leq m_\l$, and therefore, as $|\HH|=\l/6$, we have that
\[
|N_K(\UU)|\geq |R_2|-m_\l\geq \l-6r-m_\l\geq 4|\HH|\geq 2|\UU|.
\]
Therefore, by Theorem~\ref{matchingtheorem} there is a 2-matching from $\HH^{\pm}\setminus \BB$ into~$R_2$ in the graph $K$. Pick such a 2-matching.

Let $\HH_0=\{H\in \HH:H^+,H^-\in \HH^{\pm}\setminus\BB\}$, noting that $|\HH_0|\geq|\HH|-2m_\l\geq \l/8$. For each $H\in \HH_0$, let $a_H$ and $b_H$ be the vertices in~$R_2$ matched to $H^+$ in the 2-matching, and let $c_H$ be one of the vertices in~$R_2$ matched to $H^-$ in the 2-matching (we do not use the second vertex). The vertices $a_H$ and $c_H$ will be used to hold the connector $H\in \HH$ in reserve, while the vertex $b_H$ is used if $H$ is employed to connect a vertex through to the reservoir.

\textbf{Stage 3 properties.}
Let $W_0$ be the set of vertices not in any of the paths $P_{i,j}$, or in any of the connectors $H\in\HH_0$, or in the set~$R$, as depicted in Figure~\ref{spanpic7A0}. Note that
\begin{equation}\label{finallyy}
|W_0|\geq n-2rl-|\mathcal{H}_0|l-10^9\l\geq n-2n/10^{16}-o(n)\geq(1-10^{-15})n.
\end{equation}
Reveal more edges with probability $10^{24}\log n/ n$ to get the graph $G_5$. Consider the auxillary bipartite grouped graph $L$ on the vertex classes $\mathcal{H}_0^-$ and $W_0$, with respect to $G_5$. As in the graph~$K$, the edges of $L$ are present independently at random with probability at least $10^6\log\log n/\l$. Therefore, by considering $L$ as a subgraph of the random graph $\GG(n,10^6\log\log n/\l)$, and applying Proposition~\ref{generalprops}, with probability $1-o(n^{-1})$, any two subsets of $\mathcal{H}_0^-$ and $W_0$ respectively, with size $m_\l=\l/10^3$, have an edge between them.

Note that $|\mathcal{H}_0^-|\geq \lambda/8\geq 100m_\l$. Using Proposition~\ref{neat3}, take $B\subset W_0$ to be a subset satisfying $|B|\leq m_\l$ and the following property, where we set $W_1=W_0\setminus B$.
\begin{enumerate}[label=\bfseries \Alph{capitalcounter}\arabic*]\addtocounter{enumi}{1}
\item If $U\subset W_1$ and $|U|\leq m_\l$, then $|N_L(U)|\geq |U|$.\label{R3}
\end{enumerate}

Note that $rl= n/48(10^{15}\log n)\geq n\exp(4)/10^{20}\log n$ and $|W_1|\geq |W_0|-m_\l\geq n/2\geq 100rl$. Reveal more edges among the vertex set $[n]$ with probability $10^{50}\log n/ n$ to get the graph $G_6$. By~\eqref{finallyy}, and Lemma~\ref{almostspan} applied to the graph $G_6$ and the sets $W=W_1$ and $B=V(G)\setminus W_0$, with path length $5\log n$ and $w=rl$, with probability $1-o(n^{-1})$ we can find a subset $W_2\subset W_0$, with $|W_2|\leq 5n/10^4$, so that the following holds.
\begin{enumerate}[label=\bfseries \Alph{capitalcounter}\arabic*]\addtocounter{enumi}{2}
\item Suppose $S$ is a tree, with $\Delta(S)\leq \Delta$, which contains at least $n/10^5\log n$ disjoint bare paths with length $50\log n+1$, and a vertex $t\in V(S)$. Suppose $V\subset [n]$ and $v\in V\setminus W_2$, with $|V|=|S|+rl$ and $W_2\subset V$. Then, there is a copy $S'$ of~$S$ in $G_6[V]$ so that $V\setminus V(S')\subset W_1$,~$t$ is copied to~$v$, and $G_6[V\setminus V(S')]$ is Hamiltonian. \label{R4}
\end{enumerate}

\textbf{Stage 1 and 2 properties.}
Noting that $|W_0\setminus W_2|\geq n-n/10^{15}-5n/10^4$, divide $W_0\setminus W_2$ into the sets $Z_1$, $Z_2$, $Z_3$, $Z_4$ and~$R'$ so that $|Z_1|=|Z_2|=|Z_3|=n/10^5$ and $|R'|=5(10^9\l)=o(n)$. Reveal more edges among the vertex set $[n]$ with probability $10^{20}\log n/ n$ to get the graph $G_7$. By Lemma~\ref{finddevice}, we can find a $(10^9\l)$-device~$D$ for $R\cup R'$ in $G_7[R\cup R'\cup Z_4]$ (see Figure~\ref{spanpic7A}). By Lemma~\ref{generalexpand}, with probability $1-o(n^{-1})$, for every set $A\subset [n]$ with $|A|\leq 2m$ and each $i\in[3]$ we have $|N_{G_7}(A,Z_i)|\geq d|A|$, so we have, by Proposition~\ref{uttriv}, the following property.
\begin{enumerate}[label=\bfseries \Alph{capitalcounter}\arabic*]\addtocounter{enumi}{3}
\item If $i\in[3]$, and the subgraph $S$ and the set $U\subset [n]$ satisfy $S\subset G_7-Z_i$, $V(S)\cup Z_i\subset U$ and $\Delta(S)\leq d$, then $S$ is $(d,m)$-extendable in $G_7[U]$.\label{R5}
\end{enumerate}

\color{\newcolour}
\textbf{Recapping the relevant properties.} We have now all the properties we require to embed trees in Cases C and D. Letting $G$ be the graph containing all the edges revealed, we will now recap all the properties we will use for our embedding.
As depicted in Figure~\ref{spanpic7A}, we have found the following in $G$.
\begin{itemize}
\item Sets $R_1,R_2\subset R_0\subset R$ and $R'$ with $|R_1|=6r$, $|R_2|=\lambda-6r$, $|R_0|=\lambda$,  $|R|=10^9\lambda$ and $|R'|=5(10^9\lambda)$.
\item A collection of vertex-disjoint paths $\mathcal{P}=\{P_{i,j}:i\in [r],j\in [2]\}$, each with length $l-12\log n+1$, such that, for each $i\in [r]$ and $j\in [2]$, the endvertices $u_{i,j}$ and $v_{i,j}$ of $P_{i,j}$ are in $R_1$.
\item A collection of vertex-disjoint $(l,\gamma)$-connectors $\mathcal{H}_0$, where each $H\in \mathcal{H}_0$ contains sets $H^+$ and $H^-$ as in Definition~\ref{connectordefn}, and $\lambda/8\leq |\mathcal{H}_0|\leq \lambda/6$.
\item Distinct vertices $a_H,b_H,c_H$, over all $H\in \mathcal{H}_0$, in $R_2$, where there are edges from $H^+$ to both $a_H$ and $b_H$, and from $H^-$ to $c_H$,  in $G$.

\item Vertex sets $Z_1$, $Z_2$ and $Z_3$, each with size $n/10^5$, and vertex set $W_2$ with size at most $5n/10^4$.
\item A vertex set $Z_4$ containing a $\lambda$-device $D$.
\end{itemize}
The sets $R$, $R'$, $Z_1$, $Z_2$, $Z_3$, $Z_4$, $W_2$, $V(H)$, with $H\in \mathcal{H}_0$, and $V(P_{i,j})\setminus \{u_{i,j},v_{i,j}\}$, with $i\in [r]$ and $j\in [2]$, are pairwise disjoint and form a partition of $V(G)$.

The properties \ref{R1}--\ref{R5} are all unaffected by the addition of edges, and thus hold for $G$. For these properties, we need to remember two further elements, a set $W_1\supset W_2$ and an auxillary graph $L$, the latter of which we can easily see can be taken to be the grouped graph on the vertex classes $\mathcal{H}_0^-$ and $W_1$ with respect to $G$. The purpose of $W_1$ is that we can cover all the vertices outside of $W_1$ using \ref{R4}, and any set remaining in the uncovered vertices can expand in $L$ by \ref{R3}. We can now embed the trees in Cases C and D.
\color{black}

\medskip

\textbf{Embedding the tree: split the tree.} Let $T\in\TT(n,\Delta)$ satisfy $k(T)=k$. Let $k'=k'(T)$, so that $T$ contains at least $n/90k'$ vertex disjoint bare paths with length $k'$.
Using Corollary~\ref{dividepath}, divide~$T$ into two trees,~$T_1$ and $T'$, intersecting on a single vertex~$t_1$, so that both $T_1$ and $T'$ contain at least $(n/270k')-1$ disjoint bare paths with length~$k'$. Say, without loss of generality, that $|T_1|\geq n/2$. Using Corollary~\ref{dividepath}, divide $T'$ into two trees,~$T_2$ and~$T_3$, intersecting on a single vertex $t_2$, so that both~$T_2$ and $T_3$ contain at least $n/900k'$ disjoint bare paths with length~$k'$. Suppose, without loss of generality, that $t_1\in V(T_2)$. Note that, due to the paths they contain, $|T_2|,|T_3|\geq n/10^3+1$, and hence $|T_1|\leq n-2n/10^3$. The division of $T$ into subtrees is depicted in Figure~\ref{spanpic7B}. The following embedding is depicted in Figures~\ref{spanpic7D1} to~\ref{spanpic7F}.

\textbf{Stage 1.} Recall the choice of~$R$ from the start of the development of the Stage 4 properties, and the choice of $R'$, $Z_3$ and $Z_4$ from the development of the Stage 1 and 2 properties. By Property~\ref{R5}, $I(V(D)\cup R\cup R')$ is $(d,m)$-extendable in $G[Z_3\cup Z_4\cup R\cup R']$. Note that, as $D$ is a $(10^9\l)$-device, $|V(D)|=o(n)$ (see~\eqref{VDsize}). As $|R'|=5(10^9\l)$, we have that
\begin{equation*}\label{Z3Z4}
|Z_3\cup Z_4|=|W_0|-|W_2|-|Z_1|-|Z_2|-|R'|\geq n-n/10^3\geq |T_1|+|V(D)|+10dm+\log n.
\end{equation*}
In Case D, where $10^9\l=n\log\log n/\log^2 n$, $T_1$ contains at least $n/300k'$ vertex disjoint bare paths of length~$k'$. As $k'\geq k$, there are at least $k'/3k$ disjoint bare paths with length $k$ in each disjoint bare path with length $k'$. Thus, $T_1$ contains at least $n/900k$ vertex disjoint bare paths with length~$k$. Discarding some of these paths and shortening the rest, we can find at least $9(10^9\l)+1$ vertex disjoint bare paths which each have length $10^4\log^2 n/(\log\log n)^2+2\log n$. Thus, in Case~D, by Lemma~\ref{reservoirfromdevice} we can find a subset $Y_1\subset Z_3\cup Z_4$  with $|Y_1|=|T_1|-3(10^9\lambda)$, and a vertex $v_1\in Y_1$, so that $R\cup R'$ is made into a reservoir by $(G[Y_1\cup R\cup R'],v_1,T_1,t_1)$ (see Figure~\ref{spanpic7D1}).

In Case C, as $T_1$ does not contain $n/90(k+1)+1$ vertex disjoint bare paths with length $k+1$ (by the definition of $k(T)$), and $|T_1|\geq n/2$, by Lemma~\ref{farleaves}, $T_1$ must either contain at least $|T_1|/100\geq n/200$ leaves, or contain a 20-separated set of at least $n/80(k+1)$ leaves. For sufficiently large~$n$, in the first case, as $T_1$ has maximum degree $\Delta$, it must contain a 20-separated set of at least $n/400\Delta^{20}\geq n/80(k+1)$ leaves. Therefore, in both cases, $\lambda(T_1)\geq n/80(k+1)\geq n/10^2k=10^4(10^9\lambda)$. Therefore, by Lemma~\ref{reservoirfromdevice}, we can find a subset $Y_1\subset Z_3\cup Z_4$ with $|Y_1|=|T_1|-3(10^9\lambda)$, and a vertex $v_1\in Y_1$, so that $R\cup R'$ is made into a reservoir by $(G[Y_1\cup R\cup R'],v_1,T_1,t_1)$ (see Figure~\ref{spanpic7D1}).

In both cases, pick a subset $R''\subset R'$ with
\begin{equation}\label{eqrefnow}
|R''|=2(10^9\l)+6r(6\log n+1)+2(|\HH_0|-2r)\leq 2(10^9\l)+4\l+\l\leq 5(10^9\l)=|R'|.
\end{equation}
Recall that $|R|=10^9\lambda$ and, in both cases, $|Y_1|=|T_1|-3(10^9\lambda)$. Together with~\eqref{eqrefnow}, this implies that
\begin{equation}\label{Y1RR}
|Y_1|+|R|+|R''|=|T_1|+36r\log n+2|\HH_0|+2r.
\end{equation}
Note that, if $U\subset R$ and $|U|=36r\log n+2|\HH_0|+2r$, then $|U\cup (R'\setminus R'')|=3(10^9\l)$, and hence
\begin{equation*}
|(R\setminus U)\cup R''|=|R'\cup R|-|U\cup (R'\setminus R'')|=3(10^9\l).
\end{equation*}
Therefore, as $R\cup R'$ is made into a reservoir by $(G[Y_1\cup R\cup R'],v_1,T_1,t_1)$, we have the following property.
\begin{enumerate}[label=\bfseries \Alph{capitalcounter}\arabic*]\addtocounter{enumi}{4}
\item If $U\subset R$ with $|U|=36r\log n+2|\HH_0|+2r$, then there is a copy of $T_1$ with the vertex set $Y_1 \cup (R\setminus U)\cup R''$ in which $t_1$ is copied to $v_1$.\label{R6}
\end{enumerate}

\textbf{Stage 2.} Let $V_0=((Z_3\cup Z_4)\setminus Y_1)\cup (R'\setminus R'')$. Recall that $T_2$ has at least $n/900k'$  vertex disjoint bare paths of length $k'\geq k$, each of which contains at least $k'/3k$ disjoint bare paths with length $k$. Thus, $T_2$ contains at least $n/2700k\geq 2\l+2\geq |\HH_0|+r+2$ vertex disjoint bare paths of length $k$. Remove $|\HH_0|$ vertex disjoint bare paths of length $l+2\log n+1\leq k$, and~$r$ vertex disjoint bare paths of length $3l+14\log n+1\leq k$, from $T_2$ to get the forest $T_2'$, choosing the paths so that $t_1,t_2\in V(T_2')$. The subgraph $I(\{v_1\})$ is $(d,m)$-extendable in $G[V_0\cup Z_2\cup\{v_1\}]$ by Property~\ref{R5}. As $|Z_2|=n/10^5$, note that
\begin{align}
|V_0\cup Z_2\cup\{v_1\}|&=|Z_3\cup Z_4|-|Y_1|+|Z_2|+|R'\setminus R''|+1
=|W_0|-|Y_1|-|Z_1|-|R''|-|W_2|+1 \nonumber
\\
&\geq n-|Y_1|-5n/10^4-n/10^4
\geq n-|T_1|-|T_3|+n/10^4\geq |T_2|+10dm.\label{quartmid}
\end{align}
Therefore, we may, by Corollary~\ref{treebuild}, find a copy, $S_2'$ say, of $T_2'$ in $G[V_0\cup Z_2\cup\{v_1\}]$, in which $t_1$ is copied to $v_1$ (using dummy edges to replace the removed paths, as we did in the proof of Lemma~\ref{almostspan}), as shown in Figure~\ref{spanpic7D2}.

Let $V_1=(V_0\cup Z_2\cup\{v_1\})\setminus V(S_2')=(V_0\cup Z_2)\setminus V(S_2')$. Label some of the vertices in $V(S_2')$ appropriately as $a_H'$, $c_H'$, $H\in \HH'$, and $x'_i$, $y'_i$, $i\in[r]$, so that to make $S_2'$ into an embedding of $T_2$ we need to find vertex disjoint $a_H',c_H'$-paths, $H\in \HH_0$, of length $l+2\log n+1$, and vertex disjoint $x'_i,y'_i$-paths, $i\in[r]$, of length $3l+14\log n+1$. Let~$X$ be the set of vertices in the following vertex pairs, taken over all $i\in[r]$ and $H\in \HH_0$.
\begin{equation}\label{threeoclock}
 (a'_H,a_H),\;\; (c_H,c'_H),\;\; (x'_i,x_i),\;\; (y'_i,y_i)
\end{equation}
Note that $|X|=4|\HH_0|+4r\leq \l$.
The subgraph $I(X)$ is $(d,m)$-extendable in $G[X\cup V_1\cup Z_1]$ by Property~\ref{R5}. Using (\ref{quartmid}), note that
\[
|X\cup V_1\cup Z_1|\geq |V_0\cup Z_2\cup\{v_1\}|-|S_2'|+|Z_1|\geq |T_2|+10dm-|S_2'|+n/10^5\geq n/10^5+10dm.
\]
Note that $\log n\geq 2\lceil\log (2m)/\log (d-1)\rceil+1$, and $|X|(1+\log n)\leq 2\l\log n\leq n/10^{10}$. Therefore, by repeated use of Corollary~\ref{trivial}, we can connect the pairs of vertices from $X$ in (\ref{threeoclock}) by vertex disjoint paths with length $\log n$ in $G[X\cup V_1\cup Z_1]$. Use these paths to extend $S_2'$ into the subgraph~$S_2''$, so that to complete the embedding of $T_2$, we need to find vertex disjoint $a_H,c_H$-paths, $H\in\HH_0$, with length $l+1$, and vertex disjoint $x_i,y_i$-paths, $i\in [r]$, with length $3l+12\log n+1$. Note that the connector $H$ lies poised to create a path with length $l+1$ between $a_H$ and $c_H$ (see Figure~\ref{spanpic7D3}). Indeed, as $a_HH^+$ and $c_HH^-$ are edges in $K$, there are vertices $a_H''\in H^+$ and $c_H''\in H^-$ so that $a_Ha_H''$ and $c_Hc_H''$ are edges of~$G$. By the definition of an $(l,\gamma)$-connector, we can find a path with the vertex set $V(H)$ which connects~$a_H''$ to~$c_H''$. With the edges $a_Ha_H''$ and $c_Hc_H''$, this would give an $a_H,c_H$-path with length $l+1$, as required. This is how we will find the required $a_H,c_H$-path when the connector $H\in \HH_0$ is not used elsewhere in the embedding.

\textbf{Stage 3.} Let $V_2=(V_1\cup Z_1)\setminus V(S_2'')$. The set $V_2\cup W_2$ contains all the vertices not in $V(S_2'')$, $Y_1$, $R$, or $R''$, and not in the paths $P_{i,j}$, $i\in[r]$, $j\in[2]$, or in the connectors in $\HH_0$. By considering the paths that need to be added to make $S_2''$ a copy of $T_2$, we have that $|S_2''|=|T_2|-r(3l+12 \log n)-l|\mathcal{H}_0|$. Note that $Y_1\cap V(S_2'')=\{v_1\}$, so that $|Y_1\cap V(S_2'')|=1$.
Note that $R\cap V(S_2'')=\{a_H,b_H:H\in\HH_0\}\cup\{x_i,y_i:i\in[r]\}$, so that $|R\cap V(S_2'')|=2|\HH_0|+2r$.
 The paths~$P_{i,j}$, $i\in[r]$, $j\in[2]$, contain in total $2r(l-12\log n+2)$ vertices, though $4r$ of these vertices, the endvertices, $u_{i,j}$, $v_{i,j}$, $i\in[r]$, $j\in[2]$, are also in $R$. Combining all of this, and using (\ref{Y1RR}), we have
\begin{align*}
|V_2\cup W_2|&=n-|S_2''|-|Y_1|-|R|-|R''|-2r(l-12\log n)-l|\HH_0|+2|\HH_0|+2r+1\\
&=n-|S_2''|-|T_1|-2rl-12r\log n-l|\HH_0|+1\\
&=n-|T_1|-|T_2|+1+rl=|T_3|-1+rl.
\end{align*}
Let $v_2$ be the copy of $t_2$ in $S_2''$. Then, $|V_2\cup W_2\cup\{v_2\}|=|T_3|+rl$.

The tree $T_3$ contains at least $n/900k'$ vertex disjoint bare paths of length $k'$, so contains at least $\lfloor k'/(50\log n+3)\rfloor \cdot (n/900k')$ vertex disjoint bare paths of length $50\log n+1$. As $k'\geq k\geq 100\log n$, for sufficiently large $n$ this is at least $n/10^5\log n$ such paths, and we can use Property~\ref{R4}. That is, there is a copy $S_3$ of $T_3$ in $G[V_2\cup W_2\cup\{v_2\}]$ so that $t_2$ is copied to $v_2$, $(V_2\cup W_2)\setminus V(S_3)\subset W_1$, and there is a cycle, $Q$ say, with the vertex set $(V_2\cup W_2)\setminus V(S_3)$ in~$G$, where $Q$ then has $|(V_2\cup W_2)\setminus V(S_3)|=rl$ vertices (see Figure~\ref{spanpic7D4}).

\textbf{Stage 4.} To recap our situation (as shown in Figure~\ref{spanpic7D4}), we have embedded~$T_2$ except for some missing long and short paths, and also embedded~$T_3$. We have arranged the connectors in~$\HH_0$, ready to embed the missing short paths from~$T_2$. We have formed the remaining spare vertices (those vertices outside the structure we have found and outside the sets $R$, $R''$ and $Y_1$) exactly into a cycle~$Q$. To complete the embedding of~$T$ we will embed the missing long paths from~$T_2$ using sections of~$Q$, which we will attach into~$S_2''$ using some $(l,\gamma)$-connectors from $\HH_0$ and some new paths with vertices in~$R$. The (small proportion of) $(l,\gamma)$-connectors used in this way cannot then be used to embed their matching short paths from~$T_2$. We will instead embed these matching short paths using all the paths $P_{i,j}$, $i\in [r]$, $j\in [2]$, and some new paths with vertices in~$R$. By embedding the rest of the missing short paths from~$T_2$ using their matching $(l,\gamma)$-connectors, we will then complete the copy of~$T_2$. This will give a copy of $T_2\cup T_3$ that contains every vertex in the graph~$G$ except for the vertices in $Y_1\setminus \{v_1\}$ and some vertices in $R\cup R''$. Due to Property~\ref{R6}, we will be able to find a copy of~$T_1$ attached appropriately to~$v_1$ which covers exactly these remaining vertices and~$v_1$. This will complete the copy of~$T$.

Therefore, first divide $Q$ into paths $Q_i$, $i\in[r]$, with length $l-1$, and suppose each path $Q_i$ has endvertices $w_i$ and $z_i$. As $\{w_i,z_i:i\in[r]\}\subset W_1$ and $2r\leq m_\l$, by Property~\ref{R3} and Theorem~\ref{matchingtheorem} there is a matching from $\{w_i,z_i:i\in[r]\}$ into $\HH_0^-$ in the graph $L$. For each $i\in[r]$, let $H_{i,1}, H_{i,2}\in \HH_0$ be such that $H_{i,1}^-$ and $H_{i,2}^-$ are matched to $w_i$ and $z_i$ respectively under this matching (see Figure~\ref{spanpic7D5}).

Let $\HH_1=\{H_{i,1},H_{i,2}:i\in[r]\}$. Consider the following pairs of vertices from the set $R_0$, taken over $i\in [r]$.
\[
(x_i,b_{H_{i,1}}),  (b_{H_{i,2}},y_i), (a_{H_{i,1}},u_{i,1}),
(v_{i,1},c_{H_{i,1}}), (a_{H_{i,2}},u_{i,2}), (v_{i,2},c_{H_{i,2}})
\]
In total, we have $6r$ pairs of vertices, all in $R_0$, so, by Property~\ref{R1}, there are vertex disjoint paths in $G[R]$, of length $6\log n$, connecting these vertex pairs, with internal vertices in $R\setminus R_0$. Thus, such paths avoid the vertices in $\{a_{H},c_H:H\in\HH_0\setminus\HH_1\}$.

For each $H\in\HH_0\setminus\HH_1 $, we can find the required $a_H,c_H$-path with length $l+1$ by taking an appropriate path through $H$, as described previously.
For each $i\in[r]$ and $j\in[2]$, take the $a_{H_{i,j}},u_{i,j}$-path and the $v_{i,j},c_{H_{i,j}}$-path in $G[R]$, and the $u_{i,j},v_{i,j}$-path $P_{i,j}$ and combine them to get an $a_{H_{i,j}},c_{H_{i,j}}$-path with length $l+1$ (see Figure~\ref{spanpic7D6}). Along with each $a_H,c_H$-path, $H\in\HH_0\setminus\HH_1$, use these to add the required $|\HH_0|$ bare paths with length $l+1$ to $S_2''$. To complete the copy of~$T_2$ we need to find the remaining~$r$ missing bare paths with length $3l+12\log n+1$ joining the vertex pairs $(x_i,y_i)$, $i\in [r]$, and also add these to~$S_2''$.

For each $i\in[r]$, as $w_iH^-_{i,1}$ is an edge in $L$, the graph formed in developing the Stage 3 properties, there is some vertex $w_i'\in H^-_{i,1}$ so that $w_iw_i'$ is an edge in~$G$. Similarly, we can find $z_i'\in H_{i,2}^-$, $b_{i,1}\in H_{i,1}^+$ and $b_{i,2}\in H_{i,2}^+$ so that $z_iz_i'$, $b_{i,1}b_{H_{i,1}}$ and $b_{i,2}b_{H_{i,2}}$ are edges in~$G$. Using the definition of an $(l,\gamma)$-connector, find a $b_{i,1},w_i'$-path covering $H_{i,1}$, and a $b_{i,2},z_i'$-path covering $H_{i,2}$. Combining these paths with the path $Q_i$, the $x_i,b_{H_{i,1}}$-path and the $b_{H_{i,2}},y_i$-path from $G[R]$, and the edges mentioned, gives a path with length $3l+12\log n+1$ between $x_i$ and $y_i$. This allows us to add the required $r$ bare paths with length $3l+12\log n+1$ to $S_2''$, and thus complete the copy,~$S_2$ say, of $T_2$.

Finally, let $U$ consist of all the vertices in the $6r$ paths in $G[R]$ that we have just found (including their endvertices) together with the vertices $a_H$ and $c_H$ for each $H\notin \HH_1$. Then $U\subset R$ and
$|U|=36r\log n+6r+2(|\HH_0|-2r)=36r\log n+2|\HH_0|+2r$. Therefore, by Property~\ref{R6}, there is a copy, $S_1$ say, of $T_1$ with the vertex set $Y_1 \cup (R\setminus U)\cup R''$ in which $t_1$ is copied to $v_1$ (see Figure~\ref{spanpic7D7}). Taken with~$S_2$ and $S_3$, this completes the copy of ~$T$.
\oof



\section*{Acknowledgements}
The author would like to thank Andrew Thomason for his guidance and encouragement, and Oliver Riordan for making many suggestions to improve the presentation and accuracy of this paper.

\begin{figure}[h]
\centering
    \begin{subfigure}[b]{0.9\textwidth}
        \centering
        \resizebox{\linewidth}{!}{
            \input{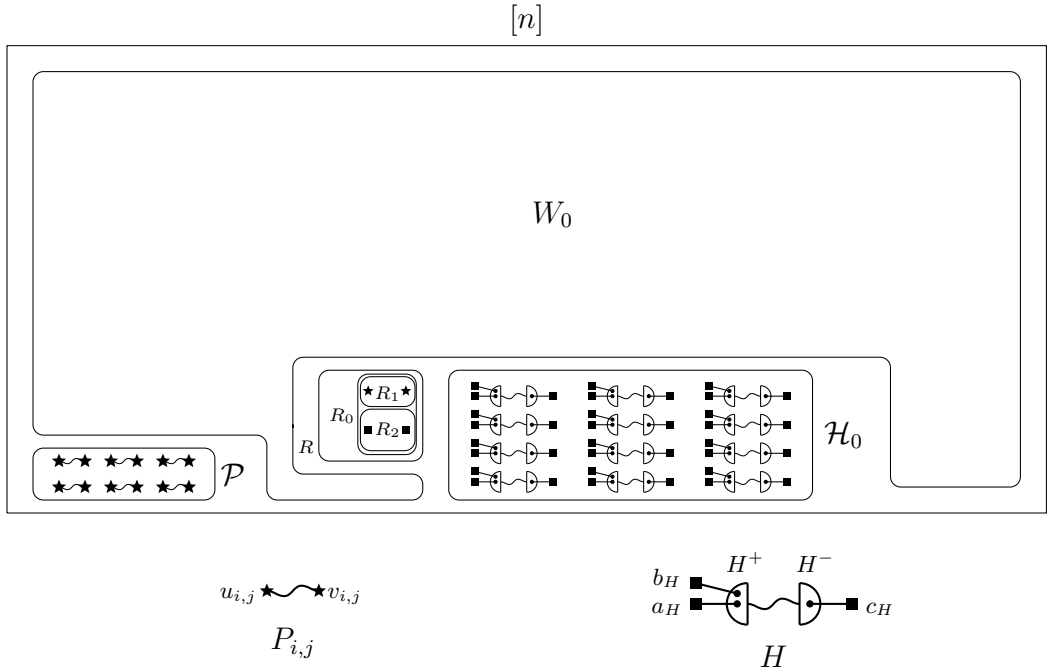}
        }
\label{spanpic7A00}
    \end{subfigure}

\vspace{-0.5cm}

\caption{The sets and subgraphs found at the start of the development of the Stage 3 properties in the proof of Theorem~\ref{unithres} in Cases~C and~D. The paths in $\PP$ are the paths $P_{i,j}$, $i\in[r]$ and $j\in[2]$, with a typical such path copied under $[n]$ and labelled. A typical $(l,\gamma)$-connector $H$ from $\HH_0$ is also copied below $[n]$ and labelled, along with the vertices $a_H$, $b_H$ and $c_H$ which each have a neighbour in~$H^+$ or~$H^-$ as depicted. The sets $W_0$,~$R$, $V(H)$, with $H\in \HH_0$, and $V(P_{i,j})\setminus \{u_{i,j},v_{i,j}\}$, with $i\in [r]$ and $j\in [2]$, form a partition of $[n]$. The vertices $u_{i,j}$, $v_{i,j}$,~$a_H$,~$b_H$, and~$c_H$, with $i\in[r]$, $j\in[2]$ and $H\in\HH_0$, are drawn along with their respective path~$P_{i,j}$ or connector~$H$, but are contained within either the set $R_1$ or the set $R_2$, as indicated by their shape of a star or a square, respectively.}\label{spanpic7A0}
\end{figure}

\newpage

\begin{figure}[p!]
\centering
    \begin{subfigure}[b]{0.9\textwidth}
        \centering
        \resizebox{\linewidth}{!}{
            \input{spanpic7A1}
        }
\label{spanpic7A1}
    \end{subfigure}

\vspace{-0.5cm}

\caption{The sets and subgraphs found in the random graph $G$ in the proof of Theorem~\ref{unithres} in Cases~C and~D. The paths in $\PP$ are the paths $P_{i,j}$, $i\in[r]$ and $j\in[2]$, with a typical such path copied under $V(G)$ and labelled. A typical $(l,\gamma)$-connector $H$ from $\HH_0$ is also copied below $V(G)$ and labelled, along with the vertices $a_H$, $b_H$ and $c_H$ which each have a neighbour in~$H^+$ or~$H^-$ as depicted. The sets $R$, $R'$, $Z_1$, $Z_2$, $Z_3$, $Z_4$, $W_2$, $V(H)$, with $H\in \HH_0$, and $V(P_{i,j})\setminus \{u_{i,j},v_{i,j}\}$, with $i\in [r]$ and $j\in [2]$, form a partition of $V(G)$. The vertices $u_{i,j}$, $v_{i,j}$,~$a_H$,~$b_H$, and~$c_H$, with $i\in[r]$, $j\in[2]$ and $H\in\HH_0$, are drawn along with their respective path $P_{i,j}$ or connector~$H$, but are contained within either the set $R_1$ or the set $R_2$, as indicated by their shape of a star or a square, respectively.}\label{spanpic7A}

\vspace{1.5cm}

\centering
    \begin{subfigure}[b]{0.8\textwidth}
        \centering
        \resizebox{\linewidth}{!}{
            \input{spanpic7B1}
        }
\label{spanpic6B1}
    \end{subfigure}

\vspace{-1.75cm}

\caption{The structure found in the tree $T$ before the start of the embedding in the proof of Theorem~\ref{unithres} in Cases~C and D. The long vertex disjoint bare paths depicted in $T_2$ will be particularly important in our embedding.} \label{spanpic7B}
\end{figure}

\newpage
\begin{figure}[p]
\centering
    \begin{subfigure}[b]{0.9\textwidth}
        \centering
        \resizebox{\linewidth}{!}{
            \input{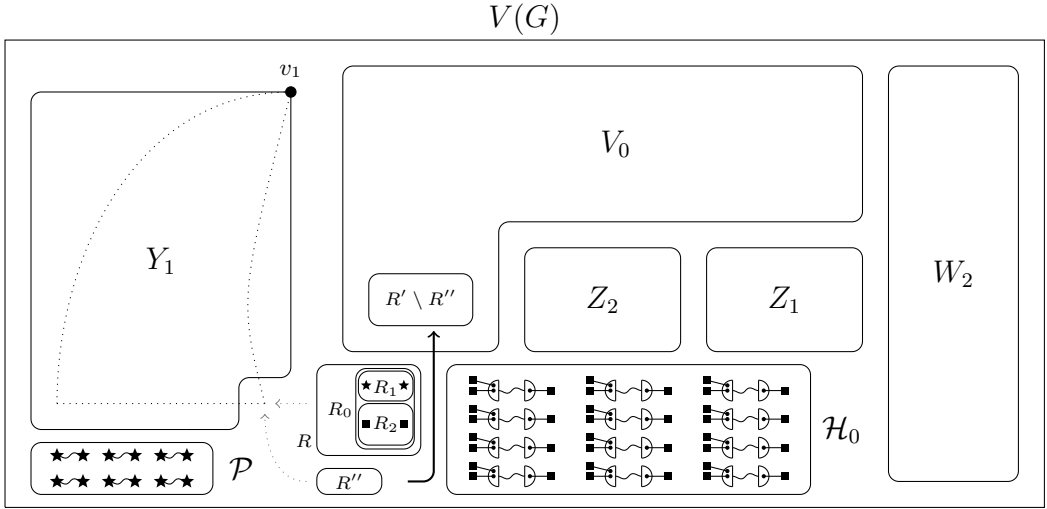}
        }
\label{spanpic7C1}
    \end{subfigure}

\vspace{-0.35cm}

\caption{Using the structure depicted in Figure~\ref{spanpic7A}, and the properties associated with them, we find a set $Y_1\subset Z_3\cup Z_4$ and a vertex $v_1\in Y_1$ such that $R\cup R'$ is made into a reservoir by $(G[Y_1\cup R\cup R'],v_1,T_1,t_1)$. That is, given the appropriate number of vertices from $R\cup R'$, we can cover exactly those vertices and those in $Y_1$ by a copy of $T_1$ in which $t_1$ is copied to $v_1$ (this property is represented by the dotted grey tree and arrows). This completes Stage 1. We then collect the vertices in $Z_3\cup Z_4$ not used in $Y_1$, along with most of the vertices in $R'$ (those not in the newly created subset $R''$), and form the set $V_0$.}
\label{spanpic7D1}
\end{figure}

\begin{figure}[p]
\centering
    \begin{subfigure}[b]{0.9\textwidth}
        \centering
        \resizebox{\linewidth}{!}{
            \input{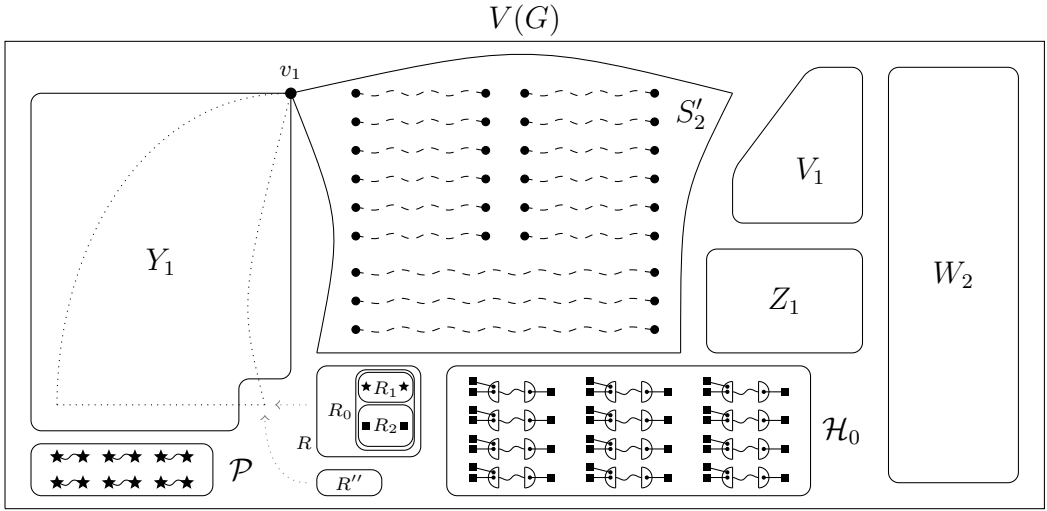}
        }
\label{spanpic7C2}
    \end{subfigure}

\vspace{-0.35cm}

\caption{Having removed many long bare paths from $T_2$ (those depicted in Figure~\ref{spanpic7B}) we use~$V_0\cup Z_2\cup \{v_1\}$ to find a copy $S_2'$ of the resulting subgraph $T_2'$ of $T_2$. The missing paths are depicted by dashed lines. In the following figures several different structures will be used to replace the missing paths. As there is not room in the figures to label all the vertices that will be used, we rely instead on the labelling in Figure~\ref{spanpic7F}, as well as the labelling in Figure~\ref{spanpic7A}.}
\label{spanpic7D2}
\end{figure}

\newpage
\begin{figure}[p]
\centering
    \begin{subfigure}[b]{0.9\textwidth}
        \centering
        \resizebox{\linewidth}{!}{
            \input{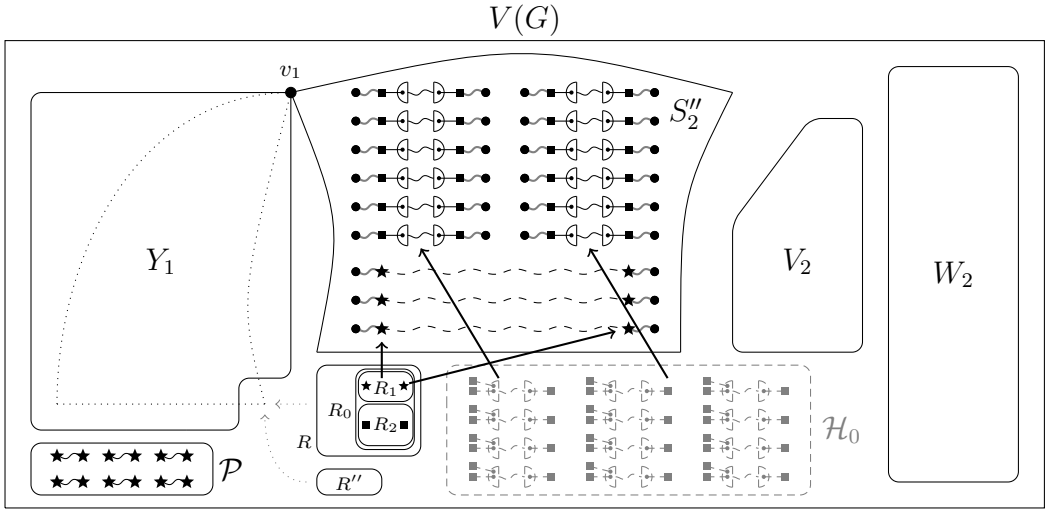}
        }
\label{spanpic7C3}
    \end{subfigure}

\vspace{-0.35cm}

\caption{Using the vertices in $V_1$ and $Z_1$, we find the paths in grey connecting the vertices associated with the $(l,\gamma)$-connectors into the subgraph $S_2'$. We also find the paths in grey connecting some vertices from~$R_1$ into this subgraph, forming in total the subgraph $S_2''$. This completes Stage 2. The vertices in~$V_1$ and $Z_1$ not used to find these paths are formed into the set $V_2$. The $(l,\gamma)$-connectors are placed ready to complete the embedding of the shorter paths removed from $T_2$, but some of these $(l,\gamma)$-connectors will be moved again and used to embed the longer paths instead.}
\label{spanpic7D3}
\end{figure}

\begin{figure}[p]
\centering
    \begin{subfigure}[b]{0.9\textwidth}
        \centering
        \resizebox{\linewidth}{!}{
            \input{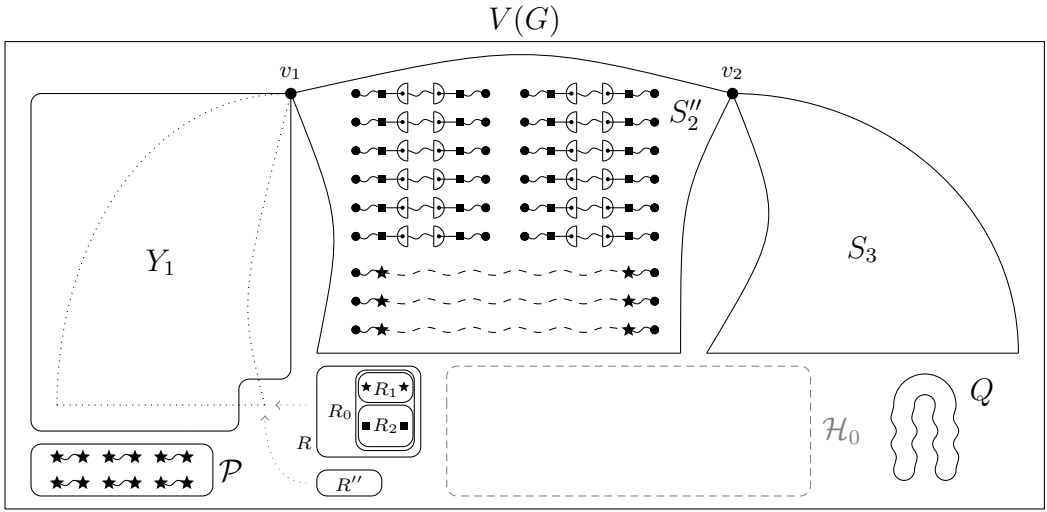}
        }
\label{spanpic7C4}
    \end{subfigure}

\vspace{-0.35cm}

\caption{Using the vertices in $V_2$ and $W_2$, we find a copy $S_3$ of the tree $T_3$, appropriately attached to~$v_2$, so that the remaining vertices in $V_2\cup W_2$ can be made into the cycle $Q$. Stage 3 is now completed.}
\label{spanpic7D4}
\end{figure}

\newpage
\begin{figure}[p]
\centering
    \begin{subfigure}[b]{0.9\textwidth}
        \centering
        \resizebox{\linewidth}{!}{
            \input{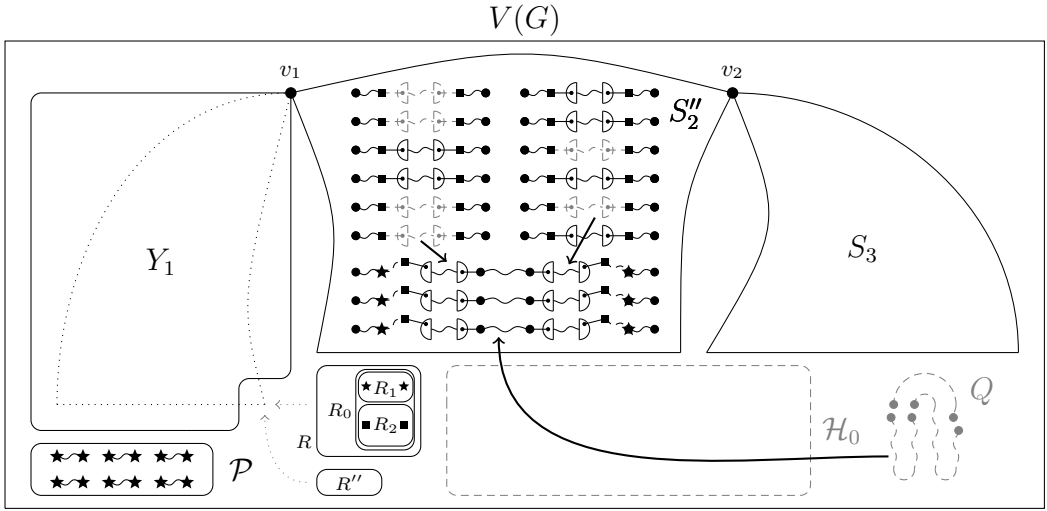}
        }
\label{spanpic7C5}
    \end{subfigure}

\vspace{-0.35cm}

\caption{Breaking the cycle $Q$ into pieces, we find a matching from the endvertices of the subpaths of~$Q$ into the $(l,\gamma)$-connectors, and arrange these paths and matched $(l,\gamma)$-connectors where they will form the majority of the embedding of the longer paths that were deleted from~$T_2$.}
\label{spanpic7D5}
\end{figure}

\begin{figure}[p]
\centering
    \begin{subfigure}[b]{0.9\textwidth}
        \centering
        \resizebox{\linewidth}{!}{
            \input{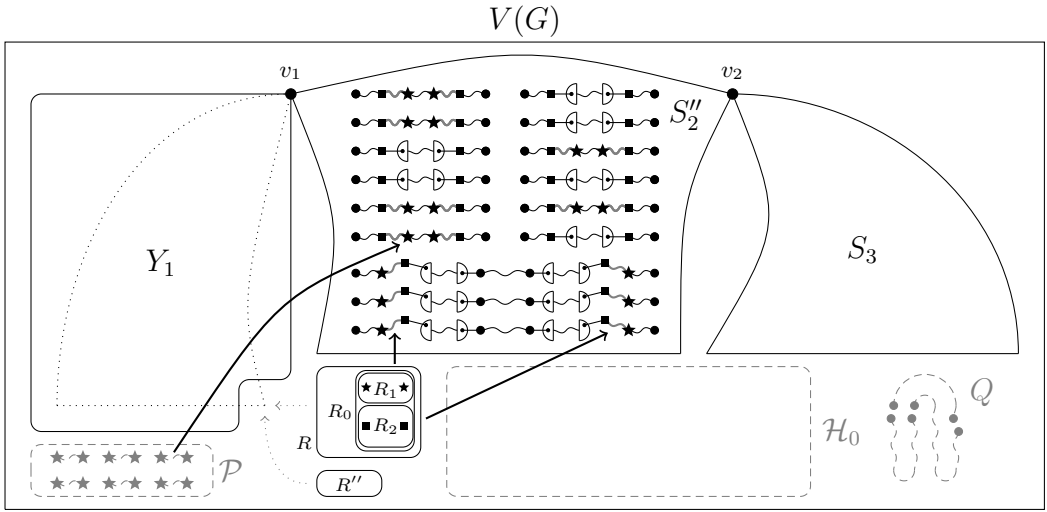}
        }
\label{spanpic7C6}
    \end{subfigure}

\vspace{-0.35cm}

\caption{We now move the paths $P_{i,j}$, $i\in [r]$ and $j\in [2]$, to occupy the space where the matched $(l,\gamma)$-connectors were, and use vertices from $R\setminus R_0$ to construct the paths depicted in grey. Note that the endvertices of these grey paths are shaped as a star, or a square, which, as noted in Figure~\ref{spanpic7A}, represents that these vertices are in $R_1$, or $R_2$, respectively. Where the grey paths join a path $P_{i,j}$, $i\in [r]$, $j\in [2]$, into the graph $S_2''$ this allows us to embed the matching deleted path from $T_2$. We will complete the embedding of all the other deleted paths using the properties of the $(l,\gamma)$-connectors, as depicted in Figure~\ref{spanpic7D7}.}
\label{spanpic7D6}
\end{figure}

\newpage

\begin{figure}[p]
\centering
    \begin{subfigure}[b]{0.9\textwidth}
        \centering
        \resizebox{\linewidth}{!}{
            \input{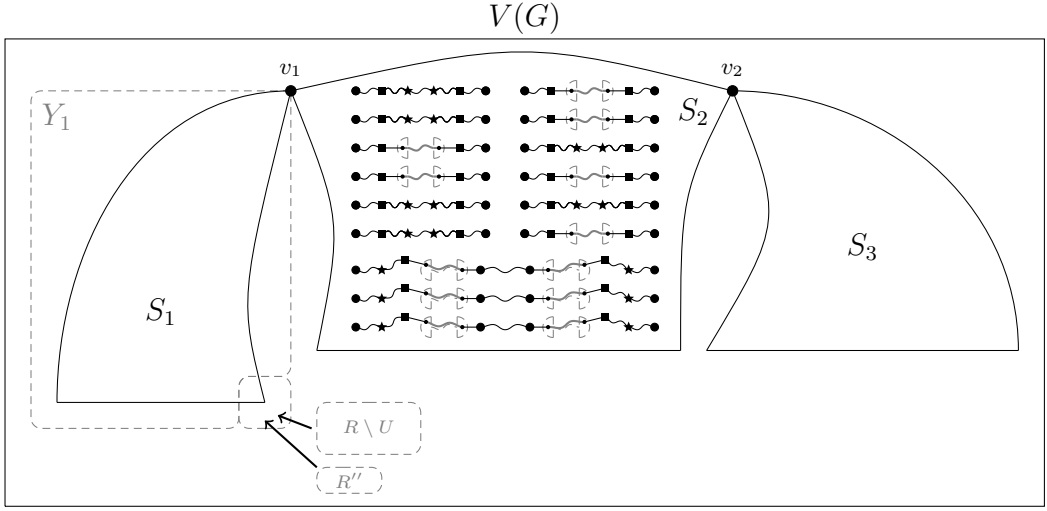}
        }
\label{spanpic7C7}
    \end{subfigure}

\vspace{-0.35cm}

\caption{By the properties of the $(l,\gamma)$-connectors, we can find the paths in grey through the $(l,\gamma)$-connectors to finish embedding the remaining paths to complete $S_2$, a copy of $T_2$. Letting $U$ be the set of vertices in $R$ used to form the paths in Figure~\ref{spanpic7D6}, we take the vertices in $R\setminus U$,~$R''$ and~$Y_1$ and use them to find $S_1$, a copy of $T_1$ attached appropriately to $v_1$. This completes the embedding of $T$.}
\label{spanpic7D7}
\end{figure}

\begin{figure}[p]
\centering
    \begin{subfigure}[b]{0.7\textwidth}
        \centering
        \resizebox{\linewidth}{!}{
            \input{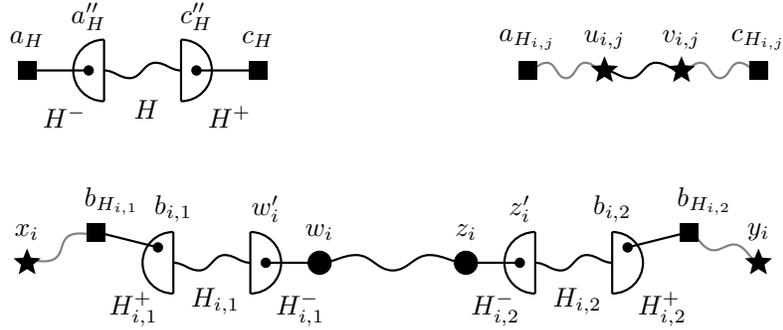}
        }
\label{spanpic7E}
    \end{subfigure}

\vspace{-0.5cm}

\caption{Towards the end of Stage~2 in the proof of Theorem~\ref{unithres} in Cases~C and~D, to complete the embedding of $T_2$ we need to find vertex disjoint $a_H,c_H$-paths, $H\in\HH_0$, with length $l+1$, and vertex disjoint $x_i,y_i$-paths, $i\in [r]$, with length $3l+12\log n+1$. We use three different constructions to find these paths, as depicted in Figures~\ref{spanpic7D1} to~\ref{spanpic7D7} and labelled in this figure. Here, $i\in [r]$, $j\in[2]$ and $H\in \HH_0\setminus \HH_1$. By the properties of the $(l,\gamma)$-connectors we can find spanning paths through the connectors from $a_H''$ to $c_H''$, $b_{i,1}$ to $w_i'$, and $z_i'$ to $b_{i,2}$ as appropriate, to form an $a_H,c_H$-path, an~$a_{H_{i,j}},c_{H_{i,j}}$-path and an $x_i,y_i$-path.}
\label{spanpic7F}
\end{figure}

\clearpage




\end{document}